\documentclass[11pt,reqno]{amsart}
\usepackage{color, amsmath,amssymb, amsfonts, amstext,amsthm, latexsym}
	
\usepackage{amssymb}

\newcommand{\RR}{{\mathbb R}}

\newcommand{\EE}{{\mathbb E}}
\newcommand{\ee}{{\tilde{e}}}

\newcommand{\LL}{{\mathbb L}}
  
\newcommand{\PP}{{\mathbb P}}

\renewcommand{\AA}{{\tilde{A}}}
\newcommand{\QQ}{{\tilde{Q}}}
\newcommand{\WW}{{\tilde{W}}}
\newcommand{\GG}{{\tilde{G}}}
\newcommand{\KK}{{\tilde{K}}}
\renewcommand{\SS}{{\tilde{S}}}

\numberwithin{equation}{section}
\newtheorem{theorem}{Theorem}[section]

\newtheorem{lemma}[theorem]{Lemma}

\newtheorem{prop}[theorem]{Proposition}
\newtheorem{coro}[theorem]{Corollary}

\begin{document}
\title[Strong Rate of Convergence of a  time Euler scheme for 2D Boussinesq]
{Rate of convergence of a semi-implicit time Euler \\scheme 
  for a 2D B\'enard-Boussinesq  model} 

\author[H. Bessaih]{Hakima Bessaih}
\address{Florida International University, Mathematics and Statistics Department, 11200 SW 8th Street, Miami, FL 33199, United States}
\email{hbessaih@fiu.edu}

\author[A. Millet]{ Annie Millet}
\address{SAMM, EA 4543,
Universit\'e Paris 1 Panth\'eon Sorbonne, 90 Rue de
Tolbiac, 75634 Paris Cedex France {\it and} Laboratoire de
Probabilit\'es, Statistique et Mod\'elisation, UMR 8001, 
  Universit\'es Paris~6-Paris~7} 
\email{annie.millet@univ-paris1.fr}



\begin{abstract}
 We prove that a semi-implicit time Euler scheme for the two-dimensional B\'enard-Boussinesq  model on the torus $D$ converges.
 The rate of convergence in probability is almost 1/2 for a multiplicative noise; this  relies on moment estimates in various norms for the processes and the scheme.
 In case of an additive noise, due to the coupling of the equations, provided that the difference on temperature between the top and bottom parts of the torus 
 is not too big compared to the viscosity and
 thermal diffusivity,  a strong polynomial rate of convergence (almost 1/2) is proven in  $(L^2(D))^2$ for the velocity  and    in $L^2(D)$ for the temperature.
  It depends on exponential moments of the scheme; due to linear terms involving the other quantity in both evolution equations, the proof has to be done simultaneaously for 
  both the velocity and the temperature. These rates in both cases are similar to that obtained for the Navier-Stokes equation. 
\end{abstract}

\maketitle


\section{Introduction}\label{s1} \smallskip

The Boussinesq equations have been used as a model in many geophysical applications. They have been widely studied in a both
 the deterministic and stochastic settings. 
 We take random forcings into account and formulate the B\'enard convection problem as a system of stochastic partial differential equations (SPDEs). 
 The need to take stochastic effects into account for modeling complex systems has now become widely recognized. Stochastic 
partial differential equations (SPDEs) arise naturally as mathematical models for nonlinear macroscopic dynamics under random influences.
 It is  a coupled system of the stochastic Navier–Stokes equations and the stochastic transport equation for temperature.
 For us this system will be subject to
an additive random perturbation which will be defined later. 
Here, $u$ describes the fluid velocity field, while ${\mathcal T}$ describes the temperature  of the buoyancy  driven  fluid, and $\pi$ is the  fluid's  pressure.

We study the  B\'enard-Boussinesq  equations on the torus $D=[0,L]^2$ subject to a multiplicative (an additive) stochastic perturbation. 
Unlike \cite{BeMi_Bou} we impose different values of the temperature ${\mathcal T}$ on the top and bottom parts of the torus,  that is ${\mathcal T}(t,.)={\mathcal T}_0$ on $\{x_2=0\}$
and ${\mathcal T}(t,.)={\mathcal T}_L$ on $\{x_2=L\}$ for $t\in [0,T]$. 
 The fluid's velocity $u$, temperature ${\mathcal T}$ and pressure $\pi$   satisfy the following B\'enard problem  (see e.g. \cite{Temam-88}  and \cite{Fer}):
\begin{align} \label{velocity-mul}
 \partial_t u - \nu \Delta u + (u\cdot \nabla) u + \nabla{\pi} & = ({\mathcal T}-T_L) v_2 + G(u) \,  dW\quad \mbox{\rm in } \quad (0,T)\times D,\\
 \partial_t {\mathcal T} - \kappa \Delta {\mathcal T}  + (u\cdot \nabla {\mathcal T}) &= {\mathcal G} ({\mathcal T})\,  d\tilde{W} \quad \mbox{\rm in } \quad (0,T)\times D, 
 \label{temperature-mul}\\
 \mbox{\rm div }u&=0 \quad \mbox{\rm in } \quad  [0,T]\times D,  \nonumber 
 \end{align}
 where  $T>0$. The processes  
$u: \Omega\times (0,T)\times D  \to \RR^2$, and ${\mathcal T} : \Omega\times (0,T)\times D  \to \RR$ 
 have  initial conditions $u_0$ and $ {\mathcal T}_0$ in $D$ respectively, and $\pi : \Omega\times (0,T)\times D  \to \RR$. 
The  parameter $\nu>0$ denotes  the  kinematic viscosity of the fluid,  $\kappa>0$ its thermal diffusivity, and $(v_1,v_2)$ is the canonical basis of $\RR^2$. 
On $[0,T]\times \partial D$ these fields satisfy 
  the following  boundary conditions  
  \begin{align} \label{periodicity}
  u(t,x+L v_i)=u(t,x), & \; i=1,2,  \quad {\mathcal T} (t,x+L v_1) = {\mathcal T}(t,x), \quad  \partial_1 {\mathcal T}(t,x+Lv_1) = {\mathcal T}(t,x), \nonumber  \\
  \pi(t,x+L v_1) =\pi(t,x),& \quad   \partial_1 \pi (t,x+Lv_1) = \pi (t,x) 
  \end{align}
 where $v_i$, $i=1,2$
 denotes the canonical basis of $\RR^2$, 
 
 More information on the stochastic perturbation will be given in the next section. 
  We then introduce a classical change of processes to describe  the temperature and pressure (see e.g. \cite{Temam-88}  and \cite{Fer}):
 \begin{equation} \label{def_theta}
  \theta(t,x)={\mathcal T}(t,x) - {\mathcal T}_0 - x_2 \frac{{\mathcal T}_L-{\mathcal T}_0}{L}, \quad 
 \hat{\pi}(t,x) = \pi(t,x) + ({\mathcal T}_L-{\mathcal T}_0) x_2 + \frac{ ({\mathcal T}_L-{\mathcal T}_0) x_2^2}{2L} .
 \end{equation}
We deduce that $\theta$ and $\hat{\pi}$ have the same periodic properties as ${\mathcal T}$ and $\pi$,  that $\theta(t,.)=0$ on $\{x_2 =0\}\cup \{x_2=L\}$
and that the pair $(u,\theta)$ satisfies the more symmetric  Boussinesq equations
\begin{align} \label{velocity}
 \partial_t u - \nu \Delta u + (u\cdot \nabla) u + \nabla \hat{\pi} & = \theta   v_2 +  G(u)  dW\quad \mbox{\rm in } \quad (0,T)\times D,\\
 \partial_t \theta- \kappa \Delta \theta + (u\cdot \nabla \theta)  +  C_L u_2&=  \tilde{G}(\theta)  d\tilde{W} \quad \mbox{\rm in } \quad (0,T) \times D, 
 \label{temperature}\\
 \mbox{\rm div }u&=0 \quad \mbox{\rm in } \quad [0,T]\times D,  \nonumber 
 \end{align}
where  $C_L:= \frac{{\mathcal T}_L-{\mathcal T}_0}{L}\in \RR$ and $\GG$ is defined by 
$\GG(\theta) = {\mathcal G}\big({\mathcal T}-{\mathcal T}_0 - x_2 \frac{{\mathcal T}_L-{\mathcal T}_0}{L}\big)$ .
 
 In two dimensions, the deterministic version of this system (without stochastic perturbations) has been extensively studied, yielding a comprehensive understanding 
 of its well-posedness and long-time behavior; see \cite{Temam-88}. Further investigations have addressed the cases where either 
$\nu=0$ or  $\kappa=0$, albeit with only partial results.

In \cite{HouLi}, the authors establish the existence and uniqueness of solutions when $\nu>0$ and  $\kappa=0$ for initial conditions $u_0\in (H^m)^2(\RR^2)$ 
and $\theta_0\in H^{m-1}(\RR^2)$ for $m\geq 3$. In \cite{LaLuTi}, the case where viscosity acts only in the horizontal direction (i.e., 
 replacing $\nu \Delta u$ by $\nu \partial_1^2 u$) and $\kappa =0$ is analyzed. The authors prove the existence of a global solution for initial data 
 $\theta_0\in L^2(D)$ and a divergence-free velocity field $u_0\in \big( W^{1,2}(D)\big)^2$.  They reformulate the system in terms of vorticity and establish uniqueness 
 under stronger regularity assumptions, leveraging techniques inspired by Yudovich. To the best of our knowledge, the case $\nu=\kappa=0$ remains an open problem.

In the stochastic setting, well-posedness and the large deviation principle have been established for the B\'enard system with $\nu>0$ and $\kappa>0$; 
see \cite{ChuMil}. Random attractors are investigated in \cite{DGS} for a more general ocean-atmospheric model, while invariant measures are studied in \cite{Fer} using dissipativity arguments.

Let $V^0 = \{ u \in (L^2)^2 : {\rm div}\,  u=0\}$, $\Pi : (L^2)^2 \to V^0$ denote the Leray projection on divergence free fields,  and $A=-\Pi \Delta$ denote he Stokes operator. 
Projecting on divergence-free fields, we consider the system 
\begin{align} \label{velocity-u}
 \partial_t u + \nu A u + (u\cdot \nabla) u  & = \Pi (\theta v_2) +   G(u) dW\quad \mbox{\rm in } \quad (0,T)\times D,\\
 \partial_t \theta- \kappa \Delta \theta + (u\cdot \nabla \theta)  +  C_L u_2&=  \GG(\theta)  d\tilde{W} \quad \mbox{\rm in } \quad (0,T) \times D, 
 \label{temperature-theta} 
 \end{align}
 If $u_0$ and $\theta_0$ are square integrable, it is known that the random system \eqref{velocity}--\eqref{temperature} is well-posed, 
 and that there exists a unique solution $(u\times \theta)$ in
 $L^2([0,T];(L^2)^2 \times L^2) \cap L^2(\Omega ; (H^1)^2 \times H^1)$; see e.g. \cite{ChuMil}. 

Numerical schemes and algorithms have been introduced to best approximate the solution to non-linear PDEs. The time approximation
is either an implicit Euler or a time splitting scheme coupled with Galerkin approximation or finite elements to approximate the space variable. 
The literature on numerical analysis for SPDEs is now very extensive. 
In many papers  the models are either linear, have global Lipschitz properties,
 or more generally some monotonicity property. In this case the convergence is proven to be in mean square. 
When nonlinearities are involved that are not of Lipschitz or monotone type,  then a rate of  convergence 
 in mean square is more difficult to  obtain.  
  Indeed, because of the stochastic perturbation, one may not use the Gronwall lemma after taking the expectation of the
   error bound   
  since it involves a nonlinear term which is often  quadratic; such a nonlinearity requires some localization.

In a random setting, the discretization of the Navier-Stokes equations has been intensively investigated.  
Various space-time numerical schemes have been studied for the stochastic Navier-Stokes equations with a multiplicative or an additive noise, that is 
where in the
right hand side of  \eqref{velocity} (with no $\theta$) we have either $G(u) \, dW$  or $dW$.
 We refer to \cite{BrCaPr,Dor, Breckner, CarPro, BreDog}, where convergence in probability is stated with various rates of convergence
  in a multiplicative setting. 
 As stated previously, the main tool to get the convergence in probability  is the localization of the nonlinear term over a space of large probability.
  We studied the strong (that is  $L^2(\Omega)$) rate of convergence of the time  implicit Euler scheme
   (resp. space-time implicit Euler scheme coupled with
  finite element space discretization) in our previous papers \cite{BeMi_multi}  (resp.  \cite{BeMi_FEM}) for an $H^1$-valued initial condition. 
  The method is based on the fact that the solution (and the scheme) have finite  moments (bounded uniformly on the time mesh). 
  For a general multiplicative noise, the strong rate is logarithmic. When the diffusion coefficient is bounded (which is a slight extension of an additive noise),
  the solution has exponential  moments;  we used this property in \cite{BeMi_multi} and \cite{BeMi_FEM} to  get an explicit polynomial strong rate
   of convergence. However, this rate depends on the viscosity and the strength of the noise, and is strictly less than 1/2 for the time parameter 
   (resp. than $1$ for the
   spatial one). For a given viscosity, the rates on convergence increase to 1/2 when the strength of the noise converges to 0. For an additive noise,
  if the strength of the noise is not too large,  the strong ($L^2(\Omega)$) rate of convergence in time is the optimal one, that is almost $1/2$,  for an $H^1$-valued initial condition 
  (see  \cite{BeMi_additive}).  The case of an additive noise (or a bounded linear coefficient) depends on exponentlal moments of the supremum of the
  $H^1$-norm of the solution (and of its Euler scheme for the space discretization); this enables to have polynomial rates. 
\smallskip

The convergence of a fully implicit time Euler scheme for the two-dimensional Boussinesq equation on the torus has been studied in \cite{BeMi_Bou} for 
periodic solutions $u,\theta$ and  a multiplicative noise on the torus. 
The rate of convergence in probability (almost 1/2) is the optimal one for $H^1$-valued initial velocity and temperature, while the strong rate is logarithmic. 
The proof heavily depends on  moments of the supremum of the $H^1$-norms for the velocity, which is quite easy to prove, and of the temperature, which is more delicate and
requires larger moments on the initial condition. The argument is based on moments of small time increments and on the strong convergence on a subset of $\Omega$ on which
the moments of the $H^1$-norms of both quantities are bounded.

In the current paper, we study the time approximation of the Boussinesq equations \eqref{velocity}-\eqref{temperature}   subject to a multiplicative and  additive perturbation 
but for the B\'enard model, 
  Unlike \cite{BeMi_Bou} we consider  semi-implicit (and not fully implicit)  time Euler schemes $\{ u^j\}_j$ and $\{\theta^j\}_j$.
On one hand, we extend the results of \cite{BeMi_Bou} to this more general setting and prove that the rate of convergence in probability of the scheme is "almost" 1/2.
 On the other hand, we  prove  the  strong (i.e., $L^2(\Omega)$) polynomial rate of convergence of the scheme, which is "almost" 1/2 (that is the optimal one due to the Gaussian noise), 
uniformly on the time grid
for the $L^2$-norms of both processes. In order to achieve this rate, we give an explicit constraint  on the average temperature difference  $C_L$ - when it is different from 0 - 
and the strength of
the perturbations  in the case of deterministic initial conditions. As expected, this constraint is stronger
for "small" values of the viscosity $\nu$ and the thermal diffusivity $\kappa$, and for "large" values of the length of the time interval and of the traces of the covariance operators of the
random additive noises.  
  As  in the case of  the 2D-Navier-Sokes equations,
the main ingredient is the study of exponential moments. Unlike references \cite{BeMi_multi} and \cite{BeMi_additive}, we rely on exponential moments  of the schemes 
$\frac{T}{N} \sum_{j=1}^N \|\nabla \theta^j\|_{H^0}^2$ and  $\frac{T}{N} \sum_{j=1}^N \|\nabla u^j\|_{V^0}^2$, and not on that of the square of the $H^1$ and $V^1$ norms
of the solutions, uniformly in time (the definition of the spaces $H^k$ and $V^k$ is given in section \ref{frame}).
 Note that such exponential moments were already proven for the scheme $u^n$ of the 2D Navier-Stokes equations  in \cite{BeMi_FEM};  
they were used to relate the time and the space time Euler schemes.  
When $C_L\neq 0$, the velocity and temperature are treated simultaneously, and the coefficient in front of the velocity is that in front of the temperature multiplied by $C_L$. 
When $C_L=0$, the argument is different; one can first prove exponential moments for the scheme of the temperature, and then deduce exponential moments for the scheme of the velocity.  
Note that when we only deal with the velocity, that is with the solution to the 2D Navier-Stokes equations,  the results we obtain show the strong
rate of convergence of a semi-implicit time Euler scheme to the solution, thus completing what was proved in \cite{BeMi_additive} for a fully implicit one.

 The paper is organized as follows. In section \ref{preliminary} we describe the model, the assumptions on the noise and recall known  results on the global well-posedness
 of  the solution to  \eqref{velocity}--\eqref{temperature} and on the moments of $H^1$-norm of $u$ 
 and  $ \theta$. In section \ref{scheme},  we prove study the semi-implicit Euler scheme for a  multiplicative stochastic perturbation and prove a strong convergence for  a localized version
 of the scheme. Section \ref{sec_rate_proba} shows that for a multiplicative stochastic perturbation we can deduce a rate of convergence in probability "almost  1/2" from the previous localized 
 convergence. In section \ref{additive_noise} we deal  with an additive noise; the main results are the existence of exponential moments for the schemes when $C_L$ is positive, and  when $C_L=0$
 with a different argument. 
In section \ref{s-converg} we prove the main result about the polynomial  rate of convergence  in $L^2(\Omega)$ of the semi-implicit scheme to the solution. 
Finally, for the sake of completeness, some proofs of section \ref{preliminary} are given in the Appendix. They slightly differ from that in \cite{BeMi_Bou} since the velocity appears in the evolution
of the temperature.

As usual, except if specified otherwise, $C$ denotes a positive constant that may change  throughout the paper,
 and $C(a)$ denotes  a positive constant depending on some parameter $a$.

\section{Description of the model}\label{preliminary} 
In this section, we describe the functional framework, the driving noise, the evolution equations, 
and state known global well-posedness results for square integrable initial conditions and moments in $H^1$ uniformly in time.

\subsection{The functional framework}\label{frame}
Let  $D= [0,L]^2$   with periodic boundary conditions for the velocity $u$ and the pressure $\pi$, and such that 
$\theta(x_1,x_2)=0 $  on $\{x_2=0\} \cup \{x_2=1\}$,    $ \partial_1(\theta)$ is periodic in $x_ 1$  with period  $L$. 

Let 
${\mathbb L}^p:=L^p(D)^2$ (resp. ${\mathbb W}^{k,p}:=W^{k,p}(D)^2$)   be  the usual Lebesgue and Sobolev spaces of 
vector-valued functions
endowed with the norms  $\|\cdot \|_{{\mathbb L}^p}$  (resp. $\|\cdot \|_{{\mathbb W}^{k,p}}$).

  Let $\tilde{A} = -\Delta$ acting on $L^2(D)$ with the boundary conditions  imposed on $\theta$.  Let $\tilde{\lambda}_1>0$ denote the smallest positive eigenvalue of the
operator $(-\Delta)$ acting on the set of functions $\theta$ described above.  Then 
\begin{equation} 	\label{tilde_lambda_1} 
|\theta\|_{H^0}^2 \leq \frac{1}{\tilde{\lambda}_1} \| \tilde{A}^{\frac{1}{2}}\theta \|_{H^0}^2, \quad 
\forall  \theta \in H^1.
\end{equation} 

  For any non negative real number $k$ let
\[  H^k={\rm Dom}\big(\tilde{A}^{\frac{k}{2}}\big) , \;   V^k=\{ \phi \in {\mathbb W}^{k,2} \, : \mbox{\rm div} \phi =0\},\; \mbox{\rm endowed with the norms }
 \|\cdot\|_{H^k} \; \mbox{\rm and } \|\cdot\|_{V^k}.\]
Thus $H^0=L^2(D)$,  $V^0 \subset {\mathbb L}^2(D)$, and $H^k=W^{k,2}$. 
 Moreover,   let  $V^{-1}$ be the dual space of $V^1$ 
   with respect to the  pivot space $V^0$,  and 
  $\langle\cdot,\cdot\rangle$ denotes the duality between $V^1$ and $V^{-1}$.   
  Let  $\Pi : \mathbb{L} ^2 \to V^0$ denote the Leray projection, and let   $A=- \Pi \Delta$ denote the Stokes operator, with  domain  
  $\mbox{\rm Dom}(A)={\mathbb W}^{2,2}\cap V^0 = V^2$.  
  
   Let $\lambda_1>0$ denote the smallest positive eigenvalue of $A$. Then 
 \begin{equation}		\label{lambda_1}
 \| u\|_{V^0}^2 \leq \frac{1}{\lambda_1} \| A^{\frac{1}{2}} u\|_{V^0}^2  \quad \forall u\in V^1. 
 \end{equation}

  Let $b:(V^1)^3 \to \RR$ denote the trilinear map defined by 
  \[ b(u_1,u_2,u_3):=\int_D  \big(\big[ u_1(x)\cdot \nabla\big] u_2(x)\big)\cdot u_3(x)\, dx. \]
  The incompressibility condition implies   $b(u_1,u_2, u_3)=-b(u_1,u_3,u_2)$  for $u_i \in V^1$, $i=1,2,3$. 
  There exists a continuous bilinear map $B:V^1\times V^1 \mapsto
  V^{-1}$ such that
  \[ \langle B(u_1,u_2), u_3\rangle = b(u_1,u_2,u_3), \quad \mbox{\rm for all } \; u_i\in V^1, \; i=1,2,3.\]
  Therefore, the map  $B$ satisfies the following antisymmetry relations:
  \begin{equation} \label{B}
  \langle B(u_1,u_2), u_3\rangle = - \langle B(u_1,u_3), u_2\rangle , \quad \langle B(u_1,u_2),  u_2\rangle = 0 
  \qquad \mbox {\rm for all } \quad u_i\in V.
  \end{equation}
  For $u,v\in V^1$,  we have $B(u,v):= \Pi \big( \big[ u\cdot \nabla\big]  v\big) $.
  
  Furthemore, since $D=[0,L]^2$ with periodic boundary conditions, we have
  \begin{equation}		\label{B(u)Au}
  \langle B(u,u), Au\rangle =0, \quad \forall  u\in V^2.
  \end{equation}
  Note that for $u\in V^1$ and $\theta_1, \theta_2\in H^1$, if $(u.\nabla)\theta = \sum_{i=1,2} u_i \partial_i \theta$, we have
  \begin{equation} 	\label{antisym-ut}
   \langle [u.\nabla]\theta_1\, , \, \theta_2 \rangle = - \langle [u.\nabla]\theta_2\, , \, \theta_1\rangle,
   \end{equation}
  so that $\langle [u.\nabla]\theta\, , \, \theta \rangle =0$ for $u\in V^1$ and $\theta\in H^1$.

  In dimension 2 the inclusions  $H^1\subset L^p$  and $V^1\subset \LL^p$  for $p\in [2,\infty)$ follow from the Sobolev embedding theorem. 
  More precisely 
  the following Gagliardo Nirenberg inequality is true for some constant $\bar{C}_p$
    \begin{equation} \label{GagNir}
  \|u\|_{\LL^p} \leq \bar{C}_p \;   \|A^{\frac{1}{2}}  u\|_{\LL^2}^{\alpha}  \|u\|_{\LL^2}^{1-\alpha}  \quad \mbox{\rm for} \quad
    \alpha = 1-\frac{2}{p}, \quad \forall u\in V^1.
  \end{equation} 
.

Finally, let us recall the following estimate of the bilinear terms $(u.\nabla)v$ and $(u.\nabla)\theta$.
\begin{lemma} \label{GiMi}
Let $\alpha,  \rho$ be positive numbers and $\delta \in [0,1)$ be such that  $\delta +\rho > \frac{1}{2}$ and $\alpha + \delta +\rho \geq 1$. 
 Let  $u\in V^\alpha$,  $v\in V^\rho$ and $\theta \in H^\rho$; then 
\begin{align} 	\label{GiMi-uv}
\| A^{-\delta} \Pi [(u.\nabla) v] \|_{V^0} &\leq C \| A^\alpha u\|_{V^0}\, \|A^\rho v\|_{V^0}, \\
\| \tilde{A}^{-\delta} [(u.\nabla) \theta]\|_{H^0} &\leq C \| A^\alpha u\|_{V^0}\, \|\tilde{A}^\rho \theta\|_{H^0}, \label{GiMi-ut}
\end{align} 
   for some positive constant $C:= C(\alpha, \delta, \rho)$.
\end{lemma}
\begin{proof}
The upper estimate \eqref{GiMi-uv} is Lemma 2.2 in \cite{GigMiy}. The argument, which is based on the Sobolev embeding theorem and
H\"older's inequality, clearly proves \eqref{GiMi-ut}. 
\end{proof}

\subsection{The multiplicative stochastic perturbation}\label{noise} 
Let $K$ (resp. $\tilde K$) be a Hilbert space 
 and let $(W(t), t\geq 0)$ (resp. $(\tilde{W}(t), t\geq 0)\, )$ be a $K$-valued (resp. $\tilde{K}$- valued) 
 Brownian motion with covariance $Q$ (resp. $\tilde{Q}$),
 that is a trace-class operator of $K$ (resp. $\tilde{K}$)  such that $Q \zeta_j = q_j \zeta_j$ (resp. $\tilde{Q} \tilde{\zeta}_j = \tilde{q}_j \tilde{\zeta}_j$), where
  $\{ \zeta_j\}_{j\geq 0}$ (resp. $\{ \tilde{\zeta}_j\}_{j\geq 0}$) is a complete orthonormal system of $K$ (resp. $\tilde{K}$), $q_j, \tilde{q}_j >0$,
  and ${\rm Tr}(Q)=\sum_{j\geq 0} q_j <\infty$
  (resp. ${\rm Tr}(\tilde{Q})=\sum_{j\geq 0} \tilde{q}_j <\infty$). Let $\{ \beta_j \}_{j\geq 0}$ (resp. $\{ \tilde{\beta}_j\}_{j\geq 0}$) be a sequence of independent
  one-dimensional Brownian motions on the same filtered probability space $(\Omega, {\mathcal F}, ({\mathcal F}_t, t\geq 0), \PP)$. Then
  \[ W(t)=\sum_{j\geq 0} \sqrt{q_j}\, \beta_j(t) \,\zeta_j, \quad \tilde{W}(t)=\sum_{j\geq 0} \sqrt{\tilde{q}_j }\, \tilde{\beta}_j \,\tilde{\zeta}_j.\] 
   For details concerning this Wiener process  we refer  to \cite{DaPZab}.

We make the following classical linear growth and Lipschitz assumptions on the diffusion coefficients $G$ and ${\mathcal G}$.

 \noindent{\bf Condition (C-u) (i)} Let $G:V^0\to {\mathcal L}(K;V^0)$ be such that
 \begin{align}  \|G(u)\|^2_{{\mathcal L}(K,V^0)} \leq&\,  K_0 + K_1 \|u\|^2_{V^0}, \quad \forall u\in V^0, 	\label{growthG-0}\\
  \|G(u_1)-G(u_2)\|^2_{{\mathcal L}(K,V^0)} \leq &\,  L_1  \|u_1-u_2|^2_{V^0} , \quad  \forall u_1,u_2\in V^0.	\label{LipG}
 \end{align}
{\bf  (ii)} Let also   $G:V^1\to {\mathcal L}(K;V^1)$ satisfy the growth condition
\begin{equation}		\label{growthG-1}
\| G(u)\|_{{\mathcal L}(K;V^1)}^2 \leq K_2 + K_3 \|u\|_{V^1}^2, \quad \forall u\in V^1.
\end{equation}
\noindent {\bf Condition  ($\tilde{C}$) (i)} Let ${\mathcal G}:H^0\to {\mathcal L}(\KK;H^0)$ be such that
 \begin{align}  \| {\mathcal G}({\mathcal T})\|^2_{{\mathcal L}(\KK,H^0)} \leq&\,  \tilde{K}_0 + \tilde{K}_1 \|{\mathcal T}\|^2_{H^0}, \quad \forall {\mathcal T} \in H^0, 	
 \label{growthGG-0}\\
  \|{\mathcal G}({\mathcal T}_1)-{\mathcal G}({\mathcal T}_2)\|^2_{{\mathcal L}(\KK,H^0)} \leq &\,  \tilde{L}_1  \|{\mathcal T}_1-{\mathcal T}_2 \|^2_{H^0} , \quad \forall {\mathcal T}_1, 
  {\mathcal T}_2\in H^0.	\label{LipGG}
 \end{align}
 {\bf (ii)} Let also ${\mathcal G}:H^1\to {\mathcal L}(K;H^1)$
 satisfy the growth condition
 \begin{equation}		\label{growthGG-1}
\| {\mathcal G}({\mathcal T})\|_{{\mathcal L}(K;H^1)}^2 \leq \KK_2 + \KK_3 \|{\mathcal T}\|_{H^1}^2, \quad \forall {\mathcal T}\in H^1.
\end{equation} 
 Let $\theta$ and $\hat{\pi}$ be defined in terms of ${\mathcal T}$ and $\pi$ by \eqref{def_theta}. Then the coeffficient
$\GG$ defined by $\GG(\theta) = {\mathcal G}\big({\mathcal T}-{\mathcal T}_0 - x_2 \frac{{\mathcal T}_L-{\mathcal T}_0}{L}\big)$ clearly satisfies the same linear growth and Lipschitz
conditions as those stated in {\bf Condition  ($\tilde{C}$)}, replacing ${\mathcal G}$ by $\GG$ and ${\mathcal T}$ by $\theta$ respectively. Therefore, the condition  {\bf Condition  (C-$\theta$)}
 from \cite{BeMi_Bou} is satisfied.
    
 Projecting the velocity on divergence free fields,  we consider the following SPDEs for the processes modeling  the velocity  $u(t)$ 
 and the temperature $\theta(t) $. The initial conditions $u_0$ and $\theta_0$ are ${\mathcal F}_0$-measurable, taking values in  
  $V^0$ and $H^0$ respectively, and the time dynamic is 
 \begin{align}
\partial_t u(t) + \big[ \nu\, A u(t) + B(u(t),u(t)) \big] dt = &\, \Pi (\theta(t) v_2) +  G\big((u(t) \big) dW(t), \label{def_u} \\
\partial_t \theta(t) + \big[\kappa \,\AA \theta(t) + (u(t).\nabla) \theta(t) + C_L u_2(t) \big] dt =&\, \GG\big(\theta(t) \big) d\WW(t), 	\label{def_t}
\end{align}

$\nu,\kappa$ are strictly positive constants, and  $v_2=(0,1)\in \RR^2$.

 \subsection{Global well posedness and moment estimates of $(u,\theta)$}	\label{gwp}
   The first result states the existence and uniqueness of   a   weak pathwise solution (that is strong probabilistic solution in the weak deterministic sense) 
  of \eqref{def_u}-\eqref{def_t}. It has been proven in \cite{ChuMil}.
 \begin{theorem}		\label{th-gwp}
Let $(\Omega, {\mathcal F}, \{{\mathcal F}_t \}_{t\geq 0},\PP)$  be a filtered probability space, and $W$ be a Brownian motion with respect to this filtration as described in Section
\ref{noise}.
 Let $u_0$ and $\theta_0$ be ${\mathcal F}_0$-measurable such that 
 $u_0\in L^{2p}(\Omega;V^0)$, $\theta_0 \in L^{2p}(\Omega;H^0)$ for  $p\in [2,\infty)$. 
 Then equations \eqref{def_u}--\eqref{def_t} have a unique pathwise solution, i.e.,
 \begin{itemize}
 \item  $u $ (resp. $\theta$) is an adapted $V^0$-valued (resp. $H^0$-valued) process which belongs a.s. to $L^2(0,T ; V^1)$
 (resp. to $L^2(0,T;H^1)$),
 \item $\PP$ a.s. we have  $u\in C([0,T];V^0)$, $\theta \in C(0,T];H^0)$, and 
  \begin{align*}
  \big(u(t), \varphi\big) +& \nu \int_0^t \big( A^{\frac{1}{2}} u(s), A^{\frac{1}{2}} \varphi\big) ds 
  + \int_0^t \big\langle [u(s) \cdot \nabla]u(s), \varphi\big\rangle ds \\
 &    = \big( u_0, \varphi) + \int_0^t \big( \Pi \theta(t) v_2,\varphi) ds + \int_0^t \big( \varphi , G(u(s))  dW(s) \big),\\
  \big(\theta(t), \psi\big) +& \kappa \int_0^t \big( \AA^{\frac{1}{2}} \theta(s), \AA^{\frac{1}{2}} \psi\big) ds 
  + \int_0^t \big\langle [u(s) \cdot \nabla]\theta(s), \psi\big\rangle ds  + \int_0^t \big( C_L u_2, \psi) ds \\
 &    = \big( \theta_0, \psi) + \int_0^t \big( \psi ,  \GG(\theta(s))  d\WW(s) \big),
 \end{align*}
for every $t\in [0,T]$ and every $\varphi \in V^1$ and $\psi\in H^1$.
  \end{itemize}
Furthermore, 
 \begin{align}		\label{mom_ut_L2}
 \EE\Big( \sup_{t\in [0,T]} &\big[ \|u(t)\|_{V^0}^{2p} +  \|\theta(t)\|_{H^0}^{2p} \big] \nonumber \\
 + & \int_0^T  \big[  \|A^{\frac{1}{2}} u(t)\|_{V^0}^2 \, \big( 1+\|u(t)\|_{V^0}^{2(p-1)}\big) 
 +  |\AA^{\frac{1}{2}} \theta(t)\|_{H^0}^2 \, \big( 1+\|\theta(t)\|_{H^0}^{2(p-1)} \big )\big] dt \Big)  \nonumber \\
 & \quad \leq C\big[ 1+\EE\big(\|u_0\|_{V^0}^{2p} + \|\theta_0\|_{H^0}^{2p}\big)\big].	
 \end{align}
 \end{theorem}

The following result proves that if $u_0\in V^1$ and $\theta_0\in H^0$, the solution $u,\theta$ to  \eqref{def_u}--\eqref{def_t} is more regular.
 It is proven in \cite{BeMi_Bou} Proposition 3.2 for a slightly simpler equation (without the B\'enard modification). 
 Since the evolution equation of $u$ is not affected by the change of function describing the temperature, the argument is unchanged. 
\begin{prop}  	\label{prop_u_V1}
Let $u_0\in L^{2p}(\Omega;V^1)$, $\theta_0\in L^{2p}(\Omega;H^0)$ for    $p\in [2,\infty)$. 
 Then the solution $u$ to  \eqref{def_u}--\eqref{def_t} belongs a.s. to $ C([0,T];V^1)  \cap L^2([0,T];V^2)$.  Moreover,
 \begin{equation}		\label{mom_u_V1}
\EE\Big( \sup_{t\in [0,T]} \|u(t)\|_{V^1}^{2p} + \int_0^T \!\! \| A u(t)\|_{V^0}^2 \,\big[ 1+  \| A^{\frac{1}{2}} u(t) \|_{V^0}^{2(p-1)}\big] dt \Big) 
 \leq C  \big[ 1+\EE\big( \|u_0\|_{V^1}^{2p} + \|\theta_0\|_{H^0}^{2p}\big)\big].
 \end{equation}
\end{prop}

The next result proves similar bounds for moments of the gradient of the temperature, uniformly in time.  It is proven in \cite{BeMi_Bou} Proposition 3.3
for a simpler model with no B\'enard correction term. 
Recall that  the higher moments required on the norms of the initial conditions are related to the fact that some
estimate of $\langle [ u(t) . ] \theta(t) \, , \, \tilde{A} \theta(t)\rangle$ has to be used since this term does not vanish. 	For the sake of completeness we prove it in the 
Appendix.

\begin{prop}		\label{prop_mom_t_H1}
Let $u_0\in L^{8p+\epsilon}(\Omega;V^1)$ and $\theta_0\in L^{8p+\epsilon}(\Omega;H^1)$ for some $\epsilon >0$ and  $p\in [2,\infty)$. 
There exists a constants $C$ 
such that 
\begin{equation} 		\label{mom_t_H1}
\EE\Big( \sup_{t\leq T}  
 \| \theta(t)\|_{H^1}^{2p}
 \Big) \leq C. 
\end{equation}
\end{prop}

We conclude this section stating the moment estimates of time increments of both $u$ and $\theta$ in various norms. 
The first result provides moments of the time increments in the $V^0$ and $H^0$-norms. The proof can be found in \cite[Proposition~4]{BeMi_Bou}.
The extra linear term coming from the B\'enard correction in the evolution equation of the temperature is easily dealt with. 
 \begin{prop}	\label{prop_regularity_L2} 
Let $u_0, \theta_0$ be ${\mathcal F}_0$-measurable and $p\in [2,\infty)$.

(i) Let  $u_0\in L^{4p}(\Omega;V^1)$ and $\theta_0\in L^{2p}(\Omega;H^0)$. 
 For $0\leq \tau_1<\tau_2 \leq T$,
\begin{equation}		\label{increm_u_L2}
\EE\big( \|u(\tau_2) -u(\tau_1) \|_{V^0}^{2p}\big) \leq C \,   |\tau_2-\tau_1|^{p}.
\end{equation}

(ii) Let $u_0\in L^{8p+\epsilon}(\Omega;V^1)$ and $\theta_0\in L^{8p+\epsilon}(\Omega;H^1)$ for some  $\epsilon >0$. 
Then for  $0\leq \tau_1<\tau_2 \leq T$,
\begin{equation} 	\label{increm_t_L2}
\EE\big(  \|\theta(\tau_2) -\theta(\tau_1) \|_{H^0}^{2p} \big) \leq C \, |\tau_2-\tau_1|^{ p}.
\end{equation} 
\end{prop}
The second result gives moments of the time increments of $u$ and $\theta$ in $V^1$ and $H^1$ in some integral-like form. 
Its proof  can be found in \cite[Proposition~5]{BeMi_Bou}.
Once more the extra linear term coming from the B\'enard correction does not affect the result.
\begin{prop}			\label{prop_regularity_H1}
 Let $N\geq 1$ be an integer,  for $k=0, \cdots, N$ set $t_k=\frac{kT}{N}$,  and let $\eta \in (0,1)$. 

(i) Let $p\in [2,\infty)$,  $u_0\in L^{4p}(\Omega;V^1)$ and $\theta_0\in L^{2p}(\Omega;H^0)$. 
Then
there exists  a positive constant $C$ (independent of $N$) such that 
\begin{align}			\label{mom_increm_u_1}
\EE\Big(  \Big|  \sum_{j=1}^N \int_{t_{j-1}}^{t_j}\!\! & \big[ 
\| u(s)-u(t_j)\|_{V^1}^2 + \|u(s)-u(t_{j-1})\|_{V^1}^2 \big] ds \Big|^p \Big)
\leq C \Big( \frac{T}{N}\Big)^{\eta p} 
\end{align} 

(ii) Let $p\in [2,\infty)$, $u_0\in L^{16p+\epsilon}(\Omega;V^1)$ and $\theta_0\in L^{16p+\epsilon}(\Omega;H^0)$ for some $\epsilon >0$. Then
 \begin{align}			\label{mom_increm_t_1}
\EE\Big(  \Big|  \sum_{j=1}^N \int_{t_{j-1}}^{t_j}\!\! & \big[ 
\| \theta(s)-\theta(t_j)\|_{H^1}^2 + \| \theta(s)-\theta(t_{j-1})\|_{H^1}^2 
\big] ds \Big|^p \Big)
\leq C \Big( \frac{T}{N}\Big)^{\eta p} 
\end{align} 
\end{prop}

\section{The semi-implicit time Euler scheme with multiplicative noise }  	\label{scheme}
We next define the semi-implicit Euler scheme. 
 Fix $N\in \{1,2, ...\}$, let $h:=\frac{T}{N}$ denote the time mesh,
 and for $j=0, 1, ..., N$ set $t_j:=j \frac{T}{N}$.  
 
 Set $\Delta_l W:= W(t_l)-W(t_{l-1})$ and $\Delta_l \WW=\WW(t_l)-\WW(t_{l-1})$, $l=1, ...,N$.

The semi-implicit time Euler scheme $\{ u^k ; k=0, 1, ...,N\}$  and $\{ \theta^k ; k=0, 1, ...,N\}$ is defined by $u^0=u_0$, $\theta^0=\theta_0$,
 and for $\varphi \in V^1$, $\psi\in H^1$ and $l=1, ...,N$, 
\begin{align}	\label{Euler_u}
\Big( u^l-u^{l-1} &+ h \nu A u^l + h B\big( u^{l-1},u^l\big)   , \varphi\Big) = \big(\Pi (\theta^{l-1} v_2), \varphi) h  + \big(G(u^{l-1}) \Delta_l W\, , \, \varphi\big),  \\
\Big( \theta^l-\theta^{l-1} &+ h \kappa \AA \theta^l +h [ u^{l-1}.\nabla]\theta^l +  h C_L u^{l-1}_2 , \psi\Big) = \big( \GG(\theta^{l-1}) \Delta_l\WW\, , \, \psi\big). \label{Euler_t}
\end{align}

Since the scheme is semi-implicit, it is defined in terms of linear equations, and hence there is a unique solution. Furthermore, it is easy to see that 
  $\{u^l\}_{l=1, ...,N} \in V^1$ and   $\{\theta^l\}_{l=1, ...,N} \in H^1$,  and that $u^l$ and $\theta^l$ are ${\mathcal F}_{t_l}$-measurable for every $l=0, ..., N$. 
\subsection{Moments of the semi-implicit time Euler scheme}
 We next state upper bounds of moments of $u^k$ and $\theta^k$ uniformly in $k=1, ..., N$. 
The proof is similar to that of Proposition~6.1  in \cite[Proposition~6]{BeMi_Bou} (see also \cite{BrCaPr}) for a fully implicit scheme, 
that is when in equation \eqref{Euler_u} we replace $B(u^{l-1}, u^l)$ by
$B(u^l,u^l)$. Indeed, in both schemes  we have $\big\langle B( u^{l-1},u^l )   , u^l \big\rangle =
\big\langle B( u^l,u^l )   , u^l \big\rangle= 0$. Furthermore,  the extra term $h C_L u^{l-1}_2$ is linear in the pair $(u,\theta)$, which does not modify the moment estimates.
\begin{prop}		\label{prop_mom_scheme}
 Let $u_0\in L^{2^K}(\Omega;V^0) $ and
$\theta_0\in L^{2^K}(\Omega;H^0)$ respectively. Let $\{u^k\}_{k=0, ...,N}$ and $\{\theta^k\}_{k=0, ..., N}$ be solution of \eqref{Euler_u} and
\eqref{Euler_t} respectively. Then
\begin{align}		\label{mom_scheme_0}
\sup_{N\geq 1} \EE\Big(   \max_{0\leq L\leq N} \|u^L\|_{V^0}^{2^K} + \max_{0\leq L \leq N} \|\theta^L\|_{H^0}^{2^K} \Big)  <\infty \\
\sup_{N\geq 1} \EE\Big( h \sum_{l=1}^N \|A^{\frac{1}{2}} u^l\|_{V^0}^2 \|u^l\|_{V^0}^{2^K-2} +  h \sum_{l=1}^N \|\AA^{\frac{1}{2}} \theta^l\|_{H^0}^2 
\|\theta^l\|_{H^0}^{2^K-2}\Big)  <\infty, 	\label{mom_scheme_1} 
\end{align} 
\end{prop}

\subsection{Strong convergence of a localized semi-implicit time Euler scheme} \label{loc-conv-multi}
Due to the bilinear terms $[u.\nabla]u$ and $[u.\nabla]\theta$, we first prove an $L^2(\Omega)$ convergence of the $\LL^2(D)$-norm of
the error,  uniformly on the time grid,  restricted to the set $\Omega_M(N)$ defined below for some $M>0$:
\begin{equation} 		\label{def-A(M)}
\Omega_M(j):= \Big\{ \sup_{s\in [0,t_j]} \|A^{\frac{1}{2}} u(s)\|_{V^0}^2 \leq M\Big\} \cap \Big\{ \sup_{s\in [0,t_j]} \|\AA^{\frac{1}{2}} \theta(s)\|_{H^0}^2 \leq M\Big\},
\; \forall j=0, ..., N,
\end{equation} 
and let $\Omega_M:= \Omega_M(N)$. 
Recall that, for $j=0, ...,N$, set $e_j:= u(t_j)-u^j$ and $\tilde{e}_j:= \theta(t_j)-\theta^j$; then, $e_0=\tilde{e}_0=0$. 
Using \eqref{def_u}, \eqref{def_t}, \eqref{Euler_u} and \eqref{Euler_t}, we deduce, for $j=1, ...,N$, $\phi\in V^1$ and $\psi\in H^1$, that
\begin{align}	\label{def_error_u}
\big( &e_j-e_{j-1}\, , \, \varphi \big) + \nu \!\int_{t_{j-1}}^{t_j} \! \! \!\big(  A^{\frac{1}{2}}  [ u(s) -   u^j ] ,  A^{\frac{1}{2}}  \varphi\big)  ds 
+ \int_{t_{j-1}}^{t_j} \!\! \big\langle B(u(s),u(s)) -   B(u^{j-1},u^j) ,  \varphi\big\rangle ds \nonumber \\
& =  \int_{t_{j-1}}^{t_{j}} \big( \Pi [\theta(s)-\theta^{j-1} ] v_2, \varphi\big) ds + 
 \int_{t_{j-1}}^{t_j} \big( [G(u(s))-G(u^{j-1}) ] dW(s) \, , \, \varphi\big),  
\end{align}
and
\begin{align}		\label{def_error_t}
\big( &\ee_j-\ee_{j-1}\, , \, \psi \big) + \kappa \!\int_{t_{j-1}}^{t_j} \! \! \big( \AA^{\frac{1}{2}}  [ \theta(s) -   \theta^j ] ,  \AA^{\frac{1}{2}}  \psi\big)  ds 
+ \int_{t_{j-1}}^{t_j} \!\! \big\langle [u(s).\nabla]\theta(s)  - [u^{j-1}.\nabla]\theta^j] ,  \psi\big\rangle ds \nonumber \\
& =  - C_L \int_{t_{j-1}}^{t_j} \big( u_2(s) - u^{j-1}_2, \psi\big) ds +  \int_{t_{j-1}}^{t_j} \big( [\GG(\theta(s))-\GG(\theta^{j-1}) ] d\WW(s) \, , \, \psi\big).  
\end{align}

In this section, we will suppose that $N$ is large enough to have $h:=\frac{T}{N} \in (0,1)$. The following result is  a crucial step towards
the rate of convergence of the implicit time Euler scheme.
 Since the model and the  scheme are different from what was used in \cite{BeMi_Bou}, we give a complete proof.
\begin{prop}	\label{th_loc_cv}
Suppose that  the  conditions {\bf (C-u)} and {\bf (C-$\theta$)} hold. \\
Let  $u_0\in L^{16+\epsilon}(\Omega ; V^1)$ and $\theta_0\in L^{16+\epsilon}(\Omega;H^1)$ for some $\epsilon>0$,
  $u,\theta$ be the solution to \eqref{def_u} and \eqref{def_t} and
$\{u^j, \theta^j\}_{j=0, ..., N}$ be the solution to \eqref{Euler_u} and \eqref{Euler_t}. Fix $M>0$ and let $\Omega_M=\Omega_M(N)$ 
be defined by \eqref{def-A(M)}. Then, 
 for $\eta\in (0,1)$, there exists a  positive constant  $C$,
 independent of $N$, such that,  for large enough $N$, 
\begin{align} 		\label{loc_cv}
\EE\Big( &1_{\Omega_M} \Big[ \max_{1\leq j\leq N} \big( \|u(t_j)-u^j\|_{V^0}^2 + \|\theta(t_j)-\theta^j\|_{H^0}^2 \big)  + \frac{T}{N} \sum_{j=1}^N 
\big[ \|A^{\frac{1}{2}} [ u(t_j) - u^j] \|_{V^0}^2 \nonumber \\
&\quad + \|\AA^{\frac{1}{2}} [ \theta(t_j) - \theta^j] \|_{H^0}^2  \Big] \Big) 
\leq  C (1+M)   e^{{\mathcal C}(M) T} \Big( \frac{T}{N}\Big)^{\eta},
\end{align}
where for some $\gamma >0$ 
\[  {\mathcal C}(M) =  \frac{9(1+\gamma) \bar{C}_4^4}{8}  \max\Big(\frac{5}{ \nu}  \, , 
 \, \frac{1}{\kappa}\Big) \, M\]  
 and $\bar{C}_4$ is the constant in the 
right hand side of the Gagliardo--Nirenberg inequality \eqref{GagNir}. 
\end{prop}
\begin{proof} 
Write \eqref{def_error_u} with $\varphi = e_j$; 
using the equality $(f,f-g)=\frac{1}{2}\big[ \|f\|_{\LL^2}^2 - \|g\|_{\LL^2}^2 + \|f-g\|_{\LL^2}^2\big]$, we obtain for $j=1, ..., N$
\begin{align}	\label{increm_error_u}
 \|e_j\|_{V^0}^2 - \|e_{j-1}\|_{V^0}^2 &+  \|e_j-e_{j-1}\|_{V^0}^2 + 2\nu h \|A^{\frac{1}{2}}  e_j\|_{V^0}^2  
 \leq2  \sum_{l=1}^7 T_{j,l}, 
 \end{align}
where, by the antisymmetry property \eqref{B}, we have 
\begin{align*}
T_{j,1}=&-\int_{t_{j-1}}^{t_j} \!\! \big\langle B\big(e_{j-1}, u^j\big) \, , \, e_j\big\rangle ds = 
-\int _{t_{j-1}}^{t_j} \!\! \big\langle B\big(e_{j-1}, u(t_j)\big) \, , \, e_j\big\rangle ds , \\
T_{j,2}=&-\int_{t_{j-1}}^{t_j}\!\!  \big\langle B\big(u(s)-u(t_{j-1})\, , u^j \big) \, e_j\big\rangle ds = - \int_{t_{j-1}}^{t_j}\!\!  \big\langle B\big(u(s)-u(t_{j-1})\, , u(t_j)\big) \, e_j\big\rangle ds , \\
T_{j,3}=&-  \int_{t_{j-1}}^{t_j}\!\!  \big\langle B\big( u(s), u(s)-u^j\big) \, , \, e_j\big\rangle ds = 
-\int_{t_{j-1}}^{t_j} \!\! \big\langle B\big( u(s), u(s)-u(t_j) \big)  \, , \,   e_j \big\rangle ds, \\
T_{j,4}=&-\nu \!\int_{t_{j-1}}^{t_j}\!\! \! \big( A^{\frac{1}{2}}  ( u(s)-u(t_j)) ,  A^{\frac{1}{2}}  e_j\big) ds, \quad 
T_{j,5}= \int_{t_{j-1}}^{t_j} \!\! \big( \Pi [\theta(s) - \theta^{j-1}] v_2 \, , \, e_j\big) ds, \\
T_{j,6}=& \int_{t_{j-1}}^{t_j}\!\! \! \big([G(u(s))-G(u^{j-1})\big] dW(s)\, ,\,  e_j-e_{j-1}\big), \\
T_{j,7}=& \int_{t_{j-1}}^{t_j}\!\!  \big([G(u(s))-G(u^{j-1}) \big] dW(s), e_{j-1}\big).
\end{align*}
We first prove upper estimates of  the terms $T_{j,l}$  for $l=1, ...,5$. 
The H\"older and Young inequalities and the Gagliardo--Nirenberg inequality \eqref{GagNir}
imply  for $\delta_1>0$ and  $j=1, ..., N$ 
 \begin{align}		\label{maj_Tj1}
 |T_{j,1}|\leq & \, \bar{C}_4^2 \, h\, \|A^{\frac{1}{2}} e_{j-1}\|_{V^0}^{\frac{1}{2}} \|e_{j-1}\|_{V^0}^{\frac{1}{2}}  
  \|A^{\frac{1}{2}} e_{j}\|_{V^0}^{\frac{1}{2}}   \|e_{j}\|_{V^0}^{\frac{1}{2}}  \|A^{\frac{1}{2}} u(t_j)\|_{V^0}  \nonumber \\
 \leq & \, \delta_1\, \nu\,  h\,  \|A^{\frac{1}{2}} e_{j-1}\|_{V^0}^2 + \delta_1\, \nu\,  h\,  \|A^{\frac{1}{2}} e_j\|_{V^0}^2  + 
 \frac{ \bar{C}_4^4}{16 \delta_1 \nu} \, h\,  \|A^{\frac{1}{2}} u(t_j)\|_{V^0}^2 \|e_{j-1}\|_{V^0}^2 \nonumber \\
 &\quad + \frac{ \bar{C}_4^4}{16 \delta_1 \nu} \, h\,  \|A^{\frac{1}{2}} u(t_j)\|_{V^0}^2 \|e_{j}\|_{V^0}^2, 
  \end{align}
  A similar computation yields for $\delta_2, \delta_3 >0$ and $\lambda_1$ defined by \eqref{lambda_1} 
\begin{align} 	\label{up_Tj2}
|T_{j,2}| \leq & \, \bar{C}_4^2  \lambda_1^{-\frac{1}{4}} \int_{t_{j-1}}^{t_j} \!\! \!  \| A^{\frac{1}{2}} [u(s)-u(t_{j-1})]\|_{V^0} 
\| A^{\frac{1}{2}} u(t_j)\|_{V^0} \| A^{\frac{1}{2}} e_j\|_{V^0}^{\frac{1}{2}}
\|e_j\|_{V^0}^{\frac{1}{2}} ds \nonumber \\
\leq & \, \delta_2 h\nu \|A^{\frac{1}{2}} e_j\|_{V^0}^2 + h \|e_j\|_{V^0}^2 \nonumber \\
&\quad + \frac{\bar{C}_4^4}{8\sqrt{\nu \delta_2 \lambda_1 }} \sup_{s\in [t_{j-1}  t_j]}  \|A^{\frac{1}{2}} u(s)\|_{V^0}^2 \int_{t_{j-1}}^{t_j} \!\!  \|A^{\frac{1}{2}} [u(s)-u(t_{j})] \|_{V^0}^2 ds, 
\end{align}
and 
\begin{align} 	\label{up_Tj3}
|T_{j,3}| \leq & \; \bar{C}_4^2 \lambda_1^{-\frac{1}{4}} 
 \int_{t_{j-1}}^{t_j} \!\! \! \| A^{\frac{1}{2}} u(s) \|_{V^0} \| A^{\frac{1}{2}} [u(s)-u(t_j)]\|_{V^0} \| A^{\frac{1}{2}} e_j \|_{V^0}^{\frac{1}{2}} \|e_j\|_{V^0}^{\frac{1}{2}}   ds \nonumber \\
\leq & \; \delta_3 \nu h \|A^{\frac{1}{2}} e_j\|_{V^0}^2 +  h \| e_j\|_{V^0}^2  \nonumber \\
&\; +  \frac{\bar{C}_4^4}{8\sqrt{\nu \delta_3 \lambda_1 }} \sup_{s\in [t_{j-1}, t_j]} \|A^{\frac{1}{2}} u(s)\|_{V^0}^2 
 \int_{t_{j-1}}^{t_j} \| A^{\frac{1}{2}} [u(s)-u(t_j)]\|_{V^0}^2 ds .
\end{align}
The Cauchy-Schwarz and Young inequalities imply  for $\delta_4>0$ 
\begin{align}		\label{up_Tj4}
|T_{j,4}|\leq &\;  \nu \int_{t_{j-1}}^{t_j} \!\!  \| A^{\frac{1}{2}} e_j\|_{V^0} \|A^{\frac{1}{2}} [u(s)-u(t_j)] \|_{V^0}^2 ds \nonumber \\
\leq & \; \delta_4 \nu h \|A^{\frac{1}{2}} e_j\|_{V^0}^2 + \frac{\nu}{4\delta_4} \int_{t_{j-1}}^{t_j} \!\! \| A^{\frac{1}{2}} [u(s)-u(t_j)]\|_{V^0}^2 ds. 
\end{align}
Finally, using once more the Cauchy-Schwarz and Young inequalities, we deduce 
\begin{align}		\label{up_Tj5}
|T_{j,5}|\leq &\; \int_{t_{j-1}}^{t_j} \Big( [\theta(s)-\theta(t_{j-1})] v_2 , e_j) ds + h \| \ee_{j-1} \|_{H^0}  \|e_j\|_{V^0}  \nonumber \\
\leq & \; h \|e_j\|_{V^0}^2 + \frac{h}{2}   \| \ee_{j-1}\|_{H^0}^2 
+ \frac{1}{2} \int_{t_{j-1}}^{t_j} \!\!  \|\theta(s)-\theta(t_{j-1})\|_{H^0}^2 ds.
\end{align}
 We next prove similar estimates for $\theta$.  Writing and \eqref{def_error_t} with $\psi=\theta^j$ we deduce 
\begin{align} 
\|\ee_j\|_{H^0}^2 - \|\ee_{j-1}\|_{H^0}^2 &+  \|\ee_j-\ee_{j-1}\|_{H^0}^2 + 
2\kappa h \|\AA^{\frac{1}{2}}  \ee_j\|_{H^0}^2   \leq2 \sum_{l=1}^7 \tilde{T}_{j,l},		\label{incrfem_error_t}
\end{align}
where using  the antisymmetry property \eqref{antisym-ut} we have 
\begin{align*}
\tilde{T}_{j,1}=&-\int_{t_{j-1}}^{t_j} \!\! \big\langle [e_{j-1}.\nabla] \theta^j  \, , \, \ee_j \big\rangle ds = 
-\int _{t_{j-1}}^{t_j} \!\! \big\langle  [e_{j-1}.\nabla] \theta(t_j) \, , \, \ee_j\big\rangle ds , \\
\tilde{T}_{j,2}=&-\int_{t_{j-1}}^{t_j}\!\! \! \! \!\big\langle \big[ \big((u(s)-u(t_{j-1})\big) .\nabla \big]  \theta^j   , \ee_j\big\rangle ds 
= -\int_{t_{j-1}}^{t_j}\!\! \!  \big\langle \big[ \big((u(s)-u(t_{j-1})\big) .\nabla \big]   \theta(t_j)   \, ,\, \ee_j\big\rangle ds,\\
\tilde{T}_{j,3}=&-  \int_{t_{j-1}}^{t_j}\!\!  \big\langle [u(s).\nabla] \big(\theta(s)-\theta^j \big)  \, , \, \ee_j\big\rangle ds = 
\int_{t_{j-1}}^{t_j} \!\! \big\langle  [u(s).\nabla]  \big(\theta(s)-\theta(t_j)\big) , \ee_j  \big\rangle ds, \\
\tilde{T}_{j,4}=&-\nu \!\int_{t_{j-1}}^{t_j}\!\! \! \big( \AA^{\frac{1}{2}}  ( \theta(s)-\theta(t_j)) ,  \AA^{\frac{1}{2}}  \ee_j\big) ds,  \quad
\tilde{T}_{j,5}= -C_L \! \int_{t_{j-1}}^{t_j} \!\! ( u_2(s)-u_2^{j-1} , \ee_j) ds, \\
\tilde{T}_{j,6}=& \int_{t_{j-1}}^{t_j}\!\! \! \big([\GG(\theta(s))-\GG(\theta^{j-1}) d\WW(s)\, ,\,  \ee_j-\ee_{j-1}\big), \\
\tilde{T}_{j,7}=& \int_{t_{j-1}}^{t_j}\!\!  \big([G(u(s))-G(u^{j-1}) \big] dW(s), e_{j-1}\big).
\end{align*}
We next prove upper estimates of $\tilde{T}_{j,i}$ for $i=1, ..., 5$. 
 Using the H\"older, Gagliardo and Young inequalities, we deduce for $\delta_5, \tilde{\delta}_1>0$  and  $j=1, ..., N$
\begin{align}		\label{up_tildeTj1}
|\tilde{T}_{j,1}| \leq &\;  h \; \bar{C}_4^2 \; \|e_{j-1}\|_{V^0}^{\frac{1}{2}} \|A^{\frac{1}{2}} e_{j-1}\|_{V^0}^{\frac{1}{2}} \|\tilde{A}^{\frac{1}{2}} \theta(t_j)\|_{H^0} \|\ee_j\|_{H^0}^{\frac{1}{2}}
\| \tilde{A}^{\frac{1}{2}} \ee_j\|_{H^0}^{\frac{1}{2}} \nonumber \\
\leq & \; \delta_5 \nu h  \|A^{\frac{1}{2}} e_{j-1}\|_{V^0}^2 + \tilde{\delta}_1 \kappa h \|\tilde{A}^{\frac{1}{2}} \ee_j\|_{H^0}^2  \nonumber \\
&\;  + \frac{\bar{C}_4^4}{16 \nu \delta_5}h  \|\tilde{A}^{\frac{1}{2}} \theta(t_j)\|_{H^0}^2
\|e_{j-1}\|_{V^0}^2 +  \frac{\bar{C}_4^4}{16 \kappa \tilde{\delta}_1}h  \|\tilde{A}^{\frac{1}{2}} \theta(t_j)\|_{H^0}^2  \| \ee_j\|_{H^0}^2 .
\end{align}
Similar computations yield for $\tilde{\delta}_2>0$ and ${\lambda}_1$ defined by \eqref{lambda_1} 
\begin{align}		\label{up_tildeTj2}
|\tilde{T}_{j,2}| \leq &\;  \bar{C}_4^2  \lambda_1^{-\frac{1}{4}}  \int_{t_{j-1}}^{t_j}\!\!   \|A^{\frac{1}{2}} [u(s)-u(t_{j-1})] \|_{V^0} \| \tilde{A}^{\frac{1}{2}} \theta(t_j)\|_{H^0}
\| \ee_j\|_{H^0}^{\frac{1}{2}} \| \tilde{A}^{\frac{1}{2}} \ee_j\|_{H^0}^{\frac{1}{2}} ds
 \nonumber \\
\leq & \; \tilde{\delta}_2 \kappa h \|\tilde{A}^{\frac{1}{2}} \ee_j\|_{H^0}^2  +h  \| \ee_j\|_{H^0}^2  \nonumber \\%
& \; + \frac{\bar{C}_4^4}{8 \sqrt{\kappa \tilde{\delta}_2 \lambda_1}}  \sup_{s\in [t_{j-1}, t_j]} \| \tilde{A}^{\frac{1}{2}} \theta(s) \|_{H^0}^2
 \int_{t_{j-1}}^{t_j} \!\! \|u(s)-u(t_{j-1})\|_{V^0}^2 ds, 
\end{align}
while a similar argument yields  for $\tilde{\delta}_3>0$ 
\begin{align}		\label{up_tildeTj3}
|\tilde{T}_{j,3}| \leq &\;  \bar{C}_4^2 \lambda_1^{-\frac{1}{4}} \; \int_{t_{j-1}}^{t_j}\!\!    \|A^{\frac{1}{2}} u(s) \|_{V^0}  \| \tilde{A}^{\frac{1}{2}}[ \theta(s)- \theta(t_j)]\|_{H^0}
\| \ee_j\|_{H^0}^{\frac{1}{2}} \| \tilde{A}^{\frac{1}{2}} \ee_j\|_{H^0}^{\frac{1}{2}} ds
 \nonumber \\
\leq & \; \tilde{\delta}_3 \kappa h \|\tilde{A}^{\frac{1}{2}} \ee_j\|_{H^0}^2  +  h \| \ee_j\|_{H^0}^2  \nonumber \\
&\;  + \frac{\bar{C}_4^4}{8 \sqrt{\kappa \tilde{\delta}_3 \lambda_1} } \sup_{s\in [t_{j-1}, t_j]}  \| A^{\frac{1}{2}} u(s) \|_{V^0}^2  \int_{t_{j-1}}^{t_j} 
\!\!  \|\tilde{A}^{\frac{1}{2}}[ \theta(s)-\theta(t_{j})]\|_{H^0}^2 ds .
\end{align}
The Cauchy-Schwarz and Young inequalities imply for $\tilde{\delta}_4>0$
\begin{align}		\label{up_tildeTj4}
|\tilde{T}_{j,4}|\leq &  \tilde{\delta}_4 \kappa h  \|\tilde{A}^{\frac{1}{2}} \ee_j\|_{H^0}^2  + \frac{\kappa}{4 \tilde{\delta}_4 }  \int_{t_{j-1}}^{t_j}  \|\tilde{A}^{\frac{1}{2}} [ \theta(s)-\theta(t_j)|\|_{H^0}^2 ds,
\end{align}
and
\begin{align}		\label{up_tildeTj5}
|\tilde{T}_{j,5}|\leq &\;  |C_L| \int_{t_{j-1}}^{t_j} \!\!  \|u(s)-u(t_{j-1})\|_{V^0} \|\ee_j\|_{H^0} ds + h \, |C_L|\,  \|e_{j-1}\|_{V^0} \|\ee_j\|_{H^0} \nonumber \\
\leq & \;  h\, |C_L|\,    \|\ee_j\|_{H^0}^2 +   h \, \frac{|C_L|}{2}\,    \|e_{j-1}\|_{V^0}^2 + \frac{|C_L|}{2} \int_{t_{j-1}}^{t_j} \|u(s)-u(t_{j-1})\|_{V^0}^2 ds. 
\end{align}
Note that the sequence of subsets $\{ \Omega_M(j)\}_{0\leq j\leq N}$ is decreasing. Therefore, since $e_0=\ee_0=0$, given $L=1, ..., N$, 
 we obtain
 \begin{align*}
 \max_{1\leq J\leq L} &\sum_{j=1}^J   1_{\Omega_M({j-1})}  \big[ \|e_j\|_{V^0}^2 - \| e_{j-1}\|_{V^0}^2 + \|\ee_j\|_{H^0}^2  - \| \ee_{j-1}\|_{H^0}^2 \big]  \\
&\;  = \max_{1\leq J\leq L}  \Big[  \Big( 1_{\Omega_M({J-1})} \big[ \|e_J\|_{V^0}^2 + \|\ee_J\|_{H^0}^2 \big] \Big) \\
&\qquad  + \sum_{j=2}^J \big( 1_{\Omega_M(j-2)} - 1_{\Omega_M({j-1})}\big) \big[ \|e_{j-1}\|_{V^0}^2 + \|\ee_{j-1} \|_{H^0}^2 \big]  \Big]\\
&\; \geq \max_{1\leq J\leq L}  \Big( 1_{\Omega_M({J-1})} \big[ \|e_J\|_{V^0}^2 + \|\ee_J\|_{H^0}^2 \big] \Big).
 \end{align*} 
Given non negative real numbers $a(i,J)$ for $i=1, ...,3$ and $J=1, ..., L$, we have \linebreak 
$ \sum_{i=1}^3 \max_{1\leq J\leq L}  a(i,J) \leq 3 \max_{1\leq J\leq L} \sum_{i=1}^3 a(i,J)$. 

 In the right hand side of \eqref{maj_Tj1} we encounter the product $\|e_{j-1}\|_{V^0}^2 \|A^{\frac{1}{2}} u(t_j)\|_{V^0}^2$. 
For measurability issues, we have to replace $\|A^{\frac{1}{2}} u(t_j)\|_{V^0}^2$ by $\|A^{\frac{1}{2}} u(t_{j-1} )\|_{V^0}^2$. 
Young's inequality implies  for any $\alpha >0$ 
\begin{align*} 
h\|A^{\frac{1}{2}} u(t_j)\|_{V^0}^2 \leq &\;  h (1+\alpha)  \|A^{\frac{1}{2}} u(t_{j-1})\|_{V^0}^2  \\
&\; +   \big( 1+\alpha^{-1}\big) \int_{t_{j-1}}^{t_j}  \big( \|A^{\frac{1}{2}} [u(s)-u(t_j)| \|_{V^0}^2 + 
\|A^{\frac{1}{2}} [u(s)-u(t_{j-1})] \|_{V^0}^2 \big) ds.
\end{align*} 
A similar upper estimate relies $h\|A^{\frac{1}{2}}u(t_j)\|_{V^0}^2 $ and $h\| A^{\frac{1}{2}} u(t_{j-2})\|_{V^0}^2 $ (using $u(t_{j-1)}$ as an intermediate step) for $j=2, ..., N$ 
as well as $h \|\AA^{\frac{1}{2}} \theta(t_j)\|_{H^0}^2 $ and $h \| \AA^{\frac{1}{2}} \theta(t_{j-1})\|_{H^0}^2 $ (resp.  $h \|\AA^{\frac{1}{2}} \theta(t_j)\|_{H^0}^2 $ and
$h \|\AA^{\frac{1}{2}} \theta(t_{j-2})\|_{H^0}^2 $) for $j=1, ..., N$ (resp. for $j=2, ..., N$).\\
Hence, adding the upper estimates \eqref{increm_error_u}--\eqref{up_tildeTj5}  and using Young's inequality, we deduce that for 
$2\delta_1 + \sum_{j=2}^5 \delta_j < \frac{1}{3}$ and $\sum_{j=1}^4 \tilde{\delta}_j<\frac{1}{3}$, we have 
 for every $\alpha >0$
 \vspace{-6pt}
\begin{align}		\label{maj_error-1}
 \frac{1}{3}& \max_{1\leq J\leq L}  \Big( 1_{\Omega_M({J-1})} \big[ \|e_J\|_{V^0}^2 + \|\ee_J \|_{H^0}^2 \big] \Big)  \nonumber \\
 &\quad +
 \frac{1}{3} \sum_{j=1}^L 1_{\Omega_M(j-1)} \big( \|e_j - e_{j-1}\|_{V^0}^2 + \|\ee_j-\ee_{j-1}\|_{H^0}^2 \big) 		\nonumber \\
 &\quad + 2  \sum_{j=1}^L 1_{\Omega_M(j-1)}\, h\,  \Big[ \nu \Big( \frac{1}{3}-2\delta_1 - \sum_{i=2}^5 \delta_i \Big) \|A^{\frac{1}{2}} e_j\|_{V^0}^2 
 + \kappa \Big( \frac{1}{3}-\sum_{i=1}^4 \tilde{\delta}_i \Big) \|\AA^{\frac{1}{2}} \ee_j\|_{H^0}^2  \Big]  		\nonumber \\
 & \leq \,2  h \sum_{j=1}^{L-1}  1_{\Omega_M(j-1)} \|e_j\|_{V^0}^2 \Big[ \frac{(1+\alpha) \bar{C}_4^2 }{16\nu } \Big( \frac{2\|A^{\frac{1}{2}} u(t_{j-1})\|_{V^0}^2  }{\delta_1}  
  + \frac{\| \AA^{\frac{1}{2}} \theta(t_{j-1})\|_{H^0}^2}{\delta_5}  \Big)  +3 + \frac{|C_L|}{2} \Big]		\nonumber \\
 &\qquad + 2 h \sum_{j=1}^{L-1} 1_{\Omega_M(j-1)} \| \ee_j\|_{H^0}^2 \Big[ \frac{(1+\alpha) \bar{C}_4^2}{16 \tilde{\delta}_1 \kappa} 
 \| \AA^{\frac{1}{2}} \theta(t_{j-1})\|_{H^0}^2 +1+|C_L|   \Big]  + Z_L 	\nonumber \\
 &\qquad + 2 \max_{1\leq J\leq L} \sum_{j=1}^J 1_{\Omega_M(j-1)} \big[ T_{j,6}+\tilde{T}_{j,6}\big]  
 + 2 \max_{1\leq J\leq L} \sum_{j=1}^J 1_{\Omega_M(j-1)} \big[ T_{j,7}+\tilde{T}_{j,7}\big],
 \end{align} 
 where 
{\small \begin{align}	\label{def_ZL}
 Z_L = &\; C(\bar{C}_4, \delta_1, \nu, \kappa, \tilde{\delta}_1, C_L) \; h \|e_L\|_{V^0}^2 \Big( \sup_{s\in [0,T]} \|A^{\frac{1}{2}} u(s)\|_{V^0}^2 + \sup_{s\in [0,T]} \|\AA^{\frac{1}{2}} \theta(s)\|_{H^0}^2  +1\Big) \nonumber \\
 &+ C(\bar{C}_4, \kappa, \tilde{\delta}_1, C_L) \; h \| \ee_L\|_{H^0}^2 \Big( \sup_{s\in [0,T]} \|\AA^{\frac{1}{2}} \theta(s)\|_{H^0}^2+1\Big) \nonumber \\
 & + C(\bar{C}_4, \delta_2, \delta_3, \delta_4, \delta_5, \tilde{\delta}_2, \nu, \kappa, \lambda_1, C_L)
 \Big(  \sum_{j=1}^L \! \int_{t_{j-1}}^{t_j} \!\!\big[ \|  u(s)-u(t_{j-1})\big\|_{V^1}^2 + \|u(s)-u(t_{j})\|_{V^1}^2 \big] ds \Big) \nonumber \\
&\quad \times  \Big( \sup_{s\in [0,T]} \|A^{\frac{1}{2}} u(s)\|_{V^0}^2 + \sup_{s\in [0,T]} \| \AA^{\frac{1}{2}} \theta(s)\|_{H^0}^2 +\max_{1\leq j\leq L} \|u^j\|_{V^0}^2 +1\Big) \nonumber \\
 & + C(\bar{C}_4, \tilde{\delta}_1, \tilde{\delta}_3, \tilde{\delta}_4, \nu, \kappa, \lambda_1)  \Big(  \sum_{j=1}^L \int_{t_{j-1}}^{t_j}  \big[ \|\theta(s)-\theta(t_j)\|_{H^1}^2 + \|\theta(s)-\theta(t_{j-1})\|_{H^1}^2\big]  ds \Big) \nonumber \\
 &\quad \times \Big( \sup_{s\in [0,T]} \|A^{\frac{1}{2}} u(s)\|_{V^0}^2 +\sup_{s\in [0,T]} \|\AA^{\frac{1}{2}} \theta (s)\|_{H^0}^2 +\max_{1\leq j\leq L} \|u^j\|_{V^0}^2 + 
 \max_{1\leq j\leq L} \|\theta^j\|_{H^0}^2 +1 \Big)  .
 \end{align} }
 Using the upper estimates  \eqref{def_ZL}, taking expected values and using the Cauchy--Schwarz  inequality, as well as
  the upper estimates \eqref{mom_ut_L2}--\eqref{mom_t_H1}, \eqref{mom_increm_u_1}, \eqref{mom_increm_t_1},   
  and \eqref{mom_scheme_0}, we deduce 
\begin{align}		\label{mom_ZL}
 &\EE(Z_L) \leq  \, C\, h \Big\{ \EE\Big( \sup_{s\in [0,T]}\big[  \|u(s)\|_{V^0}^4 + \|\theta(s)\|_{H^0}^4 \big] \max_{0\leq j\leq N} \big[ \|u^j\|_{V^0}^4  + \|\theta^j\|_{H^0}^4 \big]\Big)   \Big\}^{\frac{1}{2}} \nonumber \\
 &\qquad \times 
 \Big\{ \EE\Big( \sup_{s\in [0,T]} \|A^{\frac{1}{2}} u(s)\|_{V^0}^4 + \sup_{s\in [0,T]} \| \AA^{\frac{1}{2}} \theta(s)\|_{H^0}^4 +1\Big) \Big\}^{\frac{1}{2}} \nonumber \\
 &\quad + C \Big\{ \EE\Big( \Big|   \sum_{j=1}^N \int_{t_{j-1}}^{t_j}  \big[ \|u(s)-u(t_j)\|_{V^1}^2 +  \|u(s)-u(t_{j-1})\|_{V^1}^2  \nonumber \\
 &\qquad \qquad \qquad \qquad  +  \|\theta(s)-\theta(t_j)\|_{H^1}^2  +  \|\theta(s)-\theta(t_{j-1})\|_{H^1}^2 \big] 
 ds \Big|^2 \Big) \Big\}^{\frac{1}{2}} \nonumber \\
 &\quad \quad \times \Big\{ \EE \Big( \sup_{s\in [0,T]} \big[ \|A^{\frac{1}{2}} u(s)\|_{V^0}^4 + \| \AA^{\frac{1}{2}} \theta(s)\|_{H^0}^4\big] +
 \max_{0\leq j\leq N}  \big[ \|u^j\|_{V^0}^4  + \|\theta^j\|_{H^0}^4 \big] +1\Big) \Big\}^{\frac{1}{2}} \nonumber \\
&\leq C h^\eta,
\end{align}
 for   $\eta \in (0,1)$ and  every $L=1, ..., N$, with some constant $C$ independent of $L$ and $N$. 

 The Cauchy--Schwarz and Young inequalities imply that
 \begin{align}		\label{maj_Tj6}
2 \max_{1\leq J\leq L} \sum_{j=1}^J &1_{\Omega_M(j-1)} |T_{j,6}|  \leq \frac{1}{3} \sum_{j=1}^L 1_{\Omega_M(j-1)}  \|e_j-e_{j-1}\|_{V^0}^2 \nonumber \\
 &+ 3 \sum_{j=1}^L 1_{\Omega_M(j-1)} \Big\|\int_{t_{j-1}}^{t_j} \big[ G(u(s))-G(u^{j-1})\big] dW(s)\Big\|_{V^0}^2, \\
 2 \max_{1\leq J\leq L} \sum_{j=1}^J  &1_{\Omega_M(j-1)} |\tilde{T}_{j,6}|  \leq \frac{1}{3} \sum_{j=1}^L 1_{\Omega_M(j-1)}  \|\ee_j-\ee_{j-1}\|_{H^0}^2 \nonumber \\
 &+  3 \sum_{j=1}^L 1_{\Omega_M(j-1)} \Big\|\int_{t_{j-1}}^{t_j} \big[ \GG(\theta(s))-\GG(\theta^{j-1})\big] d\WW(s)\Big\|_{H^0}^2. 
 	\label{maj_TTj6}
 \end{align}
 Furthermore, the Lipschitz  conditions \eqref{LipG},  \eqref{LipGG}, the inclusion $\Omega_M(j-1)\subset \Omega_M(j-2)$ for $j=2, ...,N$
 and the upper estimates \eqref{increm_u_L2}  imply that
\begin{align}		\label{mom_part_sumTj6}
 \EE\Big( & \sum_{j=1}^L 1_{\Omega_M(j-1)} \Big\|\int_{t_{j-1}}^{t_j} \big[ G(u(s))-G(u^{j-1})\big] dW(s)\Big\|_{V^0}^2 \big)	\nonumber \\
 &\;  \leq 
  \sum_{j=1}^L \EE\Big( \int_{t_{j-1}}^{t_j} 1_{\Omega_M(j-1)} L_1 \| u(s)-u^{j-1}\|_{V^0}^2  {\rm Tr}(Q) ds \Big)	\nonumber \\
  &\; \leq 2 L_1{\rm Tr}(Q)\, h\,  \sum_{j=2}^L \EE( 1_{\Omega_M(j-2)}\|e_{j-1} \|_{V^0}^2  \big)
  + C \sum_{j=1}^L \EE\Big( \int_{t_{j-1}}^{t_j} \|u(s)-u(t_{j-1})\|_{V^0}^2 ds\Big)\nonumber \\
  &\; \leq 2 L_1 {\rm Tr}(Q) \, h\,  \sum_{j=2}^L \EE( 1_{\Omega_M(j-2)}\|e_{j-1} \|_{V^0}^2  \big) + C h, 
  \end{align}
  and a similar computation using  \eqref{increm_t_L2} yields
  \begin{align} 
  \EE\Big(  \sum_{j=1}^L &1_{\Omega_M(j-1)} \Big\|\int_{t_{j-1}}^{t_j} \big[ \GG(\theta(s))-\GG(\theta^{j-1})\big] d\WW(s)\Big\|_{H^0}^2 \big)	\nonumber \\
 &\;  \leq 2 \tilde{L}_1 {\rm Tr}(\QQ)\, h\,  \sum_{j=2}^L \EE( 1_{\Omega_M(j-2)}\|\ee_{j-1} \|_{H^0}^2  \big) + C h.	\label{mom_part_sumTTj5}
 \end{align}%

 Finally, the Davis inequality, the inclusion $\Omega_M(J-1)\subset \Omega_M(j-1)$ for $j\leq J$, the local property of stochastic integrals,
 the Lipschitz condition \eqref{LipG}, the Cauchy--Schwarz and Young inequalities and the upper estimate \eqref{increm_u_L2} imply,
  for $\lambda >0$, that
 \begin{align} 	\label{mom_max_Tj7}
 \EE\Big( & \max_{1\leq J\leq L} 1_{\Omega_M(J-1)} \sum_{j=1}^J T_{j7} \Big) 		\nonumber \\
 &\; \leq 3 \sum_{j=1}^L \EE\Big( \Big\{ 1_{\Omega_M(j-1)} \int_{t_{j-1}}^{t_j} 
  \|G(u(s))-G(u^{j-1})\|_{{\mathcal L}(K;V^0)}^2  {\rm Tr}(Q) \|e_{j-1}\|_{V^0}^2 ds \Big\}^{\frac{1}{2}} \Big)		\nonumber \\
 &\; \leq 3 \,\sum_{j=1}^L \EE\Big[ \Big(  \max_{1\leq j\leq L} 1_{\Omega_M(j-1)} \|e_{j-1}\|_{V^0} \Big) \Big\{ \int_{t_{j-1}}^{t_j} L_1
  {\rm Tr}(Q)  \|u(s)-u^{j-1}\|_{V^0}^2  ds
 \Big\}^{\frac{1}{2}} \Big) 		\nonumber \\
 &\; \leq  \lambda\EE \Big(  \max_{1\leq j\leq L} 1_{\Omega_M(j-1)} \|e_{j-1}\|_{V^0}^2 \Big) + 
 C(\lambda) \EE\Big( \sum_{j=1}^L \int_{t_{j-1}}^{t_j} L_1 {\rm Tr}(Q)  \|u(s)-u^{j-1}\|_{V^0}^2  ds \Big) 		\nonumber \\
 &\; \leq  \lambda\EE \Big(  \max_{1\leq j\leq L} 1_{\Omega_M(j-2)} \|e_{j-1}\|_{V^0}^2 \Big) +  C(\lambda) h \sum_{j=1}^L \EE(\|e_{j-1}\|_{V^0}^2) 
 + C(\lambda)\, h.
 \end{align} 
A similar argument, using the Lipschitz condition \eqref{LipGG} and  \eqref{increm_t_L2},  yields  for $\lambda >0$,
 \vspace{-6pt}
\begin{align} 	\label{mom_max_TTj6}
 \EE\Big(  \max_{1\leq J\leq L} &1_{\Omega_M(J-1)} \sum_{j=1}^J \tilde{T}_{j7} \Big) 	\nonumber \\
 &\leq  \lambda\EE \Big(  \max_{1\leq j\leq L} 1_{\Omega_M(j-2)} \|\ee_{j-1}\|_{H^0}^2 \Big) + C(\lambda)  h \sum_{j=1}^L \EE(\|\tilde{e}_{j-1}\|_{V^0}^2)  
 + C h.
 \end{align}
Collecting the upper estimates \eqref{maj_Tj1}--\eqref{mom_max_TTj6},  we obtain  for $\eta\in (0,1)$,  $2\delta_1 + \sum_{i=2}^5 \delta_i<\frac{1}{3}$, \linebreak[4]
 $\sum_{i=1}^4 \tilde{\delta}_i<\frac{1}{3}$,
 and $\alpha,\lambda >0$, 
{\small \begin{align}		\label{mom_max_L2-1} 
 \EE\Big( &\max_{1\leq J\leq N} 1_{\Omega_M(j-1)} \big[ \|e_j\|_{V^0}^2 + \| \ee_j\|_{H^0}^2 \big] \Big) 	\nonumber \\
 &\quad + \EE\Big( \sum_{j=1}^N 1_{\Omega_M(j-1)}
 \Big[ \nu\Big( 2-6\Big[ 2\delta_1+ \sum_{i=2}^5 \delta_i \Big]\Big)  \|A^{\frac{1}{2}} e_j\|_{V^0}^2
  + \kappa \Big( 2-6\sum_{i=1}^4 \tilde{\delta}_i\Big)  \|\AA^{\frac{1}{2}} \ee_j\|_{V^0}^2 \Big] \Big)		\nonumber \\
  \leq & \, h\, \sum_{j=1}^{N-1} \EE\Big( 1_{\Omega_M(j-1)} \|e_j\|_{V^0}^2 
 \Big[ \frac{3(1+\alpha) \bar{C}_4^4 }{8 \nu} \Big( \frac{2}{\delta_1} + \frac{1}{\delta_5}\Big) M
  +C  \Big]  \Big)		\nonumber \\
  &\quad + h\, \sum_{j=1}^{N-1} \EE\Big( 1_{\Omega_M(j-1)} \|\ee_j\|_{H^0}^2 
  \Big[ \frac{3(1+\alpha) \bar{C}_4^4 }{8 \tilde{\delta_1} \kappa}  M  +C  \Big]  \Big)		\nonumber \\
&\quad   + 6 \lambda \EE\Big( \max_{1\leq j\leq N} 1_{\Omega_M(j-1)} \big[ \|e_{j-1}\|_{V^0}^2 + \| \ee_j\|_{H^0}^2 \big] \Big) + C h^\eta.
   \end{align} }%

 Therefore, given $\gamma \in (0,1)$, choosing $\lambda \in (0, \frac{1}{6})$  and $\alpha >0$ such that $\frac{1+\alpha}{1-6\lambda} < 1+\gamma$, 
 neglecting the sum in the left hand side and using the discrete Gronwall lemma, we deduce  that for $\eta \in (0,1)$
 \begin{align}		\label{mom_loc_maxut_L2}
  \EE\Big( &\max_{1\leq J\leq N} 1_{\Omega_M(j-1)} \big[ \|e_j\|_{V^0}^2 + \| \ee_j\|_{H^0}^2 \big] \Big)  \leq C(1+M) e^{T {\mathcal C}(M)} h^\eta ,
 \end{align} 
 where 
 \[ {\mathcal C}(M):= \frac{3(1+\gamma) \bar{C}_4^4}{8}  \max\Big(\frac{2}{\delta_1 \nu} + \frac{1}{\delta_5 \nu } \, , 
 \, \frac{1}{\tilde{\delta}_1 \kappa}\Big) \, M,  \]
  for $2 \delta_1+\delta_5 <\frac{1}{3}$ and $\tilde{\delta}_1<\frac{1}{3}$ (and choosing $\delta_i, i=3,3, 4$ and $\tilde{\delta}_i, i=2,3,4$ such
 that $2\delta_1+ \sum_{i=2}^5 \delta_i < \frac{1}{3}$ and $\sum_{i=1}^4 \tilde{\delta}_i < \frac{1}{3}$).  
 Let $\delta_5<\frac{1}{15}$ and $\delta_1=2\delta_5$. 
 Then, for some  $\gamma >0$, we have that
 \[  {\mathcal C}(M) =  \frac{9(1+\gamma) \bar{C}_4^4}{8}  \max\Big(\frac{5}{ \nu}  \, , 
 \, \frac{1}{\kappa}\Big) \, M.  \]
 Plugging the upper estimate \eqref{mom_loc_maxut_L2} in \eqref{mom_max_L2-1}, we conclude the proof of \eqref{loc_cv}. 
\end{proof}

\section{Rate of Convergence in Probability and in $L^2(\Omega)$}		\label{sec_rate_proba} 
In this section, we deduce from Proposition \ref{th_loc_cv} the convergence in probability of the implicit time Euler scheme with the ``optimal'' 
rate of convergence of
``almost 1/2'' and a logarithmic speed of convergence in $L^2(\Omega)$. The presence of the bilinear term in the It\^o formula for
 $\|\AA^{\frac{1}{2}} \theta(t)\|_{H^0}^2$  does not enable us to prove exponential moments for this norm, which prevents us from using
 the general framework presented in \cite{BeMi_FEM}  to prove a polynomial rate  for the strong convergence.
 \subsection{Rate of Convergence in Probability}
 In this section, we deduce the rate of the convergence in probability (defined in \cite{Pri}) from Propositions \ref{prop_u_V1},
  \ref{prop_mom_t_H1}, \ref{prop_mom_scheme}
  and \ref{th_loc_cv}.  
  The short proof is similar to that on \cite[Theorem~xx]{BeMi_Bou}; it is included for the sake of completeness.
 \begin{theorem} 		\label{th_cv_proba}
Suppose that  the  conditions {\bf (C-u)} and {\bf (C-$\theta$)} hold. 
Let  $u_0\in L^{16+\epsilon}(\Omega ; V^1)$ and $\theta_0\in L^{16+\epsilon}(\Omega;H^1)$ for some $\epsilon>0$,
  $u,\theta$ be the solution to \eqref{def_u} and \eqref{def_t} and
$\{u^j, \theta^j\}_{j=0, ..., N}$ be the solution to \eqref{Euler_u} and \eqref{Euler_t}. 
 Then, for every $\eta \in (0,1)$, we have
\begin{equation}		\label{speed_proba} 
\lim_{N\to \infty} P\Big( \max_{1\leq J\leq N} \big[ \|e_J\|_{V^0}^2 + \|\ee_J\|_{H^0}^2 \big] + \frac{T}{N} \sum_{j=1}^N \big[ \|A^{\frac{1}{2}} e_j\|_{V^0}^2
+ \| \AA^{\frac{1}{2}} \ee_j \|_{H^0}^2 \big] \geq  N^{-\eta} \Big) = 0 .
\end{equation} 
\end{theorem} 
 
 \begin{proof} 
 For $N\geq 1$ and $\eta\in (0,1)$, let 
\[ A(N,\eta):=\Big\{ \max_{1\leq J\leq N} \big[ \|e_J\|_{V^0}^2 + \|\ee_J\|_{H^0}^2 \big] + \frac{T}{N} \sum_{j=1}^N \big[ \|A^{\frac{1}{2}} e_j\|_{V^0}^2
+ \| \AA^{\frac{1}{2}} \ee_j \|_{H^0}^2 \big] \geq  N^{-\eta}\Big\}.
\]
Let $\tilde{\eta}\in  (\eta,1)$, $M(N)=\ln(\ln N)$ for $N\geq 3$.  
Then, 
\[ P\big(A(N,\eta)\big) \leq   P\big( A(N,\eta) \cap \Omega_{M(N)} \big) + P \big( (\Omega_{M(N)})^c\big),\]
where $\Omega_{M(N)}= \Omega_{M(N)}(N)$ is defined in Proposition \ref{th_loc_cv}. 
The inequality \eqref{loc_cv}  implies that
\begin{align*}
 P\big( &A(N,\eta) \cap  \Omega_{M(N)} \big) \\ 
 &    \leq N^{\eta}  \;  \EE\Big( 1_{\Omega_{M(N)}}
 \Big[   \max_{1\leq J\leq N} \big[ \|e_J\|_{V^0}^2 + \|\ee_J\|_{H^0}^2 \big] + \frac{T}{N} \sum_{j=1}^N \big[ \|A^{\frac{1}{2}} e_j\|_{V^0}^2
+ \| \AA^{\frac{1}{2}} \ee_j \|_{H^0}^2 \big]  \Big]\Big) \\
&\leq N^{\eta} \; C \big[1+\ln(\ln N) \big] e^{T \tilde{C}  \ln(\ln N)} \Big( \frac{T}{N}\Big)^{\tilde{\eta}  } \\
&\leq C \big[1+\ln(\ln N) \big] \big(  \ln N\big)^{\tilde{C} T} N^{-\tilde{\eta}+\eta} \to 0 \quad {\rm as}\; N\to \infty.
\end{align*} 
The inequalities   \eqref{mom_ut_L2}--\eqref{mom_t_H1}  imply that
 \begin{align*}
P\big(( \Omega_{M(N)})^c\big) \leq \frac{1}{M(N)} \EE \Big( \sup_{t\in [0,T]} \|u(t)\|_{V^1}^2 + \sup_{t\in [0,T]} \|\theta(t)\|_{H^1}^2 \Big) 
\to 0 \quad {\rm as}\; N\to \infty.
\end{align*}
The two above convergence results  complete the proof of \eqref{speed_proba}. 
\end{proof}

\subsection{Rate of Convergence in $L^2(\Omega)$} 
We finally prove the strong rate of convergence, which is also a consequence of   Propositions \ref{prop_u_V1}, \ref{prop_mom_t_H1},
\ref{prop_mom_scheme}  and \ref{th_loc_cv}. 	
\begin{theorem}		\label{th_strong_rate}
Suppose that  the  conditions {\bf (C-u)} and {\bf (C-$\theta$)(i)}  hold.  
Let $u_0\in L^{2^q+\epsilon}(\Omega;V^1)$ and $\theta_0\in L^{2^q+\epsilon}(\Omega;H^1)$ for $q\in [4,\infty)$ and some $\epsilon >0$. 
Then, for  some constant $C$ and for $N$ large enough  
\begin{align}	\label{strong_rate}
\EE\Big( \max_{1\leq J\leq N}& \big[ \|e_J\|_{V^0}^2 + \|\ee_J\|_{H^0}^2 \big] + \frac{T}{N} \sum_{j=1}^N \big[ \|A^{\frac{1}{2}} e_j\|_{V^0}^2
+ \| \AA^{\frac{1}{2}} \ee_j\|_{H^0}^2
\Big) 		
 \leq C  \big( \ln (N) \big) ^{ - (2^{q-1}+1)} .
\end{align}
\end{theorem}

\begin{proof}
For any integer $N\geq 1$ and $M\in [1,+\infty)$,  
 let $\Omega_M=\Omega_M(N)$ be defined by \eqref{def-A(M)}.
Let $p$ be the conjugate exponent of $2^q$. H\"older's inequality implies that
\begin{align}	\label{strong-1}
 \EE\Big( & 
 1_{(\Omega_M)^c} \max_{1\leq J\leq N} \big[ \|e_J\|_{V^0}^2 + \|\ee_J\|_{H^0}^2 \big] \Big)  
 \leq  \Big\{ P\big(
  (\Omega_M)^c \big) \Big\}^{\frac{1}{p}}     \nonumber \\
 &\quad \times \Big\{ \EE\Big( \sup_{s\in [0,T]} \|u(s)\|_{V^0}^{2^q} +  \sup_{s\in [0,T]} \|\theta(s)\|_{H^0}^{2^q} 		
 + \max_{1\leq j\leq N} \|u^j\|_{V^0}^{2^q} +  \max_{1\leq j\leq N} \|\theta^j\|_{H^0}^{2^q} \Big) \Big\}^{2^{-q}}		\nonumber \\
 \leq &\;  C \Big\{ P\big( (\Omega_M)^c \big)  \Big\}^{\frac{1}{p}} ,
\end{align}
where the last inequality is a consequence of \eqref{mom_ut_L2} and \eqref{mom_scheme_0}.

Using  \eqref{mom_u_V1}	
and \eqref{mom_t_H1},  we deduce that
 \begin{equation}		\label{strong-2}
P\big( (\Omega_M)^c\big) \leq M^{-2^{q-1}} \EE\big( \sup_{s\in [0,T]} \|u(s)\|_{V^1}^{2^q} +\sup_{s\in [0,T]} \|\theta(s)\|_{H^1}^{2^q} \Big) = C  M^{-2^{q-1}}.
\end{equation}
Using \eqref{loc_cv}, we choose $M(N)\to \infty$ as $N\to \infty$ such that, for $\eta \in (0,1)$ and  $\gamma>0$,
\[ N^{-\eta} \exp\Big[ \frac{9(1+\gamma) \bar{C}_4^4 T }{8}  \Big( \frac{5}{\nu} \vee \frac{1}{\kappa}\Big) M(N) \Big] M(N) \asymp M(N)^{-2^{q-1}}
\]  
which,  taking logarithms, yields
\[ -\eta \ln(N) + \frac{9(1+\gamma) \bar{C}_4^2 T}{8}  \big( \frac{5}{\nu} \vee \frac{1}{\kappa}\big) M(N)  \asymp   -  2^{q-1}   \ln(M(N)).
\]
Set
\begin{align*} M(N)= &\,  \frac{8}{9(1+\gamma) \bar{C}_4^4 \big( \frac{5}{\nu} \vee \frac{1}{\kappa}\big) T} 
\big[ \eta \ln(N) -  \big( 2^{q-1} +1\big)
\ln\big( \ln(N)\big) \big] \\
\asymp &\, \frac{8}{9(1+\gamma) \bar{C}_4^4 \big( \frac{5}{\nu} \vee \frac{1}{\kappa}\big) T} \eta \ln(N). 
\end{align*} 
Then,
\begin{align*}
-\eta \ln(N) + \frac{9(1+\gamma) \bar{C}_4^4 T}{8}  \big( \frac{5}{\nu} \vee \frac{1}{\kappa}\big)  M(N) + \ln (M(N)) &\asymp  - \big( 2^{q-1} +1\big) 
\ln\big( \ln(N)\big)
+ 0(1), \\
- \big( 2^{q-1} +1\big)  \ln\big( M(N)\big)& \asymp - 
\big( 2^{q-1} +1\big) \ln(N) + 0(1).
\end{align*} 
This implies that
\[ \EE\Big( \max_{1\leq J\leq N} \big[ \|e_J\|_{V^0}^2 + \|\ee_J\|_{H^0}^2 \big] \Big) \leq C \big( \ln(N)\big)^{-\big( 2^{q-1} +1)}. 
\]
The inequalities \eqref{mom_u_V1} and \eqref{mom_t_H1} for $p=1$ and \eqref{mom_scheme_1} for $K=1$ imply 
 \[ \sup_{N\geq 1} \EE \Big( \frac{T}{N} \sum_{j=1}^N \big[ \| A^{\frac{1}{2}} u(t_j) \|_{V^0}^2 + \|A^{\frac{1}{2}} u^j\|_{V^0}^2 
 + \| \AA^{\frac{1}{2}} \theta(t_j)\|_{H^0}^2  + \| \AA^{\frac{1}{2}} \theta^j\|_{H^0}^2 \big] \Big) <\infty.
 \]
 Using   a similar argument, we obtain 		
 \[\EE\Big( \frac{T}{N} \sum_{j=1}^N \big[ \|A^{\frac{1}{2}} e_j\|_{V^0}^2 + \|\AA^{\frac{1}{2}} \ee_j\|_{H^0}^2 \big] \Big) \leq 
  \ C \big( \ln(N)\big)^{ -( 2^{q-1} +1)}
  .\] 
  This yields \eqref{strong_rate} and completes the proof. 
\end{proof}

\section{The additive stochastic perturbation - Exponential moments} \label{additive_noise}

We  assume that  $K=V^1$, $\tilde{K} = H^1$, and that $G(u)=Id_{V^1}$, $\tilde{G}(\theta) = Id_{H^1}$ for any $u_0\in V^0$ and $\theta \in H^0$. 
Thus  $W$ is a $Q$-Wiener process in $V^1$, that is solenoidal, 
 (resp. $\tilde{W}$ is a $\tilde{Q}$-Wiener process in $H^1$), where $Q$ 
(resp. $\tilde{Q}$) is a symmetric bounded operator in $V^1$ (resp. $H^1$) with trace-class denoted by ${\rm Tr}\, (Q)$ and ${\rm Tr}\, (\tilde{Q})$ respectively. 
  In particular, we choose 
 \begin{equation}\label{def_W}
 W(t)=\sum_{j\geq 1} \sqrt{q_j}\,  \beta_j(t)\, \zeta_j,\;  \; \tilde{W}(t)= \sum_{j\geq 1} \sqrt{\tilde{q}_j}\,  {\tilde{\beta}}_j(t)\, {\tilde{\zeta}}_j,\;  t\geq 0, 
 \end{equation} 
 where $\{q_j\}_{j\geq 1}$ is a sequence of positive numbers such that $Q \zeta_j = q_j \zeta_j$, $j=1, 2, ...$  
 and $\{\beta_j(t)\}_{j\geq 1}$ are independent one-dimensional Brownian motions on   $(\Omega, {\mathcal F},$ $  ({\mathcal F}_t)_t,  \PP)$, 
 while $\{ \tilde{q}_j\}_{j\geq 1}$ is a sequence of positive numbers such that $\tilde{Q} \tilde{\zeta}_j = \tilde{q}_j \tilde{\zeta}_j$, $j=1, 2, ...$  
 and $\{ \tilde{\beta}_j(t)\}_{j\geq 1}$ are independent one-dimensional Brownian motions (independent of $\{\beta_j(t)\}_{j\geq 1}$)
 on   $(\Omega, {\mathcal F},$ $  ({\mathcal F}_t)_t,  \PP)$. For details concerning this Wiener process  we refer  to \cite{DaPZab}.

Then ${\rm Tr}\, (Q)=\sum_{j\geq 1} q_j$ and  ${\rm Tr}\, (\tilde{Q})=\sum_{j\geq 1} \tilde{q}_j$
We will assume throughout this section the following stronger condition 
 \begin{equation}		\label{K0}
K_0:=  {\rm Tr}\, \big(A^{\frac{1}{2}} Q A^{\frac{1}{2}}\big) = \sum_{j\geq 1} \lambda_j\, q_j <\infty \; \mbox{\rm and} \; 
\tilde{K}_0:=  {\rm Tr}\, \big( (\tilde{A})^{\frac{1}{2}} \tilde{Q}  (\tilde{A})^{\frac{1}{2}}\big) = \sum_{j\geq 1} \tilde{\lambda}_j\, \tilde{q}_j <\infty.
 \end{equation}  
 We consider the evolution equations  similar to \eqref{velocity-mul}-\eqref{temperature-mul} for an additive noise (see e.g. \cite{Temam-88}  and \cite{Fer}):
\begin{align} \label{velocity1}
 \partial_t u - \nu \Delta u + (u\cdot \nabla) u + \nabla \pi & = ({\mathcal T} - {\mathcal T}_L)   v_2 +   dW\quad \mbox{\rm in } \quad (0,T)\times D,\\
 \partial_t {\mathcal T} - \kappa \Delta {\mathcal T} + (u\cdot \nabla {\mathcal T}) &=   d\tilde{W} \quad \mbox{\rm in } \quad (0,T) \times D, 
 \label{temperature1}\\
 \mbox{\rm div }u&=0 \quad \mbox{\rm in } \quad [0,T]\times D,  \nonumber 
 \end{align}
 where  $T>0$. The processes 
$u: \Omega\times (0,T)\times D  \to \RR^2$, and ${\mathcal T} : \Omega\times (0,T)\times D  \to \RR$ 
 have  initial conditions $u_0$ and $ {\mathcal T}_0$ in $D$ respectively, and $\pi : \Omega\times (0,T)\times D  \to \RR$  is  the  pressure. 
On $[0,T]\times \partial D$ these fields satisfy 
  the  boundary conditions  defined in \eqref{periodicity}.
   We then introduce the same classical change of processes \eqref{def_theta} as in the multiplicative case to describe  the temperature and pressure. 
   Once more $\theta$ and $\hat{\pi}$ have the same periodic properties as ${\mathcal T}$ and $\pi$,   $\theta(t,.)=0$ on $\{x_2 =0\}\cup \{x_2=L\}$
 the pair $(u,\theta)$ satisfies the more symmetric  equations 
\begin{align*} %
 \partial_t u - \nu \Delta u + (u\cdot \nabla) u + \nabla \hat{\pi} & = \theta   v_2 +   dW\quad \mbox{\rm in } \quad (0,T)\times D,\\
 \partial_t \theta- \kappa \Delta \theta + (u\cdot \nabla \theta)  +  C_L u_2&=   d\tilde{W} \quad \mbox{\rm in } \quad (0,T) \times D, 
 \\ 
 \mbox{\rm div }u&=0 \quad \mbox{\rm in } \quad [0,T]\times D, 
 \end{align*}
and projecting on divergence free fields
\begin{align}
 \partial_t u +\nu A u + (u\cdot \nabla) u 
 & = \Pi (\theta   v_2 )+   dW\quad \mbox{\rm in } \quad (0,T)\times D,
 \label{velocity-add}\\
 \partial_t \theta- \kappa \Delta \theta + (u\cdot \nabla \theta)  +  C_L u_2&=   d\tilde{W} \quad \mbox{\rm in } \quad (0,T) \times D, 
 \label{temperature-add}
\end{align}
The assumptions on the Brownian motions $W$ and $\tilde{W}$ imply that the conclusions of Theorem~\ref{th-gwp} and  Propositions~\ref{prop_u_V1}--\ref{prop_mom_t_H1}
 obviously apply to the solution $(u,\theta)$ to \eqref{velocity-add}--\eqref{temperature-add}. 

The semi-implicit time Euler scheme $\{ u^k ; k=0, 1, ...,N\}$  and $\{ \theta^k ; k=0, 1, ...,N\}$ in the additive case is defined by $u^0=u_0$, $\theta^0=\theta_0$,
 and for $\varphi \in V^1$, $\psi\in H^1$ and $l=1, ...,N$, 
\begin{align}	\label{Euler_u_add}
\Big( u^l-u^{l-1} &+ h \nu A u^l + h B\big( u^{l-1},u^l\big)   , \varphi\Big) = \big( \Pi(\theta^{l-1} v_2), \varphi) h  + \big( \Delta_l W\, , \, \varphi\big),  \\
\Big( \theta^l-\theta^{l-1} &+ h \kappa \AA \theta^l +h [ u^{l-1}.\nabla]\theta^l + h C_L u^{l-1}_2 , \psi\Big) = \big( \Delta_l\WW\, , \, \psi\big). \label{Euler_t_add}
\end{align}
The random variables $u^k$ and $\theta^k$ are ${\mathcal F}_{t_k}$-measurable for every $k=0, ..., N$ and 
the results of Proposition~\ref{prop_mom_scheme} are satisfied in the additive case for the solutions of \eqref{Euler_u_add} and \eqref{Euler_t_add}.

We conclude this section with exponential moments for the Euler scheme of the temperature  and the velocity, and  use ideas from \cite{BeMi_additive}. 
Note that we have to deal at the same time with the velocity and the temperature,  and that both evolution equations have a linear term involving the other quantity.
This will impose that  the "average" difference of temperatures $C_L$ between the top and bottom parts on the torus is small with respect  to the viscosity $\nu$ and thermal diffusivity $\kappa$.

 Suppose that for some positive constants $\gamma_0, \tilde{\gamma}_0$ we have 
\begin{equation}		\label{exp_mom_ut0L2}
\EE\big[ \exp( {\gamma}_0 \|u_0\|_{V^0}^2 ) \big] <\infty\quad \mbox{\rm and}\quad \EE\big[ \exp( \tilde{\gamma}_0 \|\theta_0\|_{H^0}^2 ) \big] <\infty.
\end{equation} 
Recall that the constants $\lambda_1$ and $\tilde{\lambda}_1$ are defined in the Poincar\'e inequalities \eqref{lambda_1} and \eqref{tilde_lambda_1} respectively. 
\begin{prop}		\label{exp_mom_scheme_ut}
Suppose that  $|C_L|<\frac{\nu \kappa \lambda_1 \tilde{\lambda}_1}{4}$. 

 (i) Let $\theta_0\in H^0$,  $u_0\in V^0$ be deterministic.  
  Set  
  \begin{equation} 	\label{def_beta0}
   \beta_0:= 
  \frac{\nu \kappa \lambda_1 \tilde{\lambda}_1 - 4 |C_L] }{4\big[ \tilde{\lambda}_1 \kappa |C_L| {\rm Tr(Q)} + \lambda_1 \nu {\rm Tr}(\QQ)\big]}.
  \end{equation} 
 Then given any  $\beta\in (0,\beta_0)$ we have for $N$ large enough, 
\begin{equation} 		\label{exp_mom_ut_reg}
 \EE\Big[ \exp  \Big( \beta   \max_{1\leq n\leq N}  \Big[ |C_L| \|u^n\|_{V^0}^2 + \| \theta^n\|_{H^0}^2 +\frac{T}{N} \sum_{l=1}^n \big[  |C_L| \nu \|A^{\frac{1}{2}} u^l \|_{V^0}^2 
+ \kappa \|\AA^{\frac{1}{2}} \theta^l \|_{H^0}^2\big]  \Big)\Big] <\infty.
\end{equation} 

 (ii)  Let $u_0$ and $\theta_0$ be random, ${\mathcal F}_0$-measurable  and satisfy  \eqref{exp_mom_ut0L2}.  Then for $\beta \in (0, \beta_1)$, where 
 \begin{equation} \label{beta_1} 
 \beta_1:= \frac{\nu \kappa\lambda_1 \tilde{\lambda}_1 - 4|C_L|}{4\big[ \tilde{\lambda}_1 \kappa |C_L| {\rm Tr}(Q)  + \lambda_1 \nu {\rm Tr}(\QQ)] 
+ (\nu \kappa \lambda_1 \tilde{\lambda}_1 -4 |C_L|) \frac{\gamma_0 + |C_L| \tilde{\gamma}_0}{\gamma_0 \tilde{\gamma_0}}}, 
 \end{equation} 
the upper bound \eqref{exp_mom_ut_reg} holds for $N$ large enough. 
\end{prop} 
\begin{proof} 
For $l = 1, ..., N$, we have $u^l\in V^1$ and $\theta^l \in H^1$. In \eqref{Euler_u_add} (resp. \eqref{Euler_t_add}) take the inner product with $\varphi = u^l$  (resp. with $\psi=\theta^l$).
Using the identity $(a-b,a)=\frac{1}{2}\big[ \|a\|^2 - \|b\|^2 + \| a-b\|^2\big]$, this yields
\begin{align}	\label{diff_u^l_add}
\|u^l\|_{V^0}^2- \| u^{l-1}\|_{V^0}^2 + \|u^l-u^{l-1}\|_{V^0}^2 + 2 h\nu \|A^{\frac{1}{2}} u^l\|_{V^0}^2 = 2h\big( \Pi( \theta^{l-1} v_2), u^l\big) + 2\big(\Delta_lW, u^l\big),\\
 \|\theta^l\|_{H^0}^2- \| \theta^{l-1}\|_{H^0}^2 + \| \theta^l-\theta^{l-1}\|_{H^0}^2 + 2 h\kappa \|\AA^{\frac{1}{2}} \theta^l\|_{H^0}^2 +2C_L h \big( u^{l-1}_2, \theta^l\big)
= 2 \big( \Delta_l \tilde{W} , \theta^l).		\label{diff_theta^l_add} 
\end{align}
The Cauchy-Schwarz and Young inequalities imply for $\epsilon >0$ 
\begin{align*}  2h\big|\big( \Pi( \theta^{l-1} v_2) , u^l\big)\big| &= 2h|\big(  \theta^{l-1} v_2, u^l\big)\big| \leq \frac{2}{\sqrt{\lambda_1}} h \|\theta^{l-1}\|_{H^0} \|A^{\frac{1}{2}} u^l\|_{V^0} \\
& \leq \epsilon \nu h \|A^{\frac{1}{2}} u^l\|_{V^0}^2 + \frac{h}{\epsilon \nu \lambda_1} \|\theta^{l-1}\|_{H^0}^2, 
\end{align*}
while
\[ 2  \big( \Delta_l W, u^l \big)  \leq 2 \big( \Delta_l W, u^{l-1} \big)  + \|u^l - u^{l-1} \|_{V^0}^2 + \| \Delta_l W\|_{V^0}^2.\]
The above inequalities and \eqref{diff_u^l_add} imply for $\epsilon \in (0,1)$ and $l=1, ..., N$
\begin{align}		\label{diff-ul_add-2}
\|u^l\|_{V^0}^2 &- \|u^{l-1}\|_{V^0}^2  + h \nu \|A^{\frac{1}{2}} u^l\|_{V^0}^2  \nonumber \\
& \leq
 2\big( \Delta_l  W, u^{l-1}\big) + \|\Delta_lW\|_{V^0}^2 + h \frac{ \|\theta^{l-1}\|_{H^0}^2 }{\epsilon \nu \lambda_1} - h\nu (1-\epsilon) \|A^{\frac{1}{2}} u^l\|_{V^0}^2.
\end{align} 
A similar computation yields for $\tilde{\epsilon}\in (0,1)$ and $l=1, ..., N$
\begin{align}		\label{diff-tl_add-2}
\|\theta^l\|_{H^0}^2 &- \|\theta^{l-1}\|_{H^0}^2  +  h\kappa  \|\AA^{\frac{1}{2}} \theta^l\|_{H^0}^2 \nonumber \\
& \leq 2\big( \Delta_l \WW, \theta^{l-1}\big) + \|\Delta_l \tilde{W}\|_{H^0}^2 + \frac{h |C_L|^2}{\tilde{\epsilon} \kappa \tilde{\lambda}_1}  \|u^{l-1}\|_{V^0}^2 
 -h\kappa (1-\tilde{\epsilon}) \|\AA^{\frac{1}{2}} \theta^l\|_{V^0}^2.
\end{align} 
Multiply \eqref{diff-ul_add-2} by $\alpha>0$ and \eqref{diff-tl_add-2} by $\beta>0$; add the corresponding inequalities for $l=1, ..., n$ for any $n=1, ..., N$;  this yields
\begin{align*}
\alpha &\Big( \|u^n\|_{V^0}^2 + h \nu \sum_{l=1}^n \|A^{\frac{1}{2}} u^l\|_{V^0}^2 \Big) +  \beta \Big(  \|\theta^n\|_{H^0}^2  + h \kappa \sum_{l=1}^n \|\AA^{\frac{1}{2}} \theta^l\|_{H^0}^2 \Big) 
 \leq   \alpha \|u_0\|_{V^0}^2 + \beta \|\theta\|_{H^0}^2\\
 &  + 2\alpha \sum_{l=1}^n \big( u^{l-1}, \Delta_l W\big) 
+ 2\beta \sum_{l=1}^n \big( \theta^{l-1} , \Delta_l \WW\big) 
+  \sum_{l=1}^n \big[ \alpha \|\Delta_l W\|_{V^0}^2 + \beta \|\Delta_l \WW\|_{H^0}^2 \big]  \\
 & + \frac{\alpha h}{\epsilon \nu \lambda_1} \sum_{l=1}^n \|\theta^{l-1}\|_{H^0}^2  + \frac{\beta |C_L|^2  h}{\tilde{\epsilon} \kappa \tilde{\lambda}_1} \sum_{l=1}^n \|u^{l-1}\|_{V^0}^2 \\
 & - \alpha \nu (1-\epsilon) h \sum_{l=1}^n \|A^{\frac{1}{2}} u^l\|_{V^0}^2 - \beta \kappa (1-\tilde{\epsilon}) h \sum_{l=1}^n \|\AA^{\frac{1}{2}} \theta^l\|_{H^0}^2.
\end{align*}
The Poincar\'e inequality implies
\begin{align}		\label{sum_errors}
\alpha &\Big( \|u^n\|_{V^0}^2 + h \nu \sum_{l=1}^n \|A^{\frac{1}{2}} u^l\|_{V^0}^2 \Big) +  \beta \Big(  \|\theta^n\|_{H^0}^2  + h \kappa \sum_{l=1}^n \|\AA^{\frac{1}{2}} \theta^l\|_{H^0}^2 \Big) \nonumber \\
   \leq  &  \Big[ \alpha + \frac{\beta |C_L|^2 h}{\tilde{\epsilon}\kappa \tilde{\lambda}_1} \Big] \|u_0\|_{V^0}^2 
+\Big[ \ \beta + \frac{\alpha h}{\epsilon \nu \lambda_1} \Big] \|\theta_0\|_{H^0}^2 + 2\alpha\! \sum_{l=1}^n\!  \big( u^{l-1}, \Delta_l W\big) 
+ 2\beta \! \sum_{l=1}^n\!  \big( \theta^{l-1} , \Delta_l \WW\big)  \nonumber \\
&
+  \sum_{l=1}^{n} \big[ 2\alpha \|\Delta_l W\|_{V^0}^2 \nonumber \\
&+ 2\beta \|\Delta_l \WW\|_{H^0}^2 \big]  -h\alpha \nu (1-\epsilon) \|A^{\frac{1}{2}} u^n\|_{V^0}^2 
+ h\sum_{j=1}^{n-1} \|A^{\frac{1}{2}} u^l\|_{V^0}^2 \Big[ - \alpha \nu (1-\epsilon) + \frac{\beta |C_L|^2}{\tilde{\epsilon} \kappa \lambda_1 \tilde{\lambda}_1} \Big] \nonumber \\
&-h\beta \kappa (1-\tilde{\epsilon}) \|\AA^{\frac{1}{2}} \theta^n\|_{V^0}^2 
+ h\sum_{j=1}^{n-1} \! \|\AA^{\frac{1}{2}} \theta^l\|_{V^0}^2 \Big[ - \beta \kappa (1-\tilde{\epsilon}) + \frac{\alpha}{\epsilon \nu \lambda_1 \tilde{\lambda}_1} \Big].
 \end{align}
Let $M_n=2\alpha \sum_{l=1}^n \big( u^{l-1}, \Delta_l W\big)$ and  $\tilde{M}_n=2\beta \sum_{l=1}^n \big( \theta^{l-1}, \Delta_l \WW\big)$; the processes $(M_n, {\mathcal F}_{t_n}, n=1, ..., N)$
and $(\tilde{M}_n, {\mathcal F}_{t_n}, n=1, ..., N)$ are discrete martingales. Since $W$ and $\WW$ are independent, $\EE\big( M_n \tilde{M}_n\big)=0$  for every $n=1, ..., N$.
 
 For $s\in [t_l, t_{l+1})$, $l=1, \cdots, N-1$, 
 set $\underline{s}=t_l$, $u^{\underline{s}}=
u^l$. and  $\theta^{\underline{s}}=
\theta^l$.  With these notations, $M_n = {\mathcal M}_{t_n}$,  $\tilde{M}_n=\tilde{\mathcal M}_{t_n}$, where 
\[ {\mathcal M}_t = 2\alpha  \int_0^t \big(    u^{\underline{s}} \ , \,  d  W(s)\big),  \quad \mbox{\rm and} \quad 
 \tilde{\mathcal M}_t=  2\beta \int_0^t \big(    \theta^{\underline{s}} \ , \,  d  \WW(s)\big), 
\quad t\in [0,T]. 
\]
The processes $({\mathcal M}_t, {\mathcal F}_t, t\in [0,T])$ and  $(\tilde{\mathcal M}_t, {\mathcal F}_t, t\in [0,T])$ are square integrable martingales such that
\begin{align}	\label{quadvar_Mn}
\langle {\mathcal M}\rangle_{t_n}& \leq 4\alpha^2 \int_0^{t_n} {\rm Tr}(Q) \|u^{\underline{s}}\|_{V^0}^2 ds = 4\alpha^2  {\rm Tr}(Q)  h \sum_{l=0}^{n-1} \|u^l\|_{V^0}^2,  \\
\langle\tilde {\mathcal M}\rangle_{t_n}& \leq 4\beta^2 \int_0^{t_n} {\rm Tr}(\QQ) \|\theta^{\underline{s}}\|_{H^0}^2 ds = 4\beta^2  {\rm Tr}(\QQ)  h \sum_{l=0}^{n-1} \|\theta^l\|_{H^0}^2 .
\label{quadvar_tildeMn}
\end{align}
Furthermore, if $Y$ (resp. $\tilde{Y}$) is a $V^0$ (resp. $H^0$)-valued centered Gaussian random variable with the same distribution as $W(1)$ (resp. $\WW(1)$), then
the covariance operator of $Y$ is $Q$ (resp. of $\tilde{Y}$ is $\QQ$). Using the scaling property and the independence of the increments $\Delta_l$, $\Delta_l\WW$, $l=1, ..., N$,
 we deduce that for any $\alpha, \beta >0$ we have
\[ \EE\Big[ \exp\Big( \alpha \sum_{l=1}^N \|\Delta_l W\|_{V^0}^2 + \beta \|\Delta_l W\|_{H^0}^2 \Big) \Big] = \Big( \EE\Big[ e^{\frac{\alpha T}{N} \|Y\|_{V^0}^2}\Big] \Big)^N
\Big( \EE\Big[ e^{\frac{\beta T}{N} \|\tilde{Y}\|_{H^0}^2}\Big] \Big)^N
\]
Proposition~\cite[Proposition~2.16]{DaPZab} implies that for $\gamma \in \big( 0, \frac{1}{4 {\rm Tr}(Q) }\big)$ and  $\bar{\gamma} \in \big[ \gamma, \frac{1}{4 {\rm Tr}(Q) }\big)$, we have 
\begin{align*} 
 \EE\big( e^{\gamma \|  {Y}\|_{V^0}^2}\big)& \leq \exp\Big( \frac{1}{2} \sum_{i=1}^\infty \frac{( 2\bar{\gamma})^i}{i} ({\rm Tr}(Q))^i \Big)  
 \leq 
\exp\Big[ - \frac{1}{2} \ln\big(1- 2 \bar{\gamma} {\rm Tr}({Q}) \big) \Big]<\infty. 
\end{align*} 
Since $-\ln(1-x) \leq 2x$ for $x\in (0,\frac{1}{2})$. we deduce 
\[ \EE\Big( e^{\gamma \|Y\|_{V^0}^2}\Big) \leq e^{\bar{\gamma} {\rm Tr}(Q)} \quad \mbox{\rm for}\quad 0<\gamma <\bar{\gamma} < \frac{1}{4{\rm Tr}(Q)}.\]
 A similar computation yields 
 \[ \EE\Big( e^{\delta \|\tilde{Y}\|_{H^0}^2}\Big) \leq e^{\bar{\delta} {\rm Tr}(\QQ)} \quad \mbox{\rm for}\quad 0<\delta <\bar{\delta} < \frac{1}{4{\rm Tr}(\QQ)}.\]
Hence  given any $\alpha, \beta >0$ and $p\in (1,\infty)$ ,  $\big(\EE\big[ \big(\frac{p\alpha T}{N} \|Y\|_{V^0}^2 \big) \big] \EE\big[ \big(\frac{p\beta T}{N} \|\tilde{Y}\|_{H^0}^2 \big)\big] \big)^N<\infty$ 
for $N$ large enough, 
which implies 
\begin{equation}  \label{mom_delta_W}
 \EE\Big[ \exp\Big(p  \sum_{l=1}^N  \big[ \alpha \|\Delta_l W\|_{V^0}^2 + \beta \|\Delta_l \WW\|_{H^0}^2 \big] \Big) \Big] <\infty.
\end{equation} 
Using \eqref{sum_errors} we deduce that for $\alpha, \beta >0$, $\epsilon, \tilde{\epsilon}\in (0,1)$ and $\mu, \tilde{\mu} \in (1,\infty)$,  
\begin{align}		\label{max_form1}
\exp\Big(& \max_{1\leq n\leq N}  \Big[ \alpha \|u^n\|_{V^0}^2 + \beta \| \theta^n\|_{H^0}^2 + \alpha \nu h \sum_{l=1}^n \|A^{\frac{1}{2}} u^l \|_{V^0}^2 
+ \beta \kappa h \sum_{l=1}^n \|\AA^{\frac{1}{2}} \theta^l \|_{H^0}^2 \Big] \Big) \nonumber \\
\leq &  \exp\Big(  \Big[  \alpha + \frac{\beta |C_L|^2 h}{\tilde{\epsilon}\kappa \tilde{\lambda}_1} \Big] \|u_0\|_{V^0}^2 \Big) \;  \exp\Big( 
\Big[  \beta + \frac{\alpha h}{\epsilon \nu \lambda_1} \Big] \|\theta_0\|_{H^0}^2  \Big) \nonumber  \\
&\times  \exp\Big( \sum_{l=1}^N  \big[ \alpha \|\Delta_l W\|_{V^0}^2 + \beta \|\Delta_l \WW\|_{H^0}^2 \big] \Big) \Big]
\Big\{ \max_{1\leq n\leq N} \exp\Big( {\mathcal M}_{t_n} - \frac{\mu}{2} \langle {\mathcal M}\rangle_{t_n}\Big) \Big\}\nonumber \\
&\times \Big\{ \max_{1\leq n\leq N} \exp\Big( \tilde{\mathcal M}_{t_n} - \frac{\tilde{\mu}}{2} \langle \tilde{\mathcal M}\rangle_{t_n}\Big) \Big\}\; 
\prod_{i=1}^4 \tau_i, 
\end{align} 
where 
\begin{align*} \tau_1=& \exp\Big( h \sum_{l=1}^{N-1}   \| A^{\frac{1}{2}} u^l\|_{V^0}^2 
\Big[ -\alpha \nu (1-\epsilon) + \frac{\beta C_L^2}{\tilde{\epsilon} \kappa \lambda_1 \tilde{\lambda}_1} \Big] \Big), \\
\tau_2=& \exp\Big( h \sum_{l=1}^{N-1}   \| \AA^{\frac{1}{2}} \theta^l\|_{H^0}^2 
\Big[ -\beta \kappa (1-\tilde{\epsilon}) + \frac{\alpha}{ \epsilon \nu \lambda_1 \tilde{\lambda}_1} \Big] \Big), \\
\tau_3= & \exp\Big( \frac{\mu}{2} 4\alpha^2 {\rm Tr}(Q) h \sum_{l=1}^{N-1} \|u^l\|_{V^0}^2 \Big), 
\quad \tau_4=  \exp\Big( \frac{\tilde{\mu}}{2} 4\beta^2 {\rm Tr}(\QQ) h \sum_{l=1}^{N-1} \|\theta^l\|_{H^0}^2 \Big),
\end{align*} 
where the last two exponential terms come from \eqref{quadvar_Mn}--\eqref{quadvar_tildeMn}
Using once more the Poinca\'e inequality, we deduce that the product $\Pi_{i=1}^4\tau_i $ can be rewritten as follows
\begin{align*}
\prod_{i=1}^4 \tau_i=& \exp\Big( \alpha h \sum_{l=1}^{N-1} \|A^{\frac{1}{2}} u^l\|_{V^0}^2 \Big[ -\nu (1-\epsilon) + \frac{\beta C_L^2}{\alpha \tilde{\epsilon} \kappa \lambda_1 \tilde{\lambda}_1}
+ \frac{2\alpha {\rm Tr}(Q) \mu}{\lambda_1}\Big] \Big)\\
&\times \exp\Big( \beta h \sum_{l=1}^{N-1} \|\AA^{\frac{1}{2}} \theta^l\|_{H^0}^2 \Big[ -\kappa (1-\tilde{\epsilon}) + \frac{\alpha}{\beta \epsilon \nu \lambda_1 \tilde{\lambda}_1}
+ \frac{2\beta {\rm Tr}(\QQ) \tilde{\mu}}{\tilde{\lambda}_1}\Big] \Big).
\end{align*}
Fix  $\delta, \tilde{\delta}  \in (0,1)$; we choose $\alpha$ and $\beta$ positive such that on one hand for $\epsilon, \tilde{\epsilon}\in (0,1)$, 
\begin{equation} 	\label{constraint} 
 \frac{\beta C_L^2}{\alpha \tilde{\epsilon} \kappa \lambda_1 \tilde{\lambda}_1}  \leq \delta \nu (1-\epsilon), \quad {\rm and} 
\quad \frac{\alpha}{\beta \epsilon \nu \lambda_1 \tilde{\lambda}_1}\leq \tilde{\delta} \kappa (1-\tilde{\epsilon}),
\end{equation} and on the other hand 
\begin{equation}	\label{constraint-Tr}
 \frac{2\alpha {\rm Tr}(Q) \mu}{\lambda_1} \leq (1-\delta) \nu (1-\epsilon), \quad {\rm and} \quad 
 \frac{2\beta {\rm Tr}(\QQ) \tilde{\mu}}{\tilde{\lambda}_1} \leq(1- \tilde{\delta}) \kappa (1-\tilde{\epsilon}).
\end{equation} 
We first satisfy \eqref{constraint}, and then find assumptions on ${\rm Tr}(Q) $ and $ {\rm Tr}(\QQ)$ to satisfy \eqref{constraint-Tr}.
Set $\alpha = \beta  \epsilon (1-\tilde{\epsilon}) \nu \kappa \lambda_1 \tilde{\lambda}_1 \tilde{\delta}$;  we need to impose 
\[ \frac{\beta C_L^2}{[\beta \epsilon (1-\tilde{\epsilon}) \nu \kappa \lambda_1 \tilde{\lambda}_1 \tilde{\delta} ] \tilde{\epsilon} \kappa \lambda_1 \tilde{\lambda}_1} \leq \delta \nu (1-\epsilon), \]
that is $C_L^2 \leq  \epsilon (1-\epsilon) \tilde{\epsilon} (1-\tilde{\epsilon}) \big( \nu \kappa \lambda_1 \tilde{\lambda}_1 \big)^2 \delta \tilde{\delta}$. 
The function $\epsilon\in (0,1) \to \epsilon (1-\epsilon)$
achieves its maximum for $\epsilon = \frac{1}{2}$ and the maximum is $\frac{1}{4}$. Let $\epsilon = \tilde{\epsilon}=\frac{1}{2}$; 
requiring that 
\[ |C_L| <  \frac{1}{4} \nu \kappa \lambda_1\tilde{\lambda}_1,  \] 
 choosing  $\delta =  \tilde{\delta}  : =  \frac{4 |C_L|}{\nu \kappa \lambda_1 \tilde{\lambda}_1} \in (0,1)$, 
 and letting 
$\alpha =  \frac{1}{4} \beta  \nu \kappa \lambda_1 \tilde{\lambda}_1 \tilde{\delta} = |C_L| \beta$,   we can fulfill \eqref{constraint}. 
Since $\tilde{\lambda} = \lambda$, the inequalities in \eqref{constraint-Tr} can be reformulated as follows
\begin{equation} 	\label{constraint_Tr_Bis}
 \frac{1}{\mu} + \frac{1}{\tilde{\mu}} <1, \quad \frac{1}{\mu} \geq  \frac{ 2 \beta |C_L| {\rm Tr}(Q)}{\nu \lambda_1 (1-\delta) \frac{1}{2}} \quad {\rm and} \quad 
\frac{1}{\tilde{\mu}} \geq  \frac{ 2 \beta {\rm Tr}(\QQ)}{\kappa \tilde{\lambda}_1 (1-\delta) \frac{1}{2}} .
\end{equation} 
Therefore, if 
 \[ \beta_0 := \frac{1}{4} \; 
 \frac{\nu \kappa \lambda_1 \tilde{\lambda}_1 -4|C_L]}{\tilde{\lambda}_1 \kappa |C_L| {\rm Tr(Q)} + \lambda_1 \nu {\rm Tr}(\QQ)},
 \]
 and $\beta < \beta_0$, we can find exponents $\mu$ and $\tilde{\mu}$ such that  \eqref{constraint_Tr_Bis} is satisfied.
  For this choice of parameters, the upper bounds  \eqref{constraint} and \eqref{constraint-Tr} are satisfied; 
 hence we have for some non negative constants $c$ and $ \tilde{c}$
\[ \prod_{i=1}^4 \tau_i = \exp\Big( -c \sum_{l=1}^{N-1} \|A^{\frac{1}{2}} u^l\|_{V^0}^2 - \tilde{c} \sum_{l=1}^{N-1} \|\AA^{\frac{1}{2}} \theta^l\|_{H^0}^2 \Big) \leq 1.\] 
 
 (i) Let $u_0$ and $\theta_0$ be deterministic;  suppose that   $|C_L|< \frac{1}{4} \nu \kappa \lambda_1\tilde{\lambda}_1$, $ \alpha = |C_L| \beta$
 and $\beta < \beta_0$,  and let $\mu$ and $\tilde{\mu}$ satisfy \eqref{constraint_Tr_Bis}. Let $p_1$ be defined by $\frac{1}{p_1} = 1-\frac{1}{\mu}- \frac{1}{\tilde{\mu}}$. 
 Take expected value in \eqref{max_form1},
apply H\"older's inequality with exponents $p_1$, $\mu$ and $\tilde{\mu}$. This yields for $c_1=|C_L| \beta \big[ 1+ \frac{2 |C_L| h}{\kappa \tilde{\lambda}_1}\big]$ and $
c_2=\beta  \big[ 1+  \frac{2 |C_L|  h}{\nu \lambda_1}\big] $ 
\begin{align*}
\EE\Big(  \Big[ & \max_{1\leq n\leq N}  \Big[ |C_L| \beta \|u^n\|_{V^0}^2 + \beta \| \theta^n\|_{H^0}^2 + |C_L| \beta\nu h \sum_{l=1}^n \|A^{\frac{1}{2}} u^l \|_{V^0}^2 
+ \beta \kappa h \sum_{l=1}^n \|\AA^{\frac{1}{2}} \theta^l \|_{H^0}^2 \Big] \Big) \nonumber \\
\leq & \exp\big( c_1 \|u_0\|_{V^0}^2 + c_2 \|\theta_0\|_{H^0}^2 \big)  
\Big\{ \EE \Big[ \exp\Big( \beta p_1  \sum_{l=1}^N \big[  |C_L| \| \Delta_l W\|_{V^0}^2 + \|\Delta_l \WW\|_{H^0}^2\big] \Big)\Big] \Big\}^{p_1}  \nonumber \\
&\times \Big\{ \EE\Big[ \max_{1\leq n\leq N} \exp\Big( \mu {\mathcal M}_{t_n} - \frac{\mu^2}{2} \langle {\mathcal M} \rangle_{t_n} \Big) \Big] \Big\}^{\frac{1}{\mu}}  
 \Big\{ \EE\Big[ \max_{1\leq n\leq N} \exp\Big( \tilde{\mu} {\mathcal M}_{t_n} - \frac{\tilde{\mu}^2}{2} \langle  \tilde{{\mathcal M}} \rangle_{t_n} \Big) \Big] \Big\}^{\frac{1}{\tilde{\mu}}} 
\end{align*} 
Since the last two factors are upper bounded by 1, \eqref{mom_delta_W} implies that the above upper bound is finite.
This concludes the proof of \eqref{exp_mom_ut_reg}.

(ii) Suppose that $u_0$ and $\theta_0$ are random, independent of $W$ and $\WW$, and that \eqref{exp_mom_ut0L2} holds. 
Let $p_1,p_2,p_3, \mu$ and $\tilde{\mu}$ belong to the interval $  (1,\infty)$
be such that $\frac{1}{p_1} + \frac{1}{p_2} + \frac{1}{p_3} + \frac{1}{\mu}+\frac{1}{\tilde{\mu}}  = 1$.  and suppose that 
$ |C_L|  < \frac{\nu \kappa \lambda_1 \tilde{\lambda}_1}{4}$  and $\alpha = |C_L| \beta$. with  
$\beta \leq  \beta_0$; then $\prod_{i=1}^4 \tau_i \leq 1$. 
Clearly, given any $\bar{\epsilon} >0 $, if $N$ is large enough we have $4 |C_L| h \big[ \frac{1}{\kappa \tilde{\lambda}_1} \vee \frac{1}{\nu \lambda_1}\big] \leq \bar{\epsilon}$.

Apply H\"older's inequality to the right hand side of \eqref{max_form1} with exponents $p_1, p_2, p_3, \mu$ and $\tilde{\mu}$;  collecting all requirements,  we have to impose 
for   $\delta = \frac{4|C_L|}{\nu \kappa \lambda_1 \tilde{\lambda}_1}$ 
 \begin{align*} 
  \frac{1}{\mu} + &\, \frac{1}{\tilde{\mu}} + \frac{1}{p_1}+\frac{1}{p_2} < 1, \quad 
 \frac{1}{\mu} \geq  \frac{ 4 \beta |C_L| {\rm Tr}(Q)}{\nu \lambda_1 (1-\delta) }, \quad 
\frac{1}{\tilde{\mu}} \geq  \frac{ 4 \beta {\rm Tr}(\QQ)}{\kappa \tilde{\lambda}_1 (1-\delta) } ,  \\
&  \; 
 p_1  |C_L| \beta  \Big[ 1+ \frac{2  |C_L| h}{\kappa \tilde{\lambda}_1}\Big] < \gamma_0, \quad {\rm and} \quad p_2 \beta  \Big[ 1+  \frac{2 |C_L|  h}{\nu \lambda_1}\Big] < \tilde{\gamma}_0, 
\end{align*} 
 that is  for $N$ large enough
\[ \beta \Big[ \frac{4 |C_L| {\rm Tr}(Q)}{\lambda_1 \nu (1-\lambda)} + \frac{4 {\rm Tr}(\QQ)}{\tilde{\lambda}_1 \kappa (1-\lambda)}+ \frac{|C_L|}{\gamma_0} + \frac{1}{\tilde{\gamma}_0}\Big] <1,
\] 
Therefore, if  $\beta < \beta_1$ where
\[ \beta_1:= \frac{\nu \kappa \lambda_1 \tilde{\lambda}_1 -4|C_L|}{4  \big[ \kappa \tilde{\lambda}_1 |C_L| {\rm Tr}(Q)  + \nu \lambda_1 {\rm Tr}(\QQ)\big]
+ \frac{|C_L| \tilde{\gamma}_0 + \gamma_0}{\gamma_0 \tilde{\gamma}_0} (\nu \kappa \lambda_1 \tilde{\lambda}_1 -4 |C_L|)}, \]
 the upper bound \eqref{exp_mom_ut_reg} is satisfied for $N$ large enough.  This completes the proof of \eqref{exp_mom_ut_reg}.

Note that if $u_0$ and $\theta_0$ are deterministic, they satisfy the condition \eqref{exp_mom_ut0L2} with any $\gamma_0$ and $\tilde{\gamma}_0$. As $\gamma_0$ and $\tilde{\gamma}_0$
 increase  to $+\infty$, we can check  that the limit of $\beta_1$ is equal to $\beta_0$, so that the results of parts (i) and (ii) are consistent.  
\end{proof}

Note that if $C_L=0$, the previous result does not give any information about exponential moments of the velocity. This is to be expected since in that case, equation \eqref{diff_theta^l_add}
does not involve the velocity. This means that we can deal with the temperature and then with the velocity. Proposition \ref{exp_mom_scheme_ut}  can be reformulated in a simpler way as follows.
\begin{coro}	\label{cor_exp_mom_t}
Suppose that $C_L=0$.

 (i) Let $\theta_0\in H^0$,  $u_0\in V^0$ be deterministic and let $\tilde{\lambda}_1$ be defined by \eqref{tilde_lambda_1}. 
 Set  $\tilde{\beta}_0:= \frac{\kappa \tilde{\lambda}_1}{4{\rm Tr}( \tilde{Q}) }$. Then given any  $\beta \in (0,\tilde{\beta}_0)$  
  there exists a constant $C(\beta)$ such that for $N$ large enough, 
\begin{equation} 	\label{exp-moments-At_al_0}
\EE\Big[ \exp\Big( \beta   \max_{0\leq n \leq N}  \Big\{  \| \theta^n\|_{H^0}^2  +   \frac{T}{N} \kappa \sum_{l=1}^n ||\tilde{A}^{\frac{1}{2}} \theta^l\|_{H^0}^2 \Big\} \Big)  \Big] 
= C(\beta) < \infty. 
\end{equation} 

 (ii)  Let $\theta_0$ be random, ${\mathcal F}_0$-measurable  and satisfy $ \EE\big[ \exp( \tilde{\gamma}_0 \|\theta_0\|_{H^0}^2 ) \big] <\infty$. 
 Set 
 \begin{equation}		\label{def_beta12}
  \tilde{\beta}_1:= \frac{\kappa \tilde{\lambda}_1 \tilde{\gamma}_0 }{4 {\rm Tr}(\QQ) \tilde{\gamma}_0 + \kappa \tilde{\lambda}_1}.
 \end{equation}
   Then   for  
   $\beta < \tilde{\beta}_1$,
 the equation 
 \eqref{exp-moments-At_al_0}) holds for $N$ large enough. 
\end{coro}

When $C_L=0$, the following result proves the existence of exponential moments for the discretization scheme $\{ u^n\}_n$.
\begin{prop}		\label{exp_mom_schemeu}
Suppose that $C_L=0$.

 (i) Let $\theta_0\in H^0$,  $u_0\in V^0$ be deterministic; let $\tilde{\lambda}_1$ and $\lambda_1$  be defined by \eqref{tilde_lambda_1} and \eqref{lambda_1} respectively. 
Let   $\alpha_1:=  \frac{\nu \kappa \lambda_1 \tilde{\lambda}_1}{4\big[ 2 T {\rm Tr}(\QQ) + \kappa \tilde{\lambda}_1 {\rm Tr}(Q)\big]} $; 
then  for any $\alpha \in \big(0,  \alpha_1\big)$ we have for $N$ large enough 
\begin{equation} 	\label{exp-moments-Aun}
\EE\Big[ \exp\Big( \alpha  \max_{0\leq n \leq N}  \Big\{  \| u^n\|_{V^0}^2  +   \frac{T}{N} \nu \sum_{l=1}^n \| {A}^{\frac{1}{2}} u^l\|_{V^0}^2 \Big\} \Big)  \Big] 
= C(\alpha) < \infty. 
\end{equation} 

 (ii)  Let $\theta_0$ and $u_0$ be random, ${\mathcal F}_0$-measurable  and satisfy \eqref{exp_mom_ut0L2}.   Set 
 \begin{equation}	\label{def_alpha01}
  \tilde{\alpha}_1 := \frac{\nu \kappa \lambda_1 \tilde{\lambda}_1}{4\big[ 2T {\rm Tr}(\QQ) + \kappa  \tilde{\lambda}_1 {\rm Tr}(Q)\big] +\frac{2\kappa \tilde{\lambda}_1 T}{\tilde{\gamma}_0}
 + \frac{\nu \kappa \lambda_1 \tilde{\lambda}_1}{\gamma_0}}.
 \end{equation} 
 Then   for 
 $\alpha \in  (0, \tilde{\alpha}_1) $,  the inequality  
   \eqref{exp-moments-Aun} holds for $N$ large enough. 

\end{prop}
\begin{proof} 
We use the notations and some computations used in the proof of Proposition \ref{exp_mom_scheme_ut}.
 Adding the upper estimates \eqref{diff-ul_add-2} for $l=1$ to $n$ we deduce that for every $n=1, ..., N$, $\alpha >0$ and $\epsilon \in (0,1)$,
\begin{align*}
\alpha\Big(  \|u^n\|_{V^0}^2  &+ h \nu \sum_{l=1}^n \|A^{\frac{1}{2}} u^l\|_{V^0}^2  \Big)
 \leq \; \alpha \|u_0\|_{V^0}^2 + \frac{\alpha h}{\epsilon \nu \lambda_1} \|\theta_0\|_{H^0}^2 + 2\alpha \sum_{l=1}^{n-1} \big( \Delta_l W, u^{l-1}\big)  \\
 & \quad  + \alpha \sum_{l=1}^n 
\| \Delta_l W\|_{V^0}^2  - \alpha (1-\epsilon) h\sum_{l=1}^n \|A^{\frac{1}{2}} u^l\|_{V^0}^2 + \frac{\alpha}{\epsilon \lambda_1 \nu} h  \sum_{l=1}^{n-1} \|\theta^l\|_{H^0}^2. 
\end{align*} 
We at first prove \eqref{exp-moments-Aun} and write an analog of equation \eqref{max_form1}, focusing on the velocity in the left hand side. Let $\{M_n\}_n$ and 
$({\mathcal M}_t, {\mathcal F}_t, t\in [0,1]) $ be defined in the proof of Proposition \ref{exp_mom_scheme_ut}. Then for $\mu \in (1,\infty)$ we have 
 \begin{align}		\label{max_form1_C0}
\exp\Big(& \max_{1\leq n\leq N}  \Big[ \alpha \|u^n\|_{V^0}^2  + \alpha \nu h \sum_{l=1}^n \|A^{\frac{1}{2}} u^l \|_{V^0}^2 \Big] \Big) \nonumber \\
\leq & \exp\big( \alpha \|u_0\|_{V^0}^2\big)  \exp\Big( \frac{\alpha h}{\epsilon \nu \lambda_1} \|\theta_0\|_{H^0}^2  \Big) 
  \exp\Big(\alpha  \sum_{l=1}^N   \|\Delta_l W\|_{V^0}^2 \Big) \nonumber  \\
&\times 
\Big\{ \max_{1\leq n\leq N} \exp\Big( {\mathcal M}_{t_n} - \frac{\mu}{2} \langle {\mathcal M}\rangle_{t_n}\Big) \Big\}\ \exp\Big( \frac{\alpha T}{\epsilon \nu \lambda_1} 
\max_{1\leq n\leq N} \|\theta^n\|_{H^0}^2 \Big) \prod_{i=1}^2 \tau_i, 
\end{align} 
where 
\begin{align*} \tau_1=& \exp\Big( - \alpha \nu (1-\epsilon) h \sum_{l=1}^{N-1}   \| A^{\frac{1}{2}} u^l\|_{V^0}^2 \Big), \\
\tau_2= & \exp\Big( \frac{\mu}{2} 4\alpha^2 {\rm Tr}(Q) h \sum_{l=1}^{N-1} \|u^l\|_{V^0}^2 \Big).
\end{align*} 
The Poincar\'e inequality \eqref{lambda_1} implies that if  $\alpha \frac{2\mu {\rm Tr}(Q)}{\lambda_1} \leq \nu (1-\epsilon)$ for some $\epsilon \in (0,1)$, 
we have $\tau_1 \tau_2 \leq 1$.
Furthermore, \eqref{mom_delta_W} implies that  for $p\in (1,\infty)$ we have  for $N$ large enough 
 $\EE\big[ \exp\big( p \alpha \sum_{l=1}^N \|\Delta_l W\|_{H^0}^2 \big) \big] <\infty$ for $N$ large enough. 
\smallskip

(i) Let $u_0$ and $\theta_0$ be deterministic. Let $p_1, p_2, \mu\in (1,\infty)$ be such that $\frac{1}{p_1} + \frac{1}{p_2} + \frac{1}{\mu} =1$ and suppose that
 $\alpha \frac{2\mu {\rm Tr}(Q)}{\lambda_1} \leq \nu (1-\epsilon)$ for some $\epsilon \in (0, 1)$.  Apply H\"older's inequality with exponents $p_1, p_2$ and $\mu$; this yields
\begin{align*}
\EE\Big[ \exp\Big(& \max_{1\leq n\leq N}  \Big[ \alpha \|u^n\|_{V^0}^2  + \alpha \nu h \sum_{l=1}^n \|A^{\frac{1}{2}} u^l \|_{V^0}^2 \Big] \Big) \Big] \\
\leq & \; \exp\big( \alpha \|u_0\|_{V^0}^2\big)  \exp\Big( \frac{\alpha h}{\epsilon \nu \lambda_1} \|\theta_0\|_{H^0}^2  \Big) 
\EE\Big[ \exp\Big(\alpha  p_2  \sum_{l=1}^N   \|\Delta_l W\|_{V^0}^2 \Big) \Big] \nonumber  \\
&\times 
\EE\Big[ \Big\{ \max_{1\leq n\leq N} \exp\Big(\mu  {\mathcal M}_{t_n} - \frac{\mu^2}{2} \langle {\mathcal M}\rangle_{t_n}\Big) \Big\}\Big] \;  \EE\Big[ \exp\Big( \frac{\alpha Tp_1}{\epsilon \nu \lambda_1} 
\max_{1\leq n\leq N} \|\theta^n\|_{H^0}^2 \Big) \Big]. 
\end{align*} 
The upper estimate \eqref{exp-moments-At_al_0} in Corollary \ref{cor_exp_mom_t} 
 implies that  for $N$ large enough, 
 \[ \EE\Big[ \exp\Big( \frac{\alpha T}{\epsilon \nu \lambda_1} p_1 \max_{1\leq n\leq N} \|\theta^n\|_{H^0}^2 \Big) \Big] <\infty\quad {\rm  if } \quad 
 \frac{\alpha T}{\epsilon \nu \lambda_1} p_1 < \frac{\kappa \tilde{\lambda}_1}{4{\rm Tr}(\QQ)},\] 
Since  $\mu$ is chosen such that $\EE\Big[ \max_{1\leq n\leq N} \exp\Big(\mu  {\mathcal M}_{t_n} - \frac{\mu^2}{2} \langle {\mathcal M}\rangle_{t_n}\Big) \Big\}\Big] \leq 1$ and
given any $p_2\in (1,\infty)$,  $ \EE\big[ \exp \big( \alpha p_2  \sum_{l=1}^N   \|\Delta_l W\|_{V^0}^2 \Big) \Big] <\infty$ for $ N$ large enough, we deduce that 
we have to find $\mu$ and $p_1$ such that 
\[ 1 > \frac{1}{p_1} + \frac{1}{\mu} > \frac{2\alpha}{\nu \lambda_1} \Big[  \frac{2 T {\rm Tr}(\QQ)}{\epsilon \kappa \tilde{\lambda}_1} + \frac{{\rm Tr}(Q)}{1-\epsilon}\Big] \]
for some $\epsilon \in (0, 1)$.  
 Choose $\epsilon$  such that $\frac{2T {\rm Tr}(\QQ)}{\epsilon \kappa \tilde{\lambda_1}} = \frac{{\rm Tr}(Q)}{1-\epsilon}$, that is
 $\epsilon = \frac{2T {\rm Tr}(\QQ)}{2T {\rm Tr}(\QQ) + \kappa \tilde{\lambda}_1 {\rm Tr}(Q)}< 1$;
  this yields 
\[ \alpha_1:= \frac{\nu \kappa \lambda_1 \tilde{\lambda}_1}{4\big[ 2T {\rm Tr}(\QQ) + \kappa \tilde{\lambda}_1 {\rm Tr}(Q)\big]} \]
such that for every $\alpha \in (0, \alpha_1)$  the estimate \eqref{exp-moments-Aun} holds, 
\smallskip

(ii) Let $u_0$ and $\theta_0$ be random such that \eqref{exp_mom_ut0L2} holds. We at first prove \eqref{exp-moments-Aun}. Since for every $\epsilon >0$, $\alpha >0$ and $p\in (1,\infty)$
 we have  $\frac{\alpha h}{\epsilon \nu \lambda_1} < \tilde{\gamma}_0$ is $N$is large enough, we have to find positive exponents $p_1, p_2$ and $\nu$ such that $\frac{1}{p_1} +
 \frac{1}{p_2} + \frac{1}{\mu} < 1$, and for  $\tilde{\beta}_1$  defined in Corollary  \ref{cor_exp_mom_t}
  \[ \frac{\alpha 2 \mu {\rm Tr}(Q)}{\lambda_1} < \nu(1-\epsilon), \quad \frac{\alpha T p_1}{\epsilon \nu \lambda_1} <  \tilde{\beta}_1
   \quad {\rm and} \quad  \alpha p_2 \leq \gamma_0  \] 
for some $\epsilon \in (0, 1)$. This can be summarized as
 \[ 1 > \alpha\Big[ \frac{T}{\epsilon \nu \lambda_1 \tilde{\beta}_1}  + \frac{1}{\gamma_0} + \frac{2{\rm Tr}(Q)}{\nu \lambda_1 (1-\epsilon)}\Big],\]
for some $\epsilon \in (0,1)$.  Choose $\epsilon$ such that $\frac{T}{\epsilon \tilde{\beta}_1} = \frac{2 {\rm Tr}(Q)}{1-\epsilon}$, that is
$\epsilon = \frac{T}{T+ 2 \tilde{\beta}_1 {\rm Tr}(Q)} \in (0, 1)$. 
We deduce  that for 
\[ \tilde{\alpha}_1:= \frac{\nu \lambda_1  \tilde{\beta}_1  \gamma_0}{2\gamma_0 \big[ T + 2 {\rm Tr}(Q)  \tilde{\beta}_1  \big] 
+ \nu \lambda_1 \tilde{\beta}_1 }, \]  
when  $\alpha \in (0,\tilde{\alpha}_1)$ the estimate \eqref{exp-moments-Aun} is satisfied for $N$ large enough.
Plugging in the value of $\tilde{\beta}_1$ from \eqref{def_beta12}  concludes the proof. 

Once more, as $\gamma_0$ and $\tilde{\gamma}_0$ increase to $+\infty$,  $\tilde{\alpha}_1 \to \alpha_1$; this shows that the results stated in parts
(i) and (ii) are consistent.
\end{proof}

\section{Strong convergence of the semi-implicit time Euler scheme - additive noise}		\label{s-converg}
The following theorem is one of the main results of this paper. It proves the "optimal" polynomial rate of convergence of the scheme uniformly on the time grid in $V^0$ 
for the velocity, and in $H^0$ for the temperature. 
For $j=0, ...,N$ recall that  $e_j:= u(t_j)-u^j$ and $ \tilde{e}_j:= \theta(t_j)-\theta^j$; then $e_0=\tilde{e}_0=0$.

\begin{theorem}		\label{th-s-cv_CLneq0}
Suppose that  $|C_L| \in \Big(0, \frac{\nu \kappa \lambda_1 \tilde{\lambda}_1}{4}\Big)$. 

Let  $u_0\in V^1$ a.s. and $\theta_0\in H^1$ a.s. 
Let $u$ and $\theta$ be the solutions of \eqref{def_u} and \eqref{def_t} respectively, $\{ u^l\}_{l=0,..., N}$ and $\{\theta^l\}_{l=1, ..., N}$ be the semi-implicit
Euler scheme defined by \eqref{Euler_u} and \eqref{Euler_t}  respectively. 
Then if either assumption (i) or (ii) is satisfied,
we have 
\begin{equation}	\label{s-rate}
\EE\Big( \max_{1\leq l\leq N} \big( \|u(t_l)-u^l\|_{V^0}^2 + \|\theta(t_l)-\theta^l\|_{H^0}^2 \big)\Big) \leq C \Big(\frac{T}{N}\Big)^\eta,\qquad \eta \in (0,1).
\end{equation}

(i)  $u_0\in V^1$,  $\theta_0\in H^1$ are deterministic,  and 
\[ \tilde{\lambda}_1 \kappa |C_L| {\rm Tr(Q)} + \lambda_1 \nu {\rm Tr}(\QQ)  <   \frac{\nu \kappa \lambda_1 \tilde{\lambda}_1 - 4 |C_L] }{
 \frac{ \bar{C}_4^4}{\kappa} \Big( \frac{1}{\nu} + \frac{1}{\kappa} + \frac{2\kappa}{\nu^2 |C_L|} \Big)}  \]

(ii)   $u_0$ and $\theta_0$ are ${\mathcal F}_0$-measurable and satisfy \eqref{exp_mom_ut0L2}.  
Suppose furthermore that $u_0 \in L^{p}(\Omega;V^1)$ and $\theta_0 \in L^p(\Omega;H^1)$ for any $p\in [2,\infty)$ (for example $u_0$ and $\theta_0$ are Gaussian random variables),  
Suppose furthermore that 
\[   \frac{ \bar{C}_4^4}{\kappa} \Big( \frac{1}{\nu} + \frac{1}{\kappa} + \frac{2\kappa}{\nu^2 |C_L|} \Big) <  4 \beta_1, 
\]
where the constant  $\beta_1$ 
is defined  by \eqref{beta_1} in Proposition  \ref{exp_mom_scheme_ut}. 
\end{theorem} 
\begin{proof}
Using \eqref{def_error_u}--\eqref{def_error_t} reformulated for an additive stochastic perturbation, we deduce that  $j=1, ...,N$, $\varphi\in V^1$ and $\psi\in H^1$, 
\begin{align}	\label{def_error_u_add}
\big( &e_j-e_{j-1}\, , \, \varphi \big) + \nu \!\int_{t_{j-1}}^{t_j} \! \!\big(  A^{\frac{1}{2}}  [ u(s) -   u^j ] ,  A^{\frac{1}{2}}  \varphi\big)  ds 
+ \int_{t_{j-1}}^{t_j} \!\! \big\langle B(u(s),u(s)) - B(u^{j-1},u^j) ,  \varphi\big\rangle ds \nonumber \\
&\qquad  =  \int_{t_j}^{t_{j+1}} \big( \Pi [\theta(s)-\theta^{j-1} ] v_2, \varphi\big) ds,  
\end{align}
and
\begin{align}		\label{def_error_t_add}
\big( &\ee_j-\ee_{j-1}\, , \, \psi \big) + \kappa \!\int_{t_{j-1}}^{t_j} \! \! \! \big( \AA^{\frac{1}{2}}  [ \theta(s) -   \theta^j ] ,  \AA^{\frac{1}{2}}  \psi\big)  ds 
+ \int_{t_{j-1}}^{t_j} \!\! \! \big\langle [u(s).\nabla]\theta(s)  - [u^{j-1}.\nabla]\theta^j] ,  \psi\big\rangle ds  \nonumber \\
&\qquad + C_L  \int_{t_{j-1}}^{t_j} \big( u_2(s)-u_2^{j-1}\, , \, \psi) ds =0 .  
\end{align}
Write \eqref{def_error_u_add} with $\varphi = e_j$;  
using the equation $(a-b,a)=\frac{1}{2} \big[ a^2-b^2+|a-b|^2\big]$, this yields
\begin{equation}		\label{error_u}
\|e_j\|_{V^0}^2 - \|e_{j-1}\|_{V^0}^2 + \|e_j-e_{j-1}\|_{V^0}^2  + 2  \nu h \|A^{\frac{1}{2}}  e_j\|_{V^0}^2  =  2 \sum_{i=1}^5 T_{j,i},\quad j=1, ..., N, 
\end{equation}
where using the antisymmetry \eqref{B} we have 
\begin{align*}
T_{j,1}=&\;  - \int_{t_{j-1}}^{t_j} \big\langle B\big(e_{j-1}, u^j \big)\, , \, e_j\big\rangle ds = - \int_{t_{j-1}}^{t_j} \big\langle B(e_{j-1}, u(t_j) \, , \, e_j\big\rangle ds \\
T_{j,2}=&\;  -\int_{t_{j-1}}^{t_j} \big\langle B\big(u(s)-u(t_{j-1}) , u^j  \big)\, , \, e_j \big\rangle ds  =  -\int_{t_{j-1}}^{t_j} \big\langle B\big(u(s)-u(t_{j-1}) , u(t_j) \big)\, , \, e_j \big\rangle ds,\\
T_{j,3}=&\;  - \int_{t_{j-1}}^{t_j} \big\langle B\big(u(s), u(s)-u^j \big) \, , \, e_j\big\rangle ds = - \int_{t_{j-1}}^{t_j} \big\langle B\big(u(s), u(s)-u(t_j) \big) \, , \, e_j\big\rangle ds, \\
T_{j,4} = &\; - \nu  \int_{t_{j-1}}^{t_j} \big( A^{\frac{1}{2}} [u(s)-u^j] \, , \, A^{\frac{1}{2}} e_j\big) ds, \quad T_{j,5}=  \int_{t_{j-1}}^{t_j} \big( [\theta(s) - \theta^{j-1} ] v_2 \, , \, e_j\big) ds.
\end{align*} 
The H\"older and Gagliardo-Nirenberg inequalities imply for $\delta_1 >0$ and $j=1, ..., n-1$
\begin{align}		\label{up_Tj1}
|T_{j,1}| \leq & \; \bar{C}_4^2 h \|e_{j-1}\|_{V^0}^{\frac{1}{2}} \|e_j\|_{V^0}^{\frac{1}{2}} \|A^{\frac{1}{2}} e_{j-1}\|_{V^0}^{\frac{1}{2}} \|A^{\frac{1}{2}} e_j\|_{V^0}^{\frac{1}{2}} 
\| A^{\frac{1}{2}} u^j \|_{V^0} \nonumber \\
\leq & \;  \delta_1 h \nu \big[ \|A^{\frac{1}{2}} e_{j-1}\|_{V^0}^2 +\|A^{\frac{1}{2}} e_j \|_{V^0}^2 \big]+ \frac{\bar{C}_4^4}{16 \delta_1 \nu } h \|A^{\frac{1}{2}} u^j \|_{V^0}^2  \big[\|e_{j-1}\|_{V^0}^2
+ \|e_j\|_{V^0}^2\big],
\end{align}
while
\begin{align}		\label{up_Tn1_Bis}
|T_{n,1}| \leq & \delta_1 h \nu \big[ \|A^{\frac{1}{2}} e_{n-1}\|_{V^0}^2 + \| A^{\frac{1}{2}} e_n\|_{V^0}^2 \big] 
 + \frac{\bar{C}_4^4}{16 \delta_1 \nu } h  \|A^{\frac{1}{2}} u(t_n)\|_{V^0}^2
\big[ \|e^{n-1}\|_{V^0}^2 + \|e^n\|_{V^0}^2\big] .
\end{align}
Note that the terms $T_{j,i}, i=2, ..., 5$ are identical to those used in the decomposition \eqref{increm_error_u}. Hence we can upper estimates them as in the proof of  Proposition \ref{th_loc_cv}.

Adding the upper estimates \eqref{up_Tj2} -- \eqref{up_Tj5} for $j=1, ..., n$, \eqref{up_Tj1}  for $j=1, ..., n-1$ and  \eqref{up_Tn1_Bis}, we deduce  from \eqref{error_u} 
\begin{align}	\label{up_e_n}
\|e_n\|_{V^0}^2 &+ \sum_{j=1}^n \|e_j-e_{j-1}\|_{V^0}^2 + 2 \nu h \Big( 1-2\delta_1 - \sum_{i=2}^4 \delta_j\Big) \sum_{j=1}^n \|A^{\frac{1}{2}} e_j\|_{V^0}^2  \nonumber \\
\leq & 2 h  \sum_{j=1}^{n-1}  \|e_j\|_{V^0}^2  \Big( 3+  \frac{\bar{C}_4^4}{16\delta_1 \nu} \|A^{\frac{1}{2}} u^{j+1} \|_{V^0}^2  + \frac{\bar{C}_4^4}{16\delta_1 \nu} \|A^{\frac{1}{2}} u^j\|_{V^0}^2 \Big) 
\nonumber \\
&\;  +2  h \sum_{j=1}^{n-1} \|  \ee_j\|_{H^0}^2 + Z_1, 
\end{align}
where since $\|u(s)-u(t_j)\|_{V^0}^2  \leq \frac{1}{\lambda_1} \| A^{\frac{1}{2}} (u(s)-u(t_j))\|_{V^0}^2$, 
\begin{align} 		\label{Z1}  
Z_1= & \;  h C(\delta_1)  \big(\| u(t_n)\|_{V^0}^2 + \| u^n\|_{V^0}^2 \big)  \big( \|A^{\frac{1}{2}} u(t_n)\|_{V^0}^2 +1\big)
+ \sum_{j=1}^n   \int_{t_{j-1}}^{t_j} \!\! \|\theta(s)-\theta(t_{j-1}) \|_{H^0}^2 ds \nonumber \\
&\; + C(\delta_2, \delta_3, \delta_4 ) \sup_{s\in [0, t_n]} \big[1+\|A^{\frac{1}{2}} u(s)\|_{V^0}^2\big] \sum_{j=1}^n \int_{t_{j-1}}^{t_j} \!\!  \|A^{\frac{1}{2}} (u(s)-u(t_j))\|_{V^0}^2 ds .
\end{align} 
We then deal with the temperature. Write \eqref{def_error_t_add} with $\psi= \tilde{e}_j$; then since $\langle [u.\nabla] \theta\, , \, \theta\rangle =0$ , we have 
\begin{align} 		\label{norm_eej_add}
\|   \ee_j\|_{H^0}^2 &-  \| \ee_{j-1}\|_{H^0}^2 + \|  \ee_j- \ee_{j-1}\|_{H^0}^2 
+ 2 \kappa h \| \tilde{A}^{\frac{1}{2}} \ee_j\|_{H^0}^2     +  2\sum_{i=1}^4 \tilde{T}_{j,i},
\end{align}
where  
\begin{align*}
\tilde{T}_{j,1}&\; = -\int_{t_{j-1}}^{t_j} \!\! \big\langle  [e_{j-1}. \nabla] \theta^j \, , \,  \ee_j \big\rangle\, ds =  -\int_{t_{j-1}}^{t_j} \!\! \big\langle  [e_{j-1}. \nabla] \theta(t _j) \, , \,  \ee_j \big\rangle\, ds , \\
\tilde{T}_{j,2}&\; = -\int_{t_{j-1}}^{t_j} \!\! \big\langle  [u(s)-u(t_{j-1})]  . \nabla] \theta^j\, , \,  \ee_j \big\rangle \, ds  
 =  -\int_{t_{j-1}}^{t_j} \!\! \big\langle  [u(s)-u(t_{j-1})]  . \nabla] \theta(t_j) \, , \,  \ee_j \big\rangle \, ds ,\\
\tilde{T}_{j,3}&\; = -\int_{t_{j-1}}^{t_j} \!\! \big\langle  [u(s) .\nabla ] \big( \theta(s)-\theta^j\big) \, , \, \ee_j \big\rangle \, ds = 
 -\int_{t_{j-1}}^{t_j} \!\! \big\langle  [u(s) .\nabla ] \big( \theta(s)-\theta(t_j)\big) \, , \, \ee_j \big\rangle \, ds ,\\
\tilde{T}_{j,4}&\; =  -\kappa \int_{t_{j-1}}^{t_j} \!\! \big( \tilde{A}^{\frac{1}{2}} [\theta(s)-\theta(t_j)]\, , \, \tilde{A}^{\frac{1}{2}} \ee_j \big) ds, 
\quad \tilde{T}_{j,5}=-C_L \int_{t_{j-1}}^{t_j} \!\! \big( u_2(s)-u^{j-1}_2 , \ee_j\big) ds,
\end{align*}
Using the H\"older, Gagliardo and Young inequalities, we deduce for $\delta_5>0$  and  $j=1, ..., n-1$
\begin{align}		\label{up_tildeTj1add}
|\tilde{T}_{j,1}| \leq &\;  h \; \bar{C}_4^2 \; \|e_{j-1}\|_{V^0}^{\frac{1}{2}} \|A^{\frac{1}{2}} e_{j-1}\|_{V^0}^{\frac{1}{2}} \|\tilde{A}^{\frac{1}{2}} \theta^j\|_{H^0} \|\ee_j\|_{H^0}^{\frac{1}{2}}
\| \tilde{A}^{\frac{1}{2}} \ee_j\|_{H^0}^{\frac{1}{2}} \nonumber \\
\leq & \; \delta_5 \nu h  \|A^{\frac{1}{2}} e_{j-1}\|_{H^0}^2 + \tilde{\delta}_1 \kappa h \|\tilde{A}^{\frac{1}{2}} \ee_j\|_{H^0}^2  \nonumber \\
&\;  + \frac{\bar{C}_4^4}{16 \nu \delta_5}h  \|\tilde{A}^{\frac{1}{2}} \theta^j\|_{H^0}^2
\|e_{j-1}\|_{V^0}^2 +  \frac{\bar{C}_4^4}{16 \kappa \tilde{\delta}_1}h  \|\tilde{A}^{\frac{1}{2}} \theta^j\|_{H^0}^2  \| \ee_j\|_{H^0}^2.
\end{align}
as well as 
\begin{align}		\label{up_tildeTj1addBis}
|\tilde{T}_{n,1}| \leq &\; \delta_5 \nu h \|A^{\frac{1}{2}} e_{n-1}\|_{H^0}^2 + \tilde{\delta}_1 \kappa h \|\tilde{A}^{\frac{1}{2}} \ee_n\|_{H^0}^2  \nonumber \\
&\;  + \frac{\bar{C}_4^4}{16 \nu \delta_5}h  \|\tilde{A}^{\frac{1}{2}} \theta(t_n)\|_{H^0}^2
\|e_{n-1}\|_{V^0}^2 +  \frac{\bar{C}_4^4}{16 \kappa \tilde{\delta}_1}h  \|A^{\frac{1}{2}} \theta(t_n)\|_{H^0}^2  \| \ee_n\|_{H^0}^2 .
\end{align}
We upper estimate the terms $\tilde{T}_{j,i}$ for $j=1, ..., n$ and $i=2, ..., 5$ as in the proof of Proposition \ref{th_loc_cv}. 
Adding the upper estimates \eqref{norm_eej_add} for $j=1, ...,n$ and using  \eqref{up_tildeTj2}--\eqref{up_tildeTj5} for $j=1, ...,n$, \eqref{up_tildeTj1add} for $j=1, ..., n-1$ and \eqref{up_tildeTj1addBis}, we obtain 
\begin{align}		\label{upper_een-v1}
\|\ee_n\|_{H^0}^2 + & \sum_{j=1}^n \|\ee_j-\ee_{j-1}\|_{H^0}^2  
+ 2h\kappa \Big( 1-\sum_{i=1}^4 \tilde{\delta}_4 \Big)\sum_{j=1}^n  \| \tilde{A}^{\frac{1}{2}} \ee_j\|_{H^0}^2 
 \nonumber \\
\leq & \; 2 \delta_5 \nu h \sum_{j=1}^n \|A^{\frac{1}{2}} e_{j-1}\|_{V^0}^2  
+  h   \sum_{j=1}^{n-1}  \Big[  \frac{\bar{C}_4^4}{8 \nu \delta_5} \|\tilde{A}^{\frac{1}{2}} \theta^{j+1} \|_{H^0}^2   + |C_L| \Big] \|e_j\|_{V^0}^2  \nonumber \\
&\;  + h   \sum_{j=1}^{n-1}  \|\ee_j\|_{H^0}^2 \Big[ \frac{\bar{C}_4^4}{8 \kappa \tilde{\delta}_1} \|\tilde{A}^{\frac{1}{2}} \theta^j \|_{H^0}^2 
+ 4 + 2 |C_L| \Big] + Z_2, 
\end{align} 
where, using the inequality \eqref{lambda_1},  we have 
\begin{align}	\label{Z2}
Z_2=&\; C(\kappa, \tilde{\delta}_2,  C_L, \lambda_1)\Big[ 1+  \sup_{s\in [0,T]} \| \tilde{A}^{\frac{1}{2}} \theta(s)\|_{H^0}^2 \Big] 
\int_{t_{j-1}}^{t_j} \!\!  \| A^{\frac{1}{2}}[ u(s)-u(t_{j-1})] \|_{V^0}^2 ds\Big] \nonumber \\
 &\; +C(\kappa, \tilde{\delta}_3,  \tilde{\delta}_4)   \sup_{s\in [0,T]} \|A^{\frac{1}{2}} u(s)\|_{V^0}^2 
 \sum_{j=1}^n 
   \int_{t_{j-1}}^{t_j} \!\!   \|\tilde{A}^{\frac{1}{2}}[ \theta(s)-\theta(t_{j})]\|_{H^0}^2 ds  \nonumber \\
   &\; + C(\delta_5, \tilde{\delta}_1) h \|\tilde{A}^{\frac{1}{2}} \theta(t_n)\|_{H^0}^2 \Big[ \|u(t_{n-1})\|_{V^0}^2 + \|u^{n-1}\|_{V^0}^2 + \|\theta(t_n)\|_{H^0}^2 + \|\theta^n\|_{H^0}^2\Big].
\end{align} 
Adding the upper estimates \eqref{up_e_n} and \eqref{upper_een-v1}, we deduce that for any $n\in \{1, ..., N\}$ we have
\begin{align*}
\sup_{1\leq l \leq n} \Big[ & \|e_l\|_{V^0}^2 + \|\ee_l\|_{H^0}^2+  2 \nu h \Big( 1- 2\delta_1 -\sum_{i=2}^5 \delta_i\Big) \sum_{j=1}^l \|A^{\frac{1}{2}} e_j\|_{V^0}^2 \\
&\quad + 2\kappa h \Big( 1-\sum_{i=1}^4 \tilde{\delta}_i \Big)
\sum_{j=1}^l  \|\tilde{A}^{\frac{1}{2}} \ee_j\|_{H^0}^2 \Big] 
\leq  h \sum_{j=1}^{n-1} \big[ \|e_j\|_{V^0}^2 + \|\ee_j\|_{H^0}^2 \big]  Y_j + Z_1+Z_2,
\end{align*} 
where for $\delta_1 \in (0,\frac{1}{2})$, $0<\delta_5<  1-2\delta_1 $  and $\tilde{\delta}_1\in (0,1)$, 
\begin{align}		\label{Yj}
 Y_j &=\;   6+2  |C_L|+  \frac{\bar{C}_4^4}{8} \Big[  \frac{\|A^{\frac{1}{2}} u^{j+1}\|_{V^0}^2 + \|A^{\frac{1}{2}} u^j\|_{V^0}^2 }{\nu \delta_1 }   
+  \frac{ \|\tilde{A}^{\frac{1}{2}} \theta^{j+1} \|_{H^0}^2  }{\nu \delta_5}  +   \frac{\|\tilde{A}^{\frac{1}{2}} \theta^j \|_{H^0}^2}{\kappa \tilde{\delta}_1} \Big] .   
\end{align} 
Let $2\delta_1 + \sum_{i=2}^5 \delta_i <1$ and $\sum_{i=1}^4 \tilde{\delta}_i<1$. Then the discrete Gronwall lemma implies that 
\begin{align}	\label{up-error}
\max_{1\leq l\leq N} \big(  \|e_l\|_{V^0}^2 + \|\ee_l\|_{H^0}^2\big) \leq \exp\Big( h \sum_{j=1}^{N-1} Y_j\Big) \big[ Z_1+Z_2]. 
\end{align} 
Using the Cauchy-Schwarz inequality and Propositions~\ref{prop_u_V1}, \ref{prop_mom_t_H1} and \ref{prop_regularity_H1}, we deduce 
 that for any exponent $q\in (1,\infty)$, 
\[ \| Z_1 + Z_2\|_{L^q(\Omega)} \leq  C \; h^\eta, \quad \eta \in \Big( 0,\frac{1}{2} \Big).\]
Using  H\"older's inequality, we see that  the proof will be complete if we  check that under appropriate conditions, $\EE\big[ \exp \big( ph \sum_{j=1}^{N-1} Y_j\big) \big] <\infty$
for some $p\in (1,\infty)$, which can be chosen close to 1. 
\medskip

$\bullet$ {\bf Deterministic initial conditions $u_0$ and $\theta_0$}
For deterministic $u_0\in V^1$ and $\theta_0\in H^1$, 
let 
 $\tilde{\delta}_1\sim 1$, $\delta_1 \in (0, \frac{1}{2})$,  $\epsilon \in (0, 1-2\delta_1) $ and $\delta_5=1-2\delta_1-\epsilon$;  let $\beta_0$
be defined by \eqref{def_beta0} in the statement of  Proposition \ref{exp_mom_scheme_ut} to have \eqref{exp_mom_ut_reg}. We have to choose $p \in (1,\infty)$, close to 1,  such that   
\begin{equation}		\label{constraint_p-1}
 \frac{p \bar{C}_4^4}{4\nu^2 \delta_1 |C_L| } < \beta_0  , \quad \mbox{\rm and } \quad 
  \frac{p \bar{C}_4^4 }{4\kappa } \Big[ \frac{1}{\nu( 1-\epsilon - 2\delta_1)}+ \frac{1}{\kappa}\Big] < \beta_0 .
  \end{equation}  
Choose $\delta_1$ such that $\frac{1}{\delta_1 \nu |C_L|} = \frac{1}{\kappa (1-2\epsilon -2\delta_1)}$, that is $\delta_1=\frac{\kappa (1-2\epsilon)}{2\kappa+|C_L| \nu} < \frac{1}{2}-\epsilon$.  
 Then,  if 
 \[\frac{ \bar{C}_4^4}{\kappa} \Big( \frac{1}{\nu} + \frac{1}{\kappa} + \frac{2\kappa}{\nu^2 |C_L|}\Big) <  4 \beta_0= 
  \frac{\nu \kappa \lambda_1 \tilde{\lambda}_1 - 4 |C_L] }{\tilde{\lambda}_1 \kappa |C_L| {\rm Tr(Q)} + \lambda_1 \nu {\rm Tr}(\QQ)}. ,\] 
  we can find $\epsilon\sim 0$ and $p\sim 1$  such that \eqref{constraint_p-1} holds true.

  $\bullet$ {\bf Random initial conditions} Let $u_0$ and $\theta_0$ satisfy \eqref{exp_mom_ut0L2}. 
 Once more let  
  $\tilde{\delta}_1\sim 1$, $\delta_1 \in (0, \frac{1}{2})$,  $\epsilon \in (0, 1-2\delta_1) $ and $\delta_5=1-2\delta_1-\epsilon$; 
  let $\beta_1$
be defined by \eqref{beta_1} in the statement of  Proposition \ref{exp_mom_scheme_ut} to have \eqref{exp_mom_ut_reg} . We have to choose $p \in (1,\infty)$, close to 1,  such that     
\begin{equation}		\label{constraint_p-1_random}
 \frac{p \bar{C}_4^4}{4\nu^2 \delta_1 |C_L| } < \beta_1  , \quad \mbox{\rm and } \quad 
  \frac{p \bar{C}_4^4 }{4\kappa } \Big[ \frac{1}{\nu( 1-\epsilon - 2\delta_1)}+ \frac{1}{\kappa}\Big] < \beta_1 .
  \end{equation} 
Choose $\delta_1$ such that $\frac{1}{\delta_1 \nu |C_L|} = \frac{1}{\kappa (1-2\epsilon -2\delta_1)}$, that is $\delta_1=\frac{\kappa (1-2\epsilon)}{2\kappa+|C_L| \nu} < \frac{1}{2}-\epsilon$.  
 Then,  if 
 \[\frac{ \bar{C}_4^4}{\kappa} \Big( \frac{1}{\nu} + \frac{1}{\kappa} + \frac{2\kappa}{\nu^2 |C_L|}\Big) <  4 \beta_1= 
  \frac{\nu \kappa\lambda_1 \tilde{\lambda}_1 - 8|C_L|}{\big[ \tilde{\lambda}_1 \kappa |C_L| {\rm Tr}(Q)  + \lambda_1 \nu {\rm Tr}(\QQ)] 
+ (\nu \kappa \lambda_1 \tilde{\lambda}_1 -8 |C_L|) \frac{\gamma_0 + |C_L| \tilde{\gamma}_0}{\gamma_0 \tilde{\gamma_0}}}, 
\]
  we can find $\epsilon\sim 0$ and $p\sim 1$  such that \eqref{constraint_p-1} holds true. 
   \end{proof}

We next. give conditions on ${\rm Tr}(Q)$ and ${\rm Tr}(\QQ)$ to have a strong rate of convergence of the schemes of order "almost" 1/2 when $C_L=0$. 
\begin{theorem}		\label{th-s-cv_CL=0}
Suppose that $C_L=0$. 

Let 
 $u_0\in V^1$ a.s. and $\theta_0\in H^1$ a.s. 
Let $u$ and $\theta$ be the solutions of \eqref{def_u} and \eqref{def_t} respectively, $\{ u^l\}_{l=0,..., N}$ and $\{\theta^l\}_{l=1, ..., N}$ be the semi-implicit
Euler scheme defined by \eqref{Euler_u} and \eqref{Euler_t}  respectively. 
Then if either assumption (i) or (ii) is satisfied,
we have 
\begin{equation}	\label{s-rate_CL0}
\EE\Big( \max_{1\leq l\leq N} \big( \|u(t_l)-u^l\|_{V^0}^2 + \|\theta(t_l)-\theta^l\|_{H^0}^2 \big)\Big) \leq C \Big(\frac{T}{N}\Big)^\eta,\qquad \eta \in (0,1).
\end{equation}

(i)  $u_0\in V^1$,  $\theta_0\in H^1$ are deterministic,  and 
  \[ \frac{\bar{C}_4^4 {\rm Tr}(\QQ)}{\kappa \nu \tilde{\lambda}_1} \Big( \frac{\nu +2 \kappa}{\kappa^2}+\frac{8T}{\nu^2 \lambda_1 } \Big)
 + \frac{\bar{C}_4^4 {\rm Tr}(Q)}{\nu^3 \lambda_1} <1.
 \] 
 
(ii)   $u_0$ and $\theta_0$ are ${\mathcal F}_0$-measurable and satisfy \eqref{exp_mom_ut0L2}.  
Suppose furthermore that $u_0 \in L^{p}(\Omega;V^1)$ and $\theta_0 \in L^p(\Omega;H^1)$ for any $p\in [2,\infty)$ (for example $u_0$ and $\theta_0$ are Gaussian random variables). 
Suppose furthermore that
 \[ \frac{\bar{C}_4^4}{4 \kappa  \tilde{\beta}_1} \Big[ \frac{2}{\nu} + \frac{1}{\kappa} \Big] + \frac{\bar{C}_4^4}{\nu^2 \tilde{\alpha}_1} <1,
   \] 
where the constants   $\tilde{\alpha}_1$ and $\tilde{\beta}_1$ are 
 defined  by 
  \eqref{def_alpha01} 
in  Proposition  \ref{exp_mom_schemeu}, and by \eqref{def_beta12} in Corollary \ref{cor_exp_mom_t}.

\end{theorem} 
\begin{proof}
The proof is that of the beginning of Theorem \ref{th-s-cv_CLneq0} until the necessity to prove that under appropriate conditions, $\EE\big[ \exp \big( ph \sum_{j=1}^{N-1} Y_j\big) \big] <\infty$
for some $p\in (1,\infty)$, which can be chosen close to 1. 

$\bullet$ {\bf Deterministic initial conditions $u_0$ and $\theta_0$} 
Let again 
$\tilde{\delta}_1\sim 1$, $\delta_1 \in (0, \frac{1}{2})$,  $\epsilon \in (0, 1-2\delta_1) $ and $\delta_5=1-2\delta_1-\epsilon$.
  we use Corollary \ref{cor_exp_mom_t} for the temperature and Proposition \ref{exp_mom_schemeu} for the velocity. Using H\"older's inequality, \eqref{exp-moments-Aun}
  and \eqref{exp-moments-At_al_0}  we have to find exponents $p_1$ and
  $p_2$ such that for some $\delta_1 \in (0, \frac{1}{2})$, 
  \begin{equation}	\label{p1p2-C=0}
  \frac{1}{p_1} + \frac{1}{p_2} = \frac{1}{p} <1, \quad  \frac{ p_1 \bar{C}_4^4}{4\nu^2  \delta_1} < \alpha_1  , \quad \mbox{\rm and } \quad 
  \frac{p_2 \bar{C}_4^4 }{4\kappa } \Big[ \frac{1}{\nu( 1-\epsilon - 2\delta_1)}+ \frac{1}{\kappa}\Big] < \tilde{\beta}_0 . 
  \end{equation} 
  Thus we have to make sure that
  \[ \frac{\bar{C}_4^4}{4} 
  \Big[ \frac{1}{\nu^2 \delta_1 \alpha_1} +  \frac{1}{\kappa \tilde{\beta}_0}   \Big( \frac{1}{\nu(1-\epsilon -2\delta_1)}  + \frac{1}{\kappa} \Big) \Big] <1.\]
  Choose $\delta_1$ such that   $\frac{1}{\nu \delta_1 \alpha_1} = \frac{1}{\kappa \tilde{\beta}_0  (1-\epsilon - 2\delta_1)}$, that is 
 $\delta_1=\frac{\kappa \tilde{\beta}_0 (1-\epsilon)}{2\kappa \tilde{\beta}_0 + \nu \alpha_1} <\frac{1}{2}$. 
   With  this choice of $\delta_1$, we deduce that the constraint reads
  \[ \frac{\bar{C}_4^4}{4 \tilde{\beta}_0  \kappa} \Big[ \frac{2}{\nu} + \frac{1}{\kappa} \Big] + \frac{\bar{C}_4^4}{\nu^2 \alpha_1} <1, \]
  that is,  plugging in the values of $\alpha_1$ and $\tilde{\beta}_0$,
  \[ \frac{\bar{C}_4^4 {\rm Tr}(\QQ)}{\kappa \nu \tilde{\lambda}_1} \Big( \frac{\nu +2 \kappa}{\kappa^2}+\frac{8T}{\nu^2 \lambda_1 } \Big)
 + \frac{\bar{C}_4^4 {\rm Tr}(Q)}{\nu^3 \lambda_1} <1.
 \] 

$\bullet$ {\bf Random initial conditions} 
 Let $u_0$ and $\theta_0$ be random and satisfy \eqref{exp_mom_ut0L2},  let $\tilde{\alpha}_1$ be defined by \eqref{def_alpha01} in Proposition~\ref{exp_mom_schemeu}, 
   and  $\tilde{\beta}_1$  be defined by \eqref{def_beta12} in Corollary \ref{cor_exp_mom_t}. 
  With the same choice of $\lambda\sim1$ and $\epsilon\sim0$,  we have to find exponents $p_1$ and
  $p_2$ such that for some $\delta_1 \in (0, \frac{1}{2})$,  
  \begin{equation}	\label{p1p2-C=0_random}
  \frac{1}{p_1} + \frac{1}{p_2} <1, \quad  \frac{ p_1 \bar{C}_4^4}{4\nu^2  \delta_1} < \tilde{\alpha}_1  , \quad \mbox{\rm and } \quad 
  \frac{p_2 \bar{C}_4^4 }{4\kappa } \Big[ \frac{1}{\nu( 1 - 2\delta_1)}+ \frac{1}{\kappa}\Big] <   \tilde{\beta}_1 . 
  \end{equation} 
   Thus we have to make sure that
   \[ \frac{\bar{C}_4^4}{4} 
  \Big[ \frac{1}{\nu^2 \delta_1 \tilde{\alpha}_1} + \frac{1}{\kappa   \tilde{\beta}_1}   \Big( \frac{1}{\nu(1 -2\delta_1)}  + \frac{1}{\kappa} \Big) \Big] <1.\]
   Choose $\delta_1$ such that $\frac{1}{\nu \delta_1 \tilde{\alpha}_1} = \frac{1}{\kappa  \tilde{\beta}_1 (1 - 2\delta_1)}$, that is 
 $\delta_1=\frac{\kappa  \tilde{\beta}_1 }{2\kappa  \tilde{\beta}_1 + \nu \tilde{\alpha}_1} <\frac{1}{2}$. 
   With  this choice of $\kappa$, we deduce that the constraint reads
 \[ \frac{\bar{C}_4^4}{4 \kappa  \tilde{\beta}_1} \Big[ \frac{2}{\nu} + \frac{1}{\kappa} \Big] + \frac{\bar{C}_4^4}{\nu^2 \tilde{\alpha}_1} <1.
   \] 
  This completes the proof.
\end{proof}
\section{Appendix}
In this section, we slightly extend some results proved in \cite{BeMi_Bou} to a more general B\'enard model (that is imposing a different temperature on the sets $\{x_2=0\}$ and $\{x_2 =L\}$ as in
the rest of the paper) with a multiplicative noise. 
This will justify Proposition~\ref{prop_mom_t_H1}  stated in  Section \ref{gwp}.

Since  $\langle [u(s). \nabla] \theta(s),\AA \theta(s)\rangle \neq 0$, 
 unlike what happens for the velocity, we keep the bilinear term. 
This creates technical problems and we proceed in two steps. First, using the mild formulation of the weak solution $\theta$ of \eqref{def_t},
 we prove that the gradient of the temperature has finite moments. Then going back to the weak form, we prove the desired result localized on a set
 where $\| A^{\frac{1}{2}} u(t)\|_{V^0}$ is bounded by a constant.
 
Let  $\{S(t)\}_{t\geq 0}$ be the semi-group generated by $-\nu A$,  $\{\SS(t)\}_{t\geq 0}$ be the semi-group generated by $-\kappa \AA$, that
is $S(t)=\exp(-\nu t A)$ and $\SS(t)=\exp(-\kappa t\AA)$ for every $t\geq 0$. Note that  for every $\alpha >0$
\begin{align}
\| A^\alpha S(t) \|_{{\mathcal L}(V^0;V^0)} \leq C t^{-\alpha}, \quad\forall  t>0 	\label{AS}
 \\
\| A^{-\alpha} \big[ {\rm Id} - S(t)\big] \|_{{\mathcal L}(V^0;V^0)}
  \leq C t^{\alpha}, \quad \forall t>0.	\label{A(I-S)}
\end{align} 
 Similar upper estimates are valid when we replace $A$ by $\AA$ and $S(t)$ by $\SS(t)$.
 
Note that  if $u_0\in L^2(\Omega;V^1)$ and $\theta_0\in L^2(\Omega;H^0)$,  $u\in L^{2}(\Omega ; C([0,T];V^0)\cap L^\infty( [0,T] ; V^1))$ 
and 
$\theta \in  L^{2}(\Omega ; C([0,T];H^0))\cap 
L^2(\Omega\times [0,T] ; H^1)$, we can write the solutions of \eqref{def_u}--\eqref{def_t} in the following mild form
\begin{align}
u(t) = &\, S(t) u_0 - \int_0^t S(t-s) B(u(s), u(s))\,  ds + \int_0^t S(t-s) \big(\Pi \theta(t) v_2\big) \,  ds  \nonumber \\
&+ \int_0^t S(t-s) G(u(s)) dW(s), 	\label{weak_u}\\
\theta(t) = &\, \tilde{S}(t)  \theta_0 - \int_0^t \SS(t-s) \big( [u(s) . \nabla] \theta(s)\big) \, ds - \int_0^t \SS(t-s) u_2(s) ds  \nonumber \\
&+ \int_0^t \SS(t-s) \GG(\theta(s)) d\WW(s), 	\label{weak_t}
\end{align}
where the first equality holds a.s. in $V^0$ and the second one in $H^0$.

Indeed, since $\|A^\alpha u\|_{V^0}\leq C  \|A^{\frac{1}{2}}  u\|_{V^0}^{2\alpha} \|u(s)\|_{V^0}^{1-2\alpha}$,
 the upper estimate \eqref{GiMi-uv} for $\delta +\rho >\frac{1}{2}$, $\delta+\alpha+\rho=1$ and the Minkowski inequality imply for every $\epsilon >0$
 \begin{align*}
\Big\| \int_0^t S(t-s)& B(u(s),u(s)) ds \Big\|_{V^0 }  \leq  \, \int_0^t \| A^{\delta-\epsilon}  A^{-\delta} B(u(s),u(s)) \|_{V^0} ds \\
& \leq C  \int_0^t (t-s)^{-\delta +\epsilon} \|A^\alpha u(s)\|_{V^0} \|A^\rho u(s)\|_{V^0} ds \\
& \leq C \,  \sup_{s\in [0,t]} \| u(s)\|_{V^1}^{2} \int_0^t (t-s)^{-\delta +\epsilon}  ds.
 \end{align*}
Since $\|S(t)\|_{{\mathcal L}(V^0;V^0)} \leq 1$, it is easy to see that 
 \[ \Big\| \int_0^t S(t-s) \Pi \theta(t) v_2 ds\Big\|_{V^0} \leq C \int_0^t \|\theta(t)\|_{H^0} ds .\]
 Furthermore,
 \[ \EE\Big( \Big\| \int_0^t S(t-s) G(u(s)) dW(s) \Big\|_{V^0}^2 \Big) \leq {\rm Tr}(Q) \EE\Big( \int_0^t [K_0+K_1 \|u(t)\|_{V^0}^2\big] ds \Big) <\infty.
 \]
 Therefore, the stochastic integral  $\int_0^t S(t-s) G(u(s)) dW(s) \in V^0$ a.s., and the identity \eqref{weak_u} is true a.s. in $V^0$.
 
  A similar argument shows that \eqref{weak_t} holds a.s. in $H^{0}$. We only show that the convolution involving the bilinear term belongs to $H^0$.
  Using the Minkowski inequality and the upper estimate \eqref{GiMi-ut} with positive constants $\delta, \alpha, \rho $ such that 
  $\alpha, \rho\in (0,\frac{1}{2})$,   $\delta + \rho > \frac{1}{2}$ and $\delta + \alpha + \rho =1$, we obtain
  \begin{align*}
  \Big\|& \int_0^t \SS(t-s) [ (u(s). \nabla) \theta(s)] ds \Big\|_{H^0} \leq \int_0^t \| \AA^\delta  \SS(t-s) \; \AA^{-\delta} [ (u(s). \nabla) \theta(s)] \|_{H^0}\,  ds  \\
  &\leq C  \int_0^t (t-s)^{-\delta} \, \|A^\alpha u(s)\|_{V^0} \, \| \AA^\rho \theta(s)\|_{H^0} d\, s \\
  & \leq C  \sup_{s\in [0,t]} \|u(s)\|_{V^1}\sup_{s\in [0,t]} \|\theta(s)\|_{H^0}^{1-2\rho}  \Big( \int_0^t (t-s)^{-\frac{\delta }{1-\rho}} ds\Big)^{1-\rho}
 \Big(  \int_0^t \|\AA^{\frac{1}{2}} \theta(s)\|_{H^0}^2 ds \Big)^\rho <\infty,
  \end{align*}
  where the last upper estimate is deduced from H\"older's inequality and $\frac{\delta}{1-\rho}<1$.

The following result shows that for fixed $t$, the $L^2$-norm of the gradient of $\theta(t)$ has finite moments.
\begin{lemma}		\label{lem_sup_E_t}
Let  $p\in [0,+\infty)$, $u_0\in L^{4p + \epsilon}(\Omega;V^1)$ and $\theta_0 \in L^{4p+\epsilon}(\Omega;H^1)$ for some  $\epsilon \in (0,\frac{1}{2})$. Let 
the diffusion coefficient $G$ and $\GG$ satisfy the condition {\bf(C)} and {\bf ($\tilde{\rm \bf C}$)} respectively. 
 For every N, let  $\tilde{\tau}_N := \inf\{ t\geq 0 : \|\AA^{\frac{1}{2}} \theta(t)\|_{H^0} \geq N\} \wedge T$; then 
\begin{equation}		\label{sup_E_t}
\sup_{N>0} 
\sup_{t\in [0,T]} \EE\big(\| \AA^{\frac{1}{2}} \theta(t \wedge \tilde{\tau}_N
)\|_{H^0}^{2p}\big)  <\infty.
\end{equation} 
\end{lemma}
\begin{proof}
Write $\theta(t)$ using \eqref{weak_t}; then $\|\AA^{\frac{1}{2}} \theta(t)\|_{H^0} \leq  \sum_{i=1}^4  T_i(t)$,  where
\begin{align*}
 T_1(t)&\, =\| \AA^{\frac{1}{2}} \SS(t) \theta_0\|_{H^0}, \quad  T_2(t)\, =\Big\| \int_0^t \AA^{\frac{1}{2}} \SS(t-s) [ (u(s).\nabla) \theta(s)] ds\Big\|_{H^0}, \\
 T_3(t)&  \,=\Big \| \int_0^t \AA^{\frac{1}{2}}  \SS(t-s) u_2(s) ds\Big\|_{H^0} ,
\quad T_4(t)\, = \Big\| \int_0^t \AA^{\frac{1}{2}} \SS(t-s) \GG(\theta(s)) d\WW(s)\Big\|_{H^0}. 
\end{align*}
The Minkowski inequality implies for $\beta \in (0,\frac{1}{2})$ 
\begin{align*}
 T_2(t)  \leq&\, \int_0^t \| \AA^{\frac{1}{2}} \SS(t-s) [(u(s) . \nabla) \theta(s)] \|_{H^0} ds\\
 & \leq  \int_0^t \|\AA^{1-\beta} \SS(t-s)\|_{{\mathcal L}(H^0;H^0)} \| \AA^{-(\frac{1}{2}-\beta)} [ (u(s).\nabla)\theta(s)\|_{H^0} ds. 
\end{align*}
Apply \eqref{GiMi-ut} with $\delta = \frac{1}{2}-\beta$, $\alpha = \frac{1}{2}$ and $\rho \in (\beta, \frac{1}{2})$. A simple computation
proves that $\|\AA^\rho f\|_{H^0} \leq \|\AA^{\frac{1}{2}} f\|_{H^0}^{2\rho} \|f\|_{H^0}^{1-2\rho}$ for any $f\in H^1$.  Therefore, 
\begin{align*} 
 \| \AA^{-(\frac{1}{2}-\beta)} [(u(s).\nabla)\theta(s)]\|_{H^0} &\leq C \|A^{\frac{1}{2}} u(s)\|_{V^0} \|\AA^\rho \theta(s)\|_{H^0}\\
 & \leq  C \|A^{\frac{1}{2}} u(s)\|_{V^0} 
\| \AA^{\frac{1}{2}} \theta(s) \|_{H^0}^{2\rho} \|\theta(s)\|_{H^0}^{1-2\rho}.
\end{align*}   
This upper estimate and \eqref{AS} imply
\[ T_2(t) \leq C \sup_{s\in [0,T]} \|A^{\frac{1}{2}} u(s)\|_{H^0}  \sup_{s\in [0,t]} \|\theta(s)\|_{H^0}^{1-2\rho} 
\int_0^t (t-s)^{-1+\beta} \|\AA^{\frac{1}{2}} 
\theta(s)\|_{H^0}^{2\rho} ds.\]
For any $p\in [1,\infty)$, using H\"older's inequality  with respect to the finite measure \linebreak  $(t-s)^{-(1-\beta)} 1_{[0,t)}(s) ds$  we obtain 
\begin{align*}
 T_2(t)^{2p} &\, \leq C \sup_{s\in [0,t]} \|A^{\frac{1}{2}} u(s)\|_{V^0}^{2p} \sup_{s\in [0,t]} \|\theta(s)\|_{H^0}^{2p(1-2\rho)} 
 \Big( \int_0^t (t-s)^{-(1-\beta)} ds\Big)^{2p-1}\\
 &\quad \times 
\Big( \int_0^t (t-s)^{-(1-\beta)} \|\AA^{\frac{1}{2}} \theta(s)\|_{H^0}^{4p\rho} ds \Big).
\end{align*}
Let $p_1=\frac{2(1-\rho)}{1-2\rho}$, $p_2=\frac{2(1-\rho)}{(1-2\rho)^2}$ and $p_3=\frac{1}{2\rho}$. Then $\frac{1}{p_1}+\frac{1}{p_2} + \frac{1}{p_3}=1$,
$4\rho p p_3=2p$ and $p p_1 = p(1-2\rho) p_2 := \tilde{p}$. Young's and H\"older's inequalities imply
\begin{align*}
 T_2(t)^{2p}   \leq &\, C \Big[ \frac{1}{p_1} \sup_{s\in [0,t]} \|A^{\frac{1}{2}} u(s)\|_{V^0}^{2\tilde{p}} 
 + \frac{1}{p_2} \sup_{s\in [0,t]} \|\theta(s)\|_{H^0}^{2\tilde{p}} \\
 &\quad 
+ \frac{1}{p_3} \Big( \int_0^t (t-s)^{-1+\beta} \|\AA^{\frac{1}{2}} \theta(s)\|_{H^0}^{2p} ds\Big) \Big( \int_0^t (t-s)^{-1+\beta} ds\Big)^{p_3-1} \Big] .
\end{align*}
Note that the continuous function $\rho\in (0,\frac{1}{2}) \mapsto  \frac{2(1-\rho)}{1-2\rho}$ is increasing with  $\lim_{\rho\to 0}  \frac{2(1-\rho)}{1-2\rho} =2$. 
Given $\epsilon>0$ choose 
$\rho \in (0,\frac{1}{2})$ close enough to 0 to have $2\tilde{p}=2p \frac{2(1-\rho)}{1-2\rho} = 4 p+\epsilon$, then choose $\beta \in (0,\rho)$. The above
computations yield
\begin{align} 	\label{T_2-convol}
T_2(t)^{2p} \leq &\; C\Big[ \sup_{s\in [0,t]} \|A^{\frac{1}{2}} u(s)\|_{V^0}^{4p+\epsilon} + \sup_{s\in [0,t]} \|\theta(s)\|_{H^0}^{4p+\epsilon} \Big] \nonumber \\
&\; + C \int_0^t (t-s)^{-1+\beta} \|\AA^{\frac{1}{2}} \theta(s)\|_{H^0}^{2p} ds.
\end{align}
The Minkowski inequality and \eqref{AS} for the pair $\AA$ and $\SS$ imply 
\begin{align*}	
T_3(t) \leq &\, \int_0^t \| \AA^{\frac{1}{2}} \SS(t-s) u_2(s)\|_{H^0} ds \leq C \sup_{s\in [0,T]} \|A^{\frac{1}{2}} u(s)\|_{V^0} \int_0^t (t-s)^{-\frac{1}{2}} ds. 
\end{align*} 
Thus for any $p\in [1,\infty)$
\begin{equation} \label{T3_convol}
T_3(t)^{2p} \leq C(T,p) \|A^{\frac{1}{2}} u(s)\|_{V^0}^{2p}. 
\end{equation}

Finally, Burhholder's inequality, the growth condition \eqref{growthGG-1}  and H\"older's inequality imply that  for $t\in [0,T]$ 
\begin{align}	\label{mom-convol}
\EE \Big( \Big\| \int_0^t &\AA^{\frac{1}{2}} \SS(t-s) \GG(\theta(s)) d\WW(s) \Big\|_{H^0}^{2p} \Big) 
\leq \, C_p \big( {\rm Tr}(Q)\big)^p \EE\Big( \Big|
\int_0^t \|\AA^{\frac{1}{2}} \GG(\theta(s))\|_{{\mathcal L}(\KK;H^0)}^2 ds \Big|^p\Big)  \nonumber \\
\leq &\, C_p \big( {\rm Tr}(Q)\big)^p\, \EE\Big( \Big| \int_0^t [\tilde{K}_2 + \tilde{K}_3 \|\theta(s)\|_{H^0}^2 
 + \tilde{K}_3 \| \AA^{\frac{1}{2}} \theta(s)\|_{H^0}^2] ds \Big|^p \Big)
\nonumber \\
\leq &\, C(p, \tilde{K}_2, \tilde{K}_3, {\rm Tr}(Q))  T^p \Big[ 1+ \EE\Big( \sup_{s\in [0,t]} \|\theta(s)\|_{H^0}^{2p} \Big]  \nonumber \\
& \quad + C_p \big( {\rm Tr}(Q)\big)^p \tilde{K_3}^p T^{p-1} 
\int_0^t \EE\big( \|\AA^{\frac{1}{2}} \theta(s)\|_{H^0}^{2p}  \big) ds.
\end{align} 
The upper estimates 
\eqref{T_2-convol}--\eqref{mom-convol} and $T_1(t) \leq \|A^{\frac{1}{2}} \theta_0\|_{H^0} \leq \|\theta_0\|_{H^1}$ used with $t\wedge \tilde{\tau}_N$
instead of $t$ imply for every $t\in [0,T]$ 
\begin{align*}
 \EE\big( \|\AA^{\frac{1}{2}} \theta(t\wedge \tilde{\tau}_N)\|_{H^0}^{2p} \big) \leq&\,  C_p
\Big[1+  \EE\Big(\| \AA^{\frac{1}{2}} \theta_0\|_{H^0}^{2p} +  \sup_{s\in [0,T]} \|A^{\frac{1}{2}} u(s)\|_{V^0}^{4p+\epsilon} +
\sup_{s\in [0,T]} \|\theta(s)\|_{H^0}^{4p+\epsilon} \Big)\Big]\\
 &\, + C_p \int_0^t  \big[ (t-s)^{-1+\beta}  + \tilde{K}_3 T^{p-1} \big] 
 \EE\big( \| \AA^{\frac{1}{2}} \theta(s\wedge \tilde{\tau}_N)\|_{H^0}^{2p} \big) ds, 
\end{align*}
where the constant $C_p$ does not depend on $t$ and $N$. 
Theorem \ref{th-gwp}, Proposition \ref{prop_u_V1}, and the version of Gronwall's lemma proved in the following lemma \ref{Gronwall} imply
\eqref{sup_E_t} 
for some constant $C$ depending on $\EE(\|u_0\|_{V^1}^{4p+\epsilon})$ and $\EE(\|\theta_0\|_{H^0}^{4p+\epsilon})$.
 The proof of the Lemma 
is complete.
\end{proof}

The following lemma is an extension of Lemma 3.3, p.~316  in \cite{Walsh}.
 Its proof is given in \cite[Lemma~3]{BeMi_Bou}. 
\begin{lemma}		\label{Gronwall}
Let $\epsilon \in (0,1)$, $a,b,c$ be positive constants and $\varphi$ be a bounded non negative function such that 
\begin{equation}  	\label{maj_Gron}
\varphi(t) \leq a+\int_0^t \big[ b+c(t-s)^{-1+\epsilon}\big] \,\varphi(s)\, ds, \quad \forall t\in [0,T].
\end{equation}
Then $\sup_{t\in [0,T]} \varphi(t)\leq C(a,b,c,T,\epsilon)<\infty$.
\end{lemma}

\bigskip

\begin{proof}[Proof of Proposition \ref{prop_mom_t_H1}]
We next prove that 
the gradient of the temperature
 has bounded moments uniformly in time. 

Applying the operator $\AA^{\frac{1}{2}}$ to Equation \eqref{def_t}, and writing It\^o's formula for the square of the corresponding $H^0$-norm,
we obtain 
\begin{align*}		
\| \AA^{\frac{1}{2}} \theta(t) &\|_{H^0}^2 +2\kappa \int_0^t \|\AA \theta(s)\|_{H^0}^2 ds = \| \AA^{\frac{1}{2}} \theta_0\|_{H^0}^2 
- 2 \int_0^t \langle (u(s).\nabla) \theta(s), \AA\theta(s) \rangle ds  \\
& + 2 \int_0^t \big( \AA^{\frac{1}{2}} \GG(\theta(s)) d\WW(s) , \AA^{\frac{1}{2}} \theta(s)\big)
+ {\rm Tr}(Q) \int_0^t \| \AA^{\frac{1}{2}} \GG(\theta(s))\|_{H^0}^2 ds.
\end{align*}
Then, apply It\^o's formula for the map $x\mapsto x^p$. This yields, using integration by parts, 
\begin{align} \label{Ito-t-H1}
\| \AA^{\frac{1}{2}}& \theta(t) \|_{H^0}^{2p} + 2p\kappa \int_0^t \|\AA \theta(s)\|_{H^0}^2 \|\AA^{\frac{1}{2}} \theta(s)\|_{H^0}^{2(p-1)} ds
= \| \AA^{\frac{1}{2}} \theta_0\|_{H^0}^{2p}  \nonumber \\
&- 2p \int_0^t \langle (u(s).\nabla) \theta(s), \AA\theta(s) \rangle \| \AA^{\frac{1}{2}} \theta(s)\|_{H^0}^{2(p-1)} ds  \nonumber \\
& + 2p \int_0^t \big( \AA^{\frac{1}{2}} \GG(\theta(s)) d\WW(s) , \AA^{\frac{1}{2}} \theta(s)\big)  \| \AA^{\frac{1}{2}} \theta(s)\|_{H^0}^{2(p-1)} \nonumber \\
&+ p {\rm Tr}(\QQ) \int_0^t \| \AA^{\frac{1}{2}} \GG(\theta(s))\|_{H^0}^2  \| \AA^{\frac{1}{2}} \theta(s)\|_{H^0}^{2(p-1)} ds\nonumber \\
&+ 2p(p-1) {\rm Tr}(\QQ)\int_0^{t} \| \big(\AA^{\frac{1}{2}} \GG\big)^*(\theta(s)) \big(\AA^{\frac{1}{2}} \theta(s)\big)\|_K^2 
\| \AA^{\frac{1}{2}} \theta(s)\|_{H^0}^{2(p-2)} ds.
\end{align}
The Gagliardo--Nirenberg inequality \eqref{GagNir} 
and the inclusion $V^1\subset \LL^4$ implies that 
\begin{align*}
  \int_0^t \big|  \langle (u(s).\nabla) &\theta(s) , \tilde{A} \theta(s)\rangle \big| \,  \|\AA^{\frac{1}{2}} \theta(s)\|_{H^0}^{2(p-1)}  ds  \\
  & \leq  C  \int_0^t \|\tilde{A} \theta(s)\|_{H^0}  \|u(s)\|_{\LL^4} \|\AA^{\frac{1}{2}} \theta(s)\|_{L^4}  \|\AA^{\frac{1}{2}} \theta(s)\|_{H^0}^{2(p-1)}  ds  \\
 &\leq  C  \int_0^t \|\tilde{A} \theta(s)\|_{H^0}^{\frac{3}{2}}   \|u(s)\|_{V^1}  \|\AA^{\frac{1}{2}} \theta(s)\|_{H^0}^{2p-\frac{3}{2}}  ds .
\end{align*}
Then, using the H\"older and Young's inequalities, we deduce
\begin{align}		\label{bilin-ut}
2p \int_0^{t} 
  &\big| \langle (u(s).\nabla) \theta(s), \AA\theta(s) \rangle \big| 
\| \AA^{\frac{1}{2}} \theta(s)\|_{H^0}^{2(p-1)} ds  	\nonumber \\
\leq & \; (2p-1)\, \kappa  \int_0^{t} 
 \|\AA(\theta(s))\|_{H^0}^{2} \| \AA^{\frac{1}{2}} \theta(s)\|_{H^0}^{2(p-1)} ds  \nonumber \\
 &\; +   C(\kappa,p) \sup_{s\in [0,T]}  \|u(s)\|_{V^1}^4 
 \int_0^t 
 \| \AA^{\frac{1}{2}} \theta(s)\|_{H^0}^{2p} ds .
\end{align}
The growth condition \eqref{growthGG-1} and H\"older's and Young inequalities imply that
\begin{equation}		\label{maj_trace1}
 \int_0^{t} 
  \|\AA^{\frac{1}{2}} \GG(\theta(s)) \|_{H^0}^2 \|\AA^{\frac{1}{2}} \theta(s)\|_{H^0}^{2(p-1)} ds 
\leq C \int_0^{t} 
\big[ 1+ \|\theta(s)\|_{H^0}^{2p} +\|\AA^{\frac{1}{2}} \theta(s)\|_{H^0}^{2p} \big] ds,
\end{equation}
and a similar computation yields 
\begin{align}		\label{maj_trace2}
 \int_0^{t} 
  \big\| \big(\AA^{\frac{1}{2}} &\, \GG\big)^*(\theta(s)) \big(\AA^{\frac{1}{2}} \theta(s)\big)\|_K^2 
\| \AA^{\frac{1}{2}} \theta(s)\|_{H^0}^{2(p-2)} ds \nonumber \\
& \leq 
C \int_0^{t} 
\big[ 1+ \|\theta(s)\|_{H^0}^{2p} +\|\AA^{\frac{1}{2}} \theta(s)\|_{H^0}^{2p} \big] ds.
\end{align}
 
 Let $\tilde{\tau}_N:=\inf\{ t\geq 0: \|\AA^{\frac{1}{2}} \theta(t)\|_{H^0} \geq N\}$. The upper estimates \eqref{Ito-t-H1}--\eqref{maj_trace2}
  written for $t\wedge \tilde{\tau}_N$ instead of $t$ imply  
  
\begin{align*}
\sup_{t\in [0,T]} &\|\AA^{\frac{1}{2}} \theta(t\wedge \tilde{\tau}_N) 
 \|_{H^0}^{2p}  + \kappa \int_0^{T\wedge \tilde{\tau}_N} \| \AA \theta(s)\|_{H^0}^2 \| \AA^{\frac{1}{2}} \theta(s)\|_{H^0}^{2(p-1)} ds 
  \leq \|\AA^{\frac{1}{2}} \theta_0\|_{H^0}^{2p}\\
&  
+ C \sup_{s\in [0,T]} \! \|u(s)\|_{V^1}^4 \!\! \int_0^{T\wedge \tilde{\tau_N}} \! \| \AA^{\frac{1}{2}} \theta(s)\|_{H_0}^{2p} ds  
+C \int_0^{T\wedge \tilde{\tau}_N} \!\! \big( 1+\|\theta(s)\|_{H^0}^{2p} + \| \AA^{\frac{1}{2}} \theta(s)\|_{H^0}^{2p}  \big)ds \\
&+ 2p \,
\sup_{t\in [0,T]} \int_0^{t\wedge \tilde{\tau}_N} 
 \big(
 \AA^{\frac{1}{2}} \GG(\theta(s)) d\WW(s), 
\AA^{\frac{1}{2}} \theta(s) \big) 
\| \AA^{\frac{1}{2}} \theta(s)\|_{H^0}^{2(p-1)}.
\end{align*}
Using the Cauchy--Schwarz inequality, Fubini's theorem, \eqref{mom_u_V1}  and \eqref{sup_E_t}, we deduce 
\begin{align}		\label{maj_u_int-theta}
\EE\Big( \sup_{s\in [0,T]}& \| u(s)\|_{V^1}^4  \int_0^{T\wedge\tilde{\tau}_N}  \| \AA^{\frac{1}{2}} \theta(s)\|_{H^0}^{2p} ds \Big) \nonumber \\
&  \leq 
\Big\{ \EE\Big( \sup_{s\in [0,T]} \| u(s)\|_{V^1}^8  \Big) \Big\}^{\frac{1}{2}} 
\Big\{ \int_0^T \EE\big( \|\AA^{\frac{1}{2}} \theta(s\wedge \tilde{\tau}_N)\|_{H^0}^{4p} \big) ds \Big\}^{\frac{1}{2}} \leq C.
\end{align} 

The Davis inequality, the growth condition \eqref{growthGG-1} and the Cauchy--Schwarz, Young and H\"older inequalities imply that
%
\begin{align*}
\EE \Big(& \sup_{t\in [0,T]} 
\Big|  \int_0^{t\wedge \tilde{\tau}_N} \big( 
\AA^{\frac{1}{2}} \GG(\theta(s)) d\WW(s), 
\AA^{\frac{1}{2}} \theta(s) \big)  \| \AA^{\frac{1}{2}} \theta(s)\|_{H^0}^{2(p-1)} \Big) 
\\
& \leq C \, \EE\Big( \Big\{ \int_0^{T} 
 {\rm Tr}(\QQ)
 \big[\tilde{K}_2  + \tilde{K}_3 \|\theta(s\wedge \tilde{\tau}_N)\|_{H^1}^2   \big] 
\|\AA^{\frac{1}{2}} \theta(s\wedge \tilde{\tau}_N) \|_{H^0}^{4p-2} ds \Big\}^{\frac{1}{2}} \Big)  \\
&\leq  C \, \EE\Big[ \big( \sup_{s\leq T} \big( 
 \|\AA^{\frac{1}{2}} \theta(s\wedge \tilde{\tau}_N)
\|_{H^0}^p\big)   ({\rm Tr}(\QQ))^{\frac{1}{2}}  \\
&\quad \times 
 \Big\{\int_0^T \big[ \tilde{K}_2  + \tilde{K}_3 \|\theta(s\wedge \tilde{\tau}_N)\|_{H^0}^2  + \tilde{K}_3
  \|\AA^{\frac{1}{2}} \theta(s\wedge \tilde{\tau}_N)\|_{H^0}^2 \big] 
\| \AA^{\frac{1}{2}} \theta(s\wedge \tilde{\tau}_N) 
\|_{H^0}^{2(p-1)} ds \Big\}^{\frac{1}{2}} \Big)\\
&\leq \frac{1}{4p} \EE\Big( \sup_{s\leq T} \big( 
\|\AA^{\frac{1}{2}}   \theta(s\wedge\tilde{\tau}_N) 
\|_{H^0}^{2p}\Big) \\
&\qquad + C \EE\Big( \int_0^T \big[1+\| \theta(s\wedge\tilde{\tau}_N)\|_{H^0}^{2p} 
+ \| \AA^{\frac{1}{2}} \theta(s\wedge \tilde{\tau}_N)\|_{H^0}^{2p}  \big] ds\Big). 
\end{align*}
Therefore, the upper estimates \eqref{mom_ut_L2}, \eqref{sup_E_t} and \eqref{maj_u_int-theta} imply that
\[ 
\frac{1}{2}  \EE\Big( \sup_{s\leq T}  \|\AA^{\frac{1}{2}} \theta(s\wedge\tilde{\tau}_N)\|_{H^0}^{2p} \Big) 
+  \kappa\; \EE\Big(  \int_0^{T\wedge \tilde{\tau}_N} \| \AA \theta(s\wedge \tilde{\tau}_N)\|_{H^0}^2 \| \AA^{\frac{1}{2}} \theta(s)\|_{H^0}^{2(p-1)} ds \Big)
 \leq C
\] 
for some constant $C$ independent of $N$. 
 As $N\to +\infty$, we deduce \eqref{mom_t_H1}; 
this completes the proof of Proposition \ref{prop_mom_t_H1}. 
\end{proof}

 \noindent {\bf Acknowledgements.}   
 Annie~Millet's research has been conducted within the FP2M federation (CNRS FR 2036). 

\noindent{\bf Declarations.}

\noindent{\bf Funding.}  Hakima Bessaih was partially supported by  NSF grant DMS: 2147189.

\noindent {\bf  Conflicts of interest.} The authors have no conflicts of interest to declare that are
 relevant to the content of this article.

\noindent{\bf  Availability of data and material.} Data sharing not applicable to this article as
no datasets were generated or analysed during the current study.

\end{document}